\documentclass[a4paper,reqno]{amsart}
\usepackage[margin=3.7cm]{geometry}

\usepackage[T1]{fontenc}

\usepackage{microtype}

\usepackage{mathtools}
\usepackage{amssymb}
\usepackage{amsthm}

\numberwithin{equation}{section}

\usepackage{slashed}

\usepackage{tikz}
\usetikzlibrary{cd}
\usetikzlibrary{arrows,matrix,backgrounds,decorations.markings}
\usetikzlibrary{arrows.meta}
\usetikzlibrary{graphs}
\usetikzlibrary{quotes}
\usetikzlibrary{decorations.markings}
\tikzset{middlearrow/.style={
	decoration={
		markings,
		mark= at position 0.5 with {\arrow{#1}} ,
	},
	postaction={decorate}
}}

\usepackage[
      pdfusetitle,
      hidelinks
    ]{hyperref}

\DeclareFontFamily{U}{BOONDOX-cal}{}
\DeclareFontShape{U}{BOONDOX-cal}{m}{n}{ <-> s*[1.0] BOONDOX-r-cal}{}
\DeclareMathAlphabet{\mathscr}{U}{BOONDOX-cal}{m}{n}

\DeclareFontFamily{OT1}{pzc}{}
\DeclareFontShape{OT1}{pzc}{m}{it}{<-> s * [1.2] pzcmi7t}{}
\DeclareMathAlphabet{\mathpzc}{OT1}{pzc}{m}{it}

\newcommand\Aut{\operatorname{Aut}}
\newcommand\End{\operatorname{End}}
\newcommand\Hom{\operatorname{Hom}}
\newcommand\Dom{\operatorname{Dom}}
\newcommand\supp{\operatorname{supp}}
\renewcommand\ker{\operatorname{Ker}}
\newcommand\im{\operatorname{Im}}
\newcommand\rd{\mathrm{d}} %{\operatorname{d}\!}

\newcommand\Ad{\operatorname{Ad}}
\newcommand\ad{\operatorname{ad}}

\renewcommand\dim{\operatorname{dim}}
\renewcommand\sup{\operatorname{sup}}
\renewcommand\max{\operatorname{max}}
\renewcommand\min{\operatorname{min}}

\renewcommand\Re{\operatorname{Re}}

\newcommand\C{\mathbb{C}}

\newcommand\R{\mathbb{R}}
\newcommand\N{\mathbb{N}}
\newcommand\Z{\mathbb{Z}}
\newcommand\T{\mathbb{T}}

\newcommand\Q{\mathbb{Q}}

\newcommand\II{\mathcal{I}}
\newcommand\JJ{\mathcal{J}}

\newcommand\LL{\mathcal{L}}

\renewcommand\SS{\mathcal{S}}

\DeclareMathOperator{\sign}{sign}

\newtheorem{thm*}{Theorem}
\newtheorem{thm}{Theorem}[section]

\newtheorem{cor}[thm]{Corollary}

\newtheorem{lemma}[thm]{Lemma}
\newtheorem{prop}[thm]{Proposition}

\theoremstyle{definition}
\newtheorem{definition}[thm]{Definition}

\theoremstyle{remark}
\newtheorem{remark}[thm]{Remark}
\newtheorem{example}[thm]{Example}

\newcommand{\sfH}{\mathsf{H}}
\DeclareMathOperator{\Ran}{Ran} % range
\newcommand{\Cl}{\mathscr{C\mspace{-2mu}l\mspace{-2mu}}} % Clifford algebra
\newcommand*{\tildeotimes}{\mathbin{\tilde{\otimes}}}

\newcommand{\VEC}[1]{{\boldsymbol{#1}}}
\newcommand\DC{\VEC{D}}

\title{Parabolic noncommutative geometry}
\author{Magnus Fries, Magnus Goffeng, Ada Masters}
\address{Magnus Fries, Magnus Goffeng, Ada Masters\newline
\indent Centre for Mathematical Sciences\newline 
\indent Lund University\newline 
\indent Box 118, SE-221 00 Lund\newline 
\indent Sweden\newline
\newline}

\email{magnus.fries@math.lth.se, magnus.goffeng@math.lth.se, ada.masters@math.lth.se}

\date{\today}

\begin{document}

\begin{abstract}
We introduce to spectral noncommutative geometry the notion of tangled spectral triple, which encompasses the anisotropies arising in parabolic geometry as well as the parabolic commutator bounds arising in so-called ``bad Kasparov products''. Tangled spectral triples incorporate anisotropy by replacing the unbounded operator in a spectral triple that mimics a Dirac operator with several unbounded operators mimicking directional Dirac operators. We allow for varying and dependent orders in different directions, controlled by using the tools of tropical combinatorics. We study the conformal equivariance of tangled spectral triples as well as how they fit into $K$-homology by means of producing higher order spectral triples. Our main examples are hypoelliptic spectral triples constructed from Rockland complexes on parabolic geometries; we also build spectral triples on nilpotent group $C^*$-algebras from the dual Dirac element and crossed product spectral triples for parabolic dynamical systems.
\end{abstract}

\maketitle

{\footnotesize
\setcounter{tocdepth}{1}
\tableofcontents}

\section{Introduction}

In Connes' program for spectral noncommutative geometry \cite{ConnesNCDG,ConnesBigRed} one encodes geometry by means of spectral triples, or $K$-cycles. The technology surrounding spectral triples \cite{CaGaReSu,Connes_1995,LoReVa,waltbook} allows us to view spectral noncommutative geometry as a vast extension of classical Riemannian geometry to more exotic geometric situations \cite{connesmetric,conneslott,connesgravity,connesreconstr,GMCK,GMRtwist,GRU,leschkaad1,Hoch3,lottlimit,neshtuset,renniereconst,rieffelquant} that give meaning to the term \emph{noncommutative} geometry. The problem we address in this work is how to do noncommutative geometry in situations where there are different directions with drastically different types of behaviour. We restrict ourselves to \emph{parabolic} situations, by which we mean that the different directions, in an appropriate way, come with mutual polynomial bounds. The study of this problem is geometrically motivated by a number of examples where the difference between the directions manifests in various ways, some of extrinsic geometric interest and some of intrinsic interest to noncommutative geometry. \medskip

We begin by describing two of our motivating examples. The first motivating example originates in parabolic geometry \cite{capslovak} where the tangent bundle is filtered and the different tangent directions capture different geometric features. One encodes the geometry through the structure of a graded nilpotent Lie group on each tangent space. Analytically one can study a parabolic geometry through a BGG-complex \cite{capslovaksoucek,davehaller} that replaces the de Rham complex. While the de Rham complex and associated Dirac operators are well understood, and even form prototypical examples in noncommutative geometry, BGG-complexes are still not well understood analytically. The study of BGG-complexes is motivated by recent work \cite{goffhelff} implying the known natural candidates for general classes of Heisenberg elliptic differential operators with interesting spectral noncommutative geometry have trivial index theory. 

The analytic foundations for BGG-complexes were developed by Dave--Haller \cite{davehaller,davehallerheat} building on ideas of Rumin \cite{rumincomp} on contact manifolds. At the level of noncommutative topology, i.e.~index theory, BGG-complexes were studied by the second listed author \cite{goffeng24}. Understanding the spectral noncommutative geometry of parabolic geometries is of interest in order to organise efficiently the differential geometric machinery into a global theory well adapted for studying global invariants. A problem motivating such a machinery is that of finding non-trivial global invariants of parabolic geometries. In fact, already for CR-manifolds this problem is non-trivial; see the prominent work of Fefferman \cite{fefferpara}. A global invariant was studied by Hirachi \cite{hirachiinv} and Ponge \cite{pongeconfinv,pongemore} that was later proven by Boutet de Monvel \cite{boutededkn} to vanish. Ponge introduced new invariants in \cite{pongemore} that were studied further in \cite{pongefurther}. For more general parabolic geometries, Haller \cite{hallergenfive} has studied analytic torsion building on the work of Rumin--Seshadri \cite{ruminsesh} for contact manifolds. \medskip

The second motivating example is a fundamental object for spectral noncommutative geometry: the unbounded Kasparov product. The product implements the Kasparov product on $KK$-groups, that provide the foundation for the utility of Kasparov's $KK$-groups, at the level of unbounded cycles. The unbounded Kasparov product was studied by Kucerovsky \cite{danthedan} and later phrased constructively by Mesland \cite{meslandbeast}. In somewhat technical terms, the unbounded Kasparov product of an unbounded $A$-$B$-cycle $(\mathpzc{E}_1,S)$ with a $B$-$C$-cycle $(\mathpzc{E}_2,T)$ along a connection $\nabla$ is the pair $(\mathpzc{E}_1\otimes_B\mathpzc{E}_2,S\otimes 1+1\otimes_\nabla T)$ that under favorable circumstances form an unbounded $A$-$C$-cycle. There are functional analytic issues with $S\otimes 1+1\otimes_\nabla T$ forming a self-adjoint operator, which additionally needs to be regular in the Hilbert $C^*$-module sense. Such questions were addressed in \cite{meslandbeast} under some technical restrictions which have since matured in the important work of Kaad--Lesch \cite{leschkaad1,kaadlesch} and Lesch--Mesland \cite{bramlesch}. 

An issue that is more delicate and has evaded a proper axiomatization in unbounded $KK$-theory concerns the condition of bounded commutators in the unbounded Kasparov product. There are natural examples arising from dynamics \cite{GMCK,GMRtwist} where $1\otimes_\nabla T$ does not have bounded commutators with a dense subspace of $A$. Rather $1\otimes_\nabla T$ ends up being of ``higher order'' in contrast to $S\otimes 1$ in the sense that commutators with $1\otimes_\nabla T$ are relatively bounded by $(1+S^2)^{-1/2+1/2m}$ for an $m\geq 1$ playing the role of an order. This phenomenon occurs for Kasparov products arising from parabolic dynamics. The ad hoc solution is to inflate the spectrum of $S \otimes 1$ or dampen the spectrum of $1\otimes_\nabla T$ to compensate. The aim of this paper is to widen our view on spectral noncommutative geometry to allow for varying orders of operators and potential anisotropies to persevere as a feature rather than a bug. \medskip

A related issue—which, although it motivated our work, we have not addressed here—stems from the early years of noncommutative geometry, when there was optimism that quantum groups would be particularly well suited for noncommutative geometry \cite{connessuq2,Hoch3,neshtuset}. While much progress has been made in low dimension, little is known in higher dimension despite algebraic versions of BGG-complexes \cite{heckkolb} that have been studied in a noncommutative geometry context by Wagner--Díaz-García--O'Buachalla \cite{wagnerobua} and Voigt--Yuncken \cite{voigtyuncken}. A fundamental problem lies precisely in the complications found in the algebraic relations between the various ``directions'' in a quantum group, a statement made precise by the work of Krähmer--Rennie--Senior \cite{Hoch3}. In fact, the problems arising in Krähmer--Rennie--Senior's work relate to the Kasparov product, as discussed above. This direction of applications for our methods is speculative since the above alluded to parabolic behaviour does not capture the wild, hyperbolic features seen for quantum groups. We mention the connection nevertheless since our main definition drew inspiration from the noncommutative geometry of quantum groups in work of Kaad--Kyed \cite{kaadkyed}, and as a source for future investigations. \medskip

The building blocks of spectral noncommutative geometry are spectral triples. The notion of a spectral triple $(\mathcal{A},\mathpzc{H},D)$ is reviewed in the general context of a higher order spectral triple (HOST) in Definition \ref{defhost}. In the prototypical examples arising from commutative situations, $\mathcal{A}$ consists of an algebra of smooth functions on a manifold and $D$ is a self-adjoint operator on a Hilbert space $\mathpzc{H}$, satisfying axioms making it similar to a differential operator acting on $\mathpzc{H}$, the $L^2$-space of a vector bundle, with some mild ellipticity-like conditions.

We extend this notion to that of a \emph{strictly tangled spectral triple} (ST\textsuperscript{2}) in Definition \ref{def:st2} where $D$ is now allowed to be a \emph{finite collection} $\DC=(D_j)_{j\in I}$ of self-adjoint operators which satisfies an analogue of a mild ellipticity condition and an anticommutation relation. The adjective \emph{strictly} is to indicate that we assume the elements in the collection to anticommute on the nose. We expect our results to hold under more general assumptions, e.g.~when the anticommutators are relatively small (see Remarks \ref{remark:possible_extensions_of_st2} and \ref{relboundgen}), but to reduce the technical burden in the paper we focus on the simpler case that on its own already holds enough interesting examples. As mentioned above, a similar idea has appeared in the work of Kaad--Kyed \cite{Kaad_2020, kaadkyed} and of Kaad--Nest--Wolfsson \cite{kaadnestwolf}. The first and second of these works respectively describe the metric geometry of crossed products by \( \Z \) and of $SU_q(2)$ by means of keeping (twisted) derivations separated according to different directions, and the third of these works studies cohomological invariants on double loop groups in terms of directional quantum derivatives. Our main results are the following. 

\begin{thm*}
\label{mainthm1}
Let $(\mathcal{A},\mathpzc{H},\DC)$ be an ST\textsuperscript{2} with $\DC=(D_j)_{j\in I}$ the finite collection of self-adjoint operators and  bounding matrix $\VEC{\epsilon} \in M_I([0,\infty))$. Consider the non-empty set 
$$\Omega(\VEC{\epsilon}):=\{\VEC{t}=(t_j)\in (0,\infty)^n: \epsilon_{ij}t_i<t_j\; \forall i,j\}.$$ 
For $\VEC{t}\in \Omega(\VEC{\epsilon})$, we define the operator 
$$\overline{D}_{\VEC{t}}:= \sum_{j=1}^n \sign(D_j)|D_j|^{t_j}.$$ 
If $\VEC{t}\in \Omega(\VEC{\epsilon})\cap(0,1]^n$, the triple $(\mathcal{A},\mathpzc{H},\overline{D}_{\VEC{t}})$ defines a higher order spectral triple. If additionally the ST\textsuperscript{2} is regular, then the same holds for any \(\VEC{t}\in \Omega(\VEC{\epsilon})\).
\end{thm*}

We note that for $t>0$, the function
$$\R\ni x\mapsto \sign(x)|x|^t\in \R,$$
is continuous, and in fact belongs to the Hörmander class $S^t(\R)$ when localized to outside $x=0$.

The reader can find Theorem \ref{mainthm1} and its proof within Theorem \ref{thm:st2_give_host} below. We provide a number of examples of ST\textsuperscript{2}s throughout the paper and study the role of the transform $(\mathcal{A},\mathpzc{H},\DC)\mapsto (\mathcal{A},\mathpzc{H},\overline{D}_{\VEC{t}})$. In Subsection \ref{subsecex1}, we give a flavour of our main examples, the Rumin complex on the Heisenberg group, and two ``bad Kasparov products'' involving the group $C^*$-algebra of the Heisenberg group and a dynamical system on the torus. These examples are revisited in further detail and generality in Sections \ref{sec:rumin}, \ref{section:nilpotent}, and \ref{section:parabolic}. Finer analytical properties of ST\textsuperscript{2}s are studied in Section \ref{secljnakjnad}, for instance finite summability and equivariance properties. The interesting examples carry conformal actions, in the sense of recent work \cite{AMsomewhere} by the third listed author and Adam Rennie, and we discuss a ``guess-and-check'' method for conformal equivariance of ST\textsuperscript{2} in Subsection \ref{subsecvonf} that we later see in play in Sections \ref{sec:complexes}, \ref{sec:rumin}, and \ref{section:nilpotent}. It is a philosophy similar to that of computing Kasparov products via Kucerovsky's theorem \cite{danthedan}. \medskip

Strictly tangled spectral triples also arise from Hilbert complexes \cite{leschbruening}. We study ST\textsuperscript{2}s arising from Hilbert complexes in some detail in Section \ref{sec:complexes}, where the main example is that of Rockland complexes on filtered manifolds. Describing the noncommutative geometry of filtered manifolds is a non-trivial problem \cite{Hasselmann_2014}. Of particular interest is to associate higher order spectral triples possessing further properties with Rockland complexes. By choosing $\VEC{t}$ in Theorem \ref{mainthm1} appropriately we can produce higher order spectral triples from Rockland complexes either that are $H$-elliptic elements in the Heisenberg calculus or that are differential operators. We summarize the results of Section \ref{sec:complexes} in a Theorem.

\begin{thm*}
\label{mainthm2} 
Consider a compact filtered manifold $X$ equipped with a volume density and hermitian vector bundles $E_j\to X$, $j=0,...,n$. 
Assume that $(C^\infty(X;E_\bullet),\rd_\bullet)$ is a Rockland complex with all differentials being differential operators. Then there is an associated ST\textsuperscript{2}  $(C^\infty(X),L^2(X;\oplus_j E_j),\DC = (\rd_j + \rd_j^*)_j )$ as in Theorem \ref{ljnknknbgs}. Moreover, for any $\tau>0$, $\DC$ assembles into an $H$-elliptic pseudodifferential operator $\overline{D}_\tau$ on $\bigoplus_j E_j$ of order $\tau$, as in Corollary \ref{thm:rockland_H_ellipitc_host}, defining a higher order spectral triple $(C^\infty(X),L^2(X;\oplus_j E_j),\overline{D}_\tau)$.
\end{thm*}

In fact, the reader can find a version of Theorem \ref{mainthm2} stated with conformally equivariant actions as Proposition \ref{conaofnadin} below. To be somewhat more precise, assume that $G$ is a locally compact group acting as filtered automorphisms on $X$ and that $(C^\infty(X;E_\bullet),\rd_\bullet)$ is Rockland and $G$-equivariant with the action of $G$ on each $E_j$ being conformal (with respect to the volume density on $X$ and the hermitian structure on $E_j$). In Proposition \ref{conaofnadin} below we show that if the conformal factors in the different degrees are multiplicatively dependent (with respect to powers from $\Omega(\VEC{\epsilon})$) then we can assemble the associated ST\textsuperscript{2}  $(C^\infty(X),L^2(X;\oplus_j E_j), \DC)$ into a conformally equivariant higher order spectral triple.

A sobering observation is that, in practice, there are Rockland complexes equivariant for semisimple Lie groups of rank $>1$ but for which the action will not have a scalar conformal factor for each degree in the complex. Our framework cannot be applicable to semisimple Lie groups $G$ of rank $>1$. Indeed, by Theorem \ref{pmsumconf}, Proposition \ref{conaofnadin} would give a $G$-equivariant finitely summable bounded Fredholm module, which is impossible for a Lie group of rank $>1$ as shown by Puschnigg \cite{puschhigh}. The obstructions in higher rank are discussed in further detail in Remarks \ref{hosthighrank} and \ref{hosthighrank2}.
\medskip

Let us also mention another natural example of an ST\textsuperscript{2} built from the dual Dirac element of a nilpotent group. If $G$ is a simply connected nilpotent Lie group, the image of the dual Dirac element under the descent map \( KK^G_*(\C, C_0(G))\to KK_*(C^*(G),\C) \) produces a $K$-homology class on the group $C^*$-algebra. We discuss in Section \ref{section:nilpotent} how computing this element explicitly at the unbounded level produces an ST\textsuperscript{2}. We summarize the result as follows.

\begin{thm*}
Let \( G \) be a simply connected nilpotent Lie group of depth $s$ and \( H \) be a cocompact, closed subgroup (possibly \( G \) itself). Choose a Malcev basis \( ((e_{j, k})_{k=1}^{\dim \mathfrak{g}_j / \mathfrak{g}_{j+1}})_{j=1}^s \) of \( \mathfrak{g} \) through the lower central series \( \mathfrak{g}_1 = \mathfrak{g}, \mathfrak{g}_2 = [\mathfrak{g}, \mathfrak{g}], \ldots, \mathfrak{g}_s \). Let \( V \) be an irreducible Clifford module for \( \Cl_{\dim \mathfrak{g}} \), whose generators we label \( ((\gamma_{j, k})_{k=1}^{\dim \mathfrak{g}_j} )_{j=1}^s \). Then the collection 
\( (\ell_j)_{j=1}^s : G \to \End_\C(V) \)
of matrix-valued weights given by
\[ \ell_j : \exp_{\mathfrak{g}}\Biggl(\sum_{i=1}^s \sum_{k=1}^{\dim \mathfrak{g}_i / \mathfrak{g}_{i+1}} x_{i, k} e_{i, k}\Biggr) \mapsto \sum_{k=1}^{\dim \mathfrak{g}_j / \mathfrak{g}_{j+1}} x_{j, k} \gamma_{j, k} \]
gives rise to a strictly tangled spectral triple
\[ \left( C^*(H), L^2(H, V), (M_{\ell_n})_{n=1}^s \right) \]
with nontrivial class in \( KK_{\dim \mathfrak{g}}(C^*(H), \C) \) and bounding matrix \( \epsilon_{i j} = \max\{i - j, 0\} \). Moreover, the dual Dirac element of a cocompact closed subgroup of a nilpotent Lie group can be realized the Baaj--Skandalis dual of a strictly tangled spectral triple of the form above.
\end{thm*}

If the group \( G \) is Carnot, it is possible to obtain a higher order spectral triple for \( C^*(G) \) which is conformally equivariant under the dilation action. \medskip

In Section \ref{section:parabolic}, we show that parabolic dynamical systems \cite[Chapter 8]{Hasselblatt_2002} give rise to crossed product ST\textsuperscript{2}s, generalising the constructions of \cite{Cornelissen_2008,Bellissard_2010,Hawkins_2013,Paterson_2014} for elliptic dynamical systems. The following result appears as Corollary \ref{cor:parabolicmanifold}.

\begin{thm*}
	Let \( (C_c^{\infty}(X), L^2(X, S), D) \) be the Atiyah–Singer or Hodge–de Rham Dirac spectral triple on a complete Riemannian manifold \( (X, \mathbf{g}) \). Let \( \varphi \) be an action of a locally compact group \( G \) by diffeomorphisms on \( X \). Let \( \ell : G \to \End V \) be a self-adjoint, proper, translation-bounded weight where \(V\) is some finite-dimensional complex vector
	space. Suppose that $\varphi$ is parabolic in the sense that for some $s\geq 0$, the section $d\varphi_g:X\to \Hom(TX,g^*TX)$ satisfies an $L^\infty$-bound
	\[ \| d \varphi_g \|_{\infty} \leq C (1 + | \ell(g) |^s) \]
 for some constant \( C > 0 \) independent of $g$. Then
\[ (C_c^\infty(X) \rtimes G, L^2(G, V) \tildeotimes L^2(X, S), (M_{\ell} \tildeotimes 1, 1 \tildeotimes D) \]
is a strictly tangled spectral triple representing the Kasparov product of
\[ (C_c^\infty(X) \rtimes G, L^2(G, V) \otimes C_0(X)_{C_0(X)}, M_{\ell} \otimes 1) \]
and \( (C_c^\infty(X), L^2(X, S), D) \).
\end{thm*}

\subsection*{Acknowledgements}

AM thanks Lund University, and MG and MF, for their hospitality during a research stay in the Spring of 2024 and also acknowledges the support of an Australian Government RTP scholarship.

\section{Preliminaries}

We recall some notions from the literature in this section. Throughout the paper, we shall need somewhat flexible notions in spectral noncommutative geometry, namely higher order spectral triples and the notion of conformal equivariance of ditto.

\subsection{Higher order spectral triples}

In spectral noncommutative geometry one studies spectral triples as an analogue of a smooth Riemannian manifold, encoded via a first-order elliptic differential operator such as a Dirac operator. In the literature, as well as for the examples motivating this paper, the operators rarely behave like first-order operators and so we turn to higher order spectral triples. These latter have been studied, more or less independently, by several authors \cite[Lemma 51]{grenshomo}, \cite[Appendix 5.1]{wahlhomo}, \cite[Appendix A]{GMCK}. We recall the definition here to fix notation.

\begin{definition}
\label{defhost}
A \emph{higher order spectral triple (HOST)} consists of the data
$$(\mathcal{A},\mathpzc{H},D),$$
where $\mathpzc{H}$ is a Hilbert space, $\mathcal{A}\subseteq \mathbb{B}(\mathpzc{H})$ is a $*$-subalgebra, and $D$ is a self-adjoint operator on $\mathpzc{H}$ such that for an $\epsilon\in [0,1)$ we have that
\begin{enumerate}
\item $\mathcal{A}$ preserves the domain of $D$;
\item $D$ has locally compact resolvent, i.e.~$a(i\pm D)^{-1}\in \mathbb{K}(\mathpzc{H})$ for $a\in \mathcal{A}$;
\item for any $a\in \mathcal{A}$, the densely defined operator 
$$[D,a]\left(1+|D|^\epsilon\right)^{-1},$$
is bounded in the norm on $\mathpzc{H}$.
\end{enumerate}
We refer to the number $m=(1-\epsilon)^{-1}$ as the order of $(\mathcal{A},\mathpzc{H},D)$. If \( \mathcal{A} \) is unital and $(i\pm D)^{-1}\in \mathcal{L}^p(\mathpzc{H})$, a Schatten ideal, we say that $(\mathcal{A},\mathpzc{H},D)$ is \emph{$p$-summable}.
\end{definition}

Here we use the Schatten ideals with exponent $p>0$
$$\mathcal{L}^p(\mathpzc{H}):=\{T\in \mathbb{K}(\mathpzc{H}): \; \mathrm{Tr}(|T|^p)<\infty\}.$$
For any $p>0$, $\mathcal{L}^p(\mathpzc{H})$ is a symmetrically quasinormed ideal in the bounded operators and, for $p\geq 1$, it is a symmetrically normed ideal. We will also use the notation $\mathcal{L}^p(\mathpzc{H}_1, \mathpzc{H}_2):=\{T\in \mathbb{K}(\mathpzc{H}_1, \mathpzc{H}_2): \; |T|\in \mathcal{L}^p(\mathpzc{H}_2)\}.$

The prototypical example of a higher order spectral triple of order $m$ arises from an elliptic pseudodifferential operator $D$ of order $m$ on sections $E\to M$ of a hermitian vector bundle on a Riemannian manifold $M$. In this instance, $\mathcal{A}=C^\infty_c(M)$ and $\mathpzc{H}=L^2(M;E)$. If $M$ is not compact, some care is needed to equip $D$ with a domain making it self-adjoint. Much less than classical ellipticity of $D$ is needed, as discussed in numerous examples in \cite{frieskhom}. We note that a higher order spectral triple of order $m=1$ is nothing but a spectral triple.

A related notion is that of a bounded $K$-cycle, known also as a Fredholm module. A bounded $K$-cycle $(\mathpzc{H},F)$ over $A$ consists of a Hilbert space $\mathpzc{H}$ on which $A$ is represented and an operator $F\in \mathbb{B}(\mathpzc{H})$ such that $[F,a], a(F^*-F), a(F^2-1)\in\mathbb{K}(\mathpzc{H})$ for any $a\in A$. If \( A \) is unital, we say that $(\mathpzc{H},F)$ is $p$-summable if 
\begin{equation} F^*-F, F^2-1 \in\mathcal{L}^{p/2}(\mathpzc{H}) \quad \text{and} \quad [F, a]\in \mathcal{L}^p(\mathpzc{H}),
\end{equation}
for $a$ in a $*$-invariant dense subset of $A$. We interpret a bounded $K$-cycle, i.e.~Fredholm module, without a prescribed summability degree as having summability $p=\infty$. 

\begin{thm}
\label{thm:host_summability} cf. \cite[Theorem A.6]{GMCK}
Let $(\mathcal{A},\mathpzc{H},D)$ be a higher order spectral triple and write $A$ for the $C^*$-algebra closure of $\mathcal{A}$. Write 
$$F_D:=  D(1+D^2)^{-1/2}.$$
It holds that $(\mathpzc{H},F_D)$ is a Fredholm module for $A$ of the same parity as $(\mathcal{A},\mathpzc{H},D)$. If \( \mathcal{A} \) is unital and $(\mathcal{A},\mathpzc{H},D)$ has order $m$ and is $p$-summable, then the Fredholm module $(\mathpzc{H},F_D)$ is $mp$-summable over $\mathcal{A}$.
\end{thm}

\begin{proof}
We confine ourselves to proving the statement about summability, as the bounded transform for HOSTs has appeared several times in the literature. Note that \(F_D^2-1 = -(1+D^2)^{-1}\) and \(F_D^*-F_D=0\) ensuring that, if $(\mathcal{A}, \mathpzc{H}, D)$ has order $m$ and is $p$-summable, $ F_D^*-F_D, F_D^2-1 \in\mathcal{L}^{p/2}(\mathpzc{H})$.

To finish the proof we will show that $[F_D,a]\in \mathcal{L}^{mp}(\mathpzc{H})$ for $a\in \mathcal{A}$. Our argument is inspired by the method of Schrohe--Walze--Warzecha \cite{sww}. We can assume that $a^*=-a$ ensuring that $[F_D,a]$ is self-adjoint. Then by showing that an operator inequality of the form
\begin{equation}
\label{swwest}
    -C (1+D^2)^{-\frac{1}{2m}} \leq [F_D,a] \leq C (1+D^2)^{-\frac{1}{2m}}
\end{equation}
holds for some $C>0$ we shall obtain from a quadratic form argument that \([F_D,a]\in \mathcal{L}^{pm}(\mathpzc{H})\) for $a\in \mathcal{A}$. 

Using the standard resolvent integral formula
$$(1+D^2)^{-1/2}=\frac{1}{\pi} \int_1^\infty (\lambda-1)^{-\frac{1}{2}} (\lambda+D^2)^{-1} d\lambda,$$
and the fact that $a$ preserves the domain of $D$, we see that
\begin{multline}
	\label{eq:host_summability_resolvent_integral}
	[F_D,a] = \frac{1}{\pi} \int_1^\infty (\lambda-1)^{-\frac{1}{2}} \bigg( \lambda(\lambda+D^2)^{-1}[D,a](\lambda+D^2)^{-1} \\
	- D(\lambda+D^2)^{-1} [D, a](\lambda+D^2)^{-1} D \bigg) d\lambda
\end{multline}
holds pointwise on \(\Dom D\). Let $A(\lambda)=(\lambda + D^2)^{-\frac{1}{4}+\frac{1}{4m}}[D,a](\lambda + D^2)^{-\frac{1}{4}+\frac{1}{4m}}$ which is uniformly bounded in $\lambda\geq 1$ and self-adjoint. In particular, we have the estimates $-\|A(\lambda) \|\leq \pm A(\lambda) \leq \| A(\lambda) \|$. Since both terms in the integrand of \eqref{eq:host_summability_resolvent_integral} consist of \(\pm A(\lambda)\) conjugated by self-adjoint elements we can conclude that
$$
    - \sup_{\lambda\geq 1} \|A(\lambda)\| f(D) \leq [F_D,a] \leq  \sup_{\lambda\geq 1} \|A(\lambda)\| f(D)
$$
where
\begin{equation}
    f(x) = \frac{1}{\pi} \int_1^\infty (\lambda-1)^{-\frac{1}{2}} (\lambda+x^2)^{-\frac{1}{2}-\frac{1}{2m}}  d\lambda = \frac{1}{\sqrt{\pi}} \frac{\Gamma(\frac{1}{2m})}{\Gamma(\frac{1}{2}+\frac{1}{2m})} (1+x^2)^{-\frac{1}{2m}}.
\end{equation}
In particular, the estimate \eqref{swwest} holds.
\end{proof}

\begin{remark}
The proof of Theorem~\ref{thm:host_summability}, in particular the estimate \eqref{swwest}, shows that the singular values of $[F_D,a]$ satisfy that 
\begin{equation}
\label{asympsing}
\mu_k([F_D,a])\leq C\mu_k(D)^{-1/m}
\end{equation}
for the constant $C$ arising from the $m$-th order bound on $[D,a]$ as above. We can conclude that the summability bound \eqref{asympsing} is sharp from considering Weyl law for an elliptic pseudodifferential operator $D$ on a closed manifold $M$ and taking $a\in C^\infty(M)$ appropriately.
\end{remark}

\subsection{Conformal equivariance of higher order spectral triples}

Let us discuss group actions on spectral triples, beginning with the definition of conformal equivariance of a higher order spectral triple. We restrict the definition to the unital case.

\begin{definition}
	\cite[Definition 3.19]{AMsomewhere}
	\label{def:conf_geq_host}
	Let $G$ be a locally compact group. A unital \emph{conformally \( G \)-equivariant} higher order spectral triple of order $m$ is a higher order spectral triple $(\mathcal{A},\mathpzc{H},D)$ of order $m=(1-\epsilon)^{-1}$ together with a unitary action $U$ of $G$ on $\mathpzc{H}$ and a \(*\)-strongly continuous family of invertible bounded operators \( (\mu_g)_{g \in G} \) such that
	\begin{enumerate}
		\item the action of \(G\) implements an action on the unital $*$-algebra $\mathcal{A}$ in the sense that \( \alpha_g(a) := U_g a U_g^* \) preserves $\mathcal{A}$;
		\item for all $g\in G$, $\mu_g$ and $\mu_g^*$ preserve $\Dom D$, with 
		\[
			g \mapsto [D, \mu_g] (1 + |D|^{\epsilon})^{-1}\quad\mbox{and} \quad
			g \mapsto [D, \mu_g^*] (1 + |D|^{\epsilon})^{-1}
		\]
		defining $*$-strongly continuous maps from \( G \) into the space of bounded operators on $\mathpzc{H}$; and
		\item for all $g\in G$, $U_g$ preserves $\Dom D$, with the maps
		\begin{align*}
			g & \mapsto (U_g D U_g^* - \mu_g D \mu_g^*) (1 + |D|^{\epsilon})^{-1} \quad \text{and} \\
			g & \mapsto U_g (1 + |D|^{\epsilon})^{-1} U_g^* (U_g D U_g^* - \mu_g D \mu_g^*)
		\end{align*}
		$*$-strongly continuous from \( G \) into the space of bounded operators on $\mathpzc{H}$.
	\end{enumerate}
	If $\mu_g=1$ for all $g$, we say that $(\mathcal{A},\mathpzc{H},D)$ is \emph{uniformly $G$-equivariant} or just \emph{\( G \)-equivariant}.
\end{definition}

The family \( (\mu_g)_{g \in G} \) from Definition \ref{def:conf_geq_host} is an auxiliary tool and is clearly not unique, e.g. for a $G$-equivariant higher order spectral triple $(\mathcal{A},\mathpzc{H},D)$ Definition \ref{def:conf_geq_host} is fulfilled for any \(*\)-strongly continuous family of scalars \( (\mu_g)_{g \in G} \subseteq U(1)\). For the definition of conformal $G$-equivariance in the further generality of unbounded Kasparov modules that are not necessarily unital, we refer to \cite[Definition 3.19]{AMsomewhere}. 

\begin{remark}
	We may encode the unitary implemeter \( U \) of \( G \) and the conformal factor \( \mu \) as a pair \( (U, \mu) \). In good circumstances, $\mu$ is a cocycle for the action $U$, i.e.
	$$\mu_{gg'}=\mu_gU_g\mu_{g'}U_g^*.$$ 
	Such rigidity is not required by Definition \ref{def:conf_geq_host}, although $\mu$ is a cocycle for the action $U$ ``up to lower order'' by conditions (2) and (3). If, however, the cocycle property holds, we may consider
	$$(\mu,U):G\to \mathrm{GL}(\mathpzc{H})\rtimes U(\mathpzc{H})$$
	to be a continuous group homomorphism. Here we use the full conformal group of a Hilbert space $\mathpzc{H}$ defined to be the crossed product
$$\mathfrak{C}(\mathpzc{H})=\mathrm{GL}(\mathpzc{H})\rtimes U(\mathpzc{H}),$$
where $U(\mathpzc{H})$ acts by conjugation on $\mathrm{GL}(\mathpzc{H})$. An element of $\mathfrak{C}(\mathpzc{H})$ is a pair $(\mu,U)$ and products are given by 
$$(\mu,U)(\mu',U')=(\mu U\mu'U^*,UU').$$
The reader can note that a homomorphism $(\mu,U) : G \to \mathfrak{C}(\mathpzc{H})$ is the same as a choice of lift of a continuous homomorphism $G\to \mathrm{GL}(\mathpzc{H})$ along the surjection $\mathfrak{C}(\mathpzc{H})\to \mathrm{GL}(\mathpzc{H})$, $(\mu,U)\mapsto \mu U$. Of course, there are plenty of possible lifts of a continuous homomorphism $V:G\to \mathrm{GL}(\mathpzc{H})$ along the surjection $\mathfrak{C}(\mathpzc{H})\to \mathrm{GL}(\mathpzc{H})$ and there always exist lifts (e.g.~the trivial lift, $U=1$ and $\mu=V$).  Moreover, fixing $U$ then the cohomology class $[\mu]\in H^1(G,\mathrm{GL}(\mathpzc{H}))$ is the obstruction for the representation $G\ni g\mapsto \mu_g U_g\in \mathrm{GL}(\mathpzc{H})$ to be equivalent to the unitary representation $G\ni g\mapsto U_g\in U(\mathpzc{H})$. 
\end{remark}

\begin{example}
\label{hodgeandsign}
A prototypical example of a conformally $G$-equivariant spectral triple arises from conformal actions on Riemannian manifolds. This example is discussed in more detail in \cite[Examples 2.12 \& 3.24]{AMsomewhere}, 
so we restrict ourselves to recalling its main features. Assume that the group $G$ acts conformally and smoothly on an oriented, $d$-dimensional, Riemannian manifold $M$ and that $G$ preserves the orientation. To reduce analytic complications, we assume that $M$ is compact but similar statements hold also when $M$ is just complete. If $D$ denotes the signature operator $\rd+\rd^*$ acting on differential forms realized as its closure on $L^2(M;\wedge^*_\C T^*M)$, $(C^\infty(M),L^2(M;\wedge^*_\C T^*M),D)$ is a spectral triple. The action of $G$ lifts to an action via pullback on $k$-forms $L^2(M;\wedge^k_\C T^*M)$, we denote this action by $V_k:G\to \mathrm{GL}(L^2(M;\wedge^k_\C T^*M))$. We write $\lambda\in C^\infty(M\times G,\R_{>0})$ for the conformal factor; i.e.~for $g\in G$, $g^*\mathrm{g}=\lambda(g)^2\mathrm{g}$ where $\mathrm{g}$ denotes the Riemannian metric. A direct computation shows
$$V_k(g)^*V_k(g)=\lambda(g)^{d-k},$$
and the action $V_k$ unitarizes to $U_k=V_k\lambda^{-(d-k)/2}:G\to U(L^2(M;\wedge^k_\C T^*M))$. We set $U:=\bigoplus_{k=0}^d U_k$.
For any $\VEC{a}=(a_0,\ldots, a_d)\in \R^{d+1}$ we can define an ansatz conformal factor by
$$\mu_{\VEC{a}}:=\bigoplus_{k=0}^d \lambda^{a_k}.$$
The collection $(C^\infty(M),L^2(M;\wedge^*_\C T^*M),D)$ is a spectral triple, so its order is $m=1$. For the argument that follows, we can in fact view the order to be any number $m\geq 1$, so take $\epsilon\in [0,1)$. By computing the principal symbol, we see that, if $\lambda(g)$ is not the constant function $1$ for some $g$,
\[ (U_g D U_g^* - \mu_{\VEC{a},g} D \mu_{\VEC{a},g}^*) (1 + |D|^{\epsilon})^{-1} \]
is bounded for all $g$ if and only if $\VEC{a}=(a_0,\ldots, a_d)\in \R^{d+1}$ satisfies
$$a_{k-1}+a_{k}=1,\quad\mbox{for} \quad k=1,\ldots,d.$$
On the other hand, principal symbol computations show that $[D, \mu_g] (1 + |D|^{\epsilon})^{-1}$ is bounded if and only if all objects in condition (2) of Definition \ref{def:conf_geq_host} are bounded, and this holds if and only if 
$$a_0=a_1=\cdots =a_d.$$ 
In particular, $(C^\infty(M),L^2(M;\wedge^*_\C T^*M),D)$ is a conformally $G$-equivariant (higher order) spectral triple for the conformal factor $\mu_{\VEC{a}}$ if and only if $a_j=1/2$ for all $j$. Below we return to similar examples arising in parabolic geometry rather than in Riemannian geometry. The need for conformal factors to `conform', as it were, will be crucial in such examples.
\end{example}

For a locally compact group $G$ and a unitary action \(U\) on \(\mathpzc{H}\) that implements an action $\alpha_g(a):=U_gaU_g^*$ on \(A\), then a bounded \(K\)-cycle $(\mathpzc{H},F)$ together with the unitary action of \(G\) is said to be \(G\)-equivariant if \(g\mapsto a(U_gFU_g^* - F)\) is a norm-continuous map into \(\mathbb{K}(\mathpzc{H})\) for all \(a\in A\) and \(g\in G\). If additionally $A$ is unital, we say that $(\mathpzc{H},F)$ is a \(p\)-summable \(G\)-equivariant bounded \(K\)-cycle if $(\mathpzc{H},F)$ is $p$-summable and \(U_gFU_g^* - F\in \mathcal{L}^p(\mathpzc{H})\) for all $g\in G$ cf. \cite[Definition 2.7]{puschhigh}.

\begin{thm}
\label{pmsumconf}
Let $(\mathcal{A},\mathpzc{H},D)$ be a conformally $G$-equivariant higher order spectral triple. Writing $A$ for the $G$-$C^*$-algebra closure of $\mathcal{A}$, the Fredholm module $(\mathpzc{H},F_D)$ for $A$ is $G$-equivariant. If $(\mathcal{A},\mathpzc{H},D)$ has order $m$ and is $p$-summable, the $G$-equivariant Fredholm module $(\mathpzc{H},F_D)$ is \( q \)-summable over $\mathcal{A}$ for any \( q > mp \).
\end{thm}

The reader can find a proof of this statement in \cite[Theorem 3.21]{AMsomewhere}, with the summability following from the fact that \( (U_g F_D U_g^* - F_D) (1 + D^2)^{\beta/2} \) is bounded for \( \beta < 1 - \epsilon = m^{-1} \).

\begin{remark}
\label{hosthighrank}
We note in particular that there are obstructions to finite summability which persist also in the setting above.
Connes' obstruction \cite{connesmetric} (see also \cite{GRU}) shows that there are no finitely summable $G$-equivariant higher order spectral triples over $A$ if $A\rtimes G$ is purely infinite. Puschnigg's generalization \cite{puschhigh} of the rigidity results of Bader--Furman--Gelander--Monod \cite{baderetal} goes even further when $G$ is a higher rank lattice and implies in combination with Theorem \ref{pmsumconf} that there are no conformally $G$-equivariant, finitely summable, higher order spectral triples over a unital $A$.
\end{remark}

\section{Strictly tangled spectral triples (ST\texorpdfstring{\textsuperscript{2}}{2}), definition and examples}

We now come to the main object of study in this paper, the strictly tangled spectral triples which we abbreviate ST\textsuperscript{2}.

\subsection{Anticommuting collections of operators}
\label{subsecprel1}

The operator $D$ in a higher order spectral triple will in an ST\textsuperscript{2} be replaced by a collection of operators $\DC$. We first introduce some terminology, make some preliminary observations and provide some examples that motivate our definition of ST\textsuperscript{2}s.

\begin{definition}
\label{formandst}
Given a finite collection of self-adjoint operators $\DC=(D_j)_{j\in I}$ on a Hilbert space $\mathpzc{H}$ sharing a common core, and $\VEC{t}=(t_j)_{j\in I}\in(0,\infty)^I$, we define the positive self-adjoint operator $\Delta^{\VEC{t}}_\DC$ from the positive, closed, quadratic form 
$$\mathfrak{q}_{\VEC{t}}(u):=\sum_{j\in I} \||D_j|^{t_j} u\|^2_{\mathpzc{H}}, \quad u\in \cap_{j\in I} \Dom(|D_j|^{t_j}).$$
We use the notation $\Delta^{\VEC{t}}_\DC=\sum_{j\in I} |D_j|^{2t_j}$ and equipp $\Dom(\mathfrak{q}_{\VEC{t}})=\cap_{j\in I} \Dom(|D_j|^{t_j})$ with the Hilbert space structure defined from the norm 
$$\|u\|_{\VEC{t}}^2:=\|u\|^2_{\mathpzc{H}}+\mathfrak{q}_{\VEC{t}}(u).$$

A collection of self-adjoint operators $\DC=(D_j)_{j\in I}$ on $\mathpzc{H}$ is said to be strictly anti-commuting if, for all \( j \neq k \),
$$D_jD_k+D_kD_j=0$$
on a common core invariant under $(D_j)_{j\in I}$.
\end{definition}

\begin{example}
\label{example:extprod}
Assume that $M_1$ and $M_2$ are two oriented, compact, Riemannian manifolds, with Clifford bundles $E_1\to M_1$ and $E_2\to M_2$ with Dirac operators $\slashed{D}_1$ and $\slashed{D}_2$ thereon. For simplicity, we assume that the manifolds are even dimensional so all Clifford bundles and Dirac operators are graded. We write $E_1 \mathbin{\tilde{\boxtimes}} E_2\to M_1\times M_2$ for their graded exterior tensor product. By construction, the pair of operators 
$$D_1:=\slashed{D}_1 \tildeotimes 1_{E_2} \quad\mbox{and}\quad D_2:=1_{E_1}  \tildeotimes \slashed{D}_2 $$
form a strictly anticommuting collection on the Hilbert space $L^2(M_1\times M_2, E_1 \mathbin{\tilde{\boxtimes}} E_2)$. Here the domain of $D_1$ is the (graded) Hilbert space tensor product $H^1(M_1,E_1) \tildeotimes L^2(M_2,E_2)$ and the domain of $D_2$ is the (graded) Hilbert space tensor product $L^2(M_1,E_1)\tildeotimes H^1(M_2,E_2)$. Here $D:=D_1+D_2$ is a Dirac operator on the Clifford bundle $E_1 \mathbin{\tilde{\boxtimes}} E_2\to M_1\times M_2$. A similar construction can also be made for a foliated manifold \cite{connesskandalis,kordy}, with $D_1$ being a tangential Dirac operator and $D_2$ a transversal Dirac operator but in this case $D_1D_2+D_2D_1$ is generally not zero, and only lower order if the foliation is Riemannian. 
\end{example}

\begin{example}
We can more generally consider the (constructive) external Kasparov product. If $(\mathcal{A}_1,\mathpzc{H}_1,D_1)$ and $(\mathcal{A}_2,\mathpzc{H}_2,D_2)$ are two even higher order spectral triples, their external Kasparov product is constructed as $(\mathcal{A}_1\otimes\mathcal{A}_2,\mathpzc{H}_1\tildeotimes\mathpzc{H}_2,D_1\tildeotimes 1+1\tildeotimes D_2)$. Here $( D_1\tildeotimes 1,1\tildeotimes D_2)$ form a strictly anticommuting collection on the Hilbert space $\mathpzc{H}_1\tildeotimes \mathpzc{H}_2$ and their sum is the operator in the external Kasparov product. This example goes back to Baaj--Julg's seminal paper \cite{baajjulg} where the unbounded picture was first introduced. The two pairs of strictly anticommuting operators discussed in this example will fit into the framework of ST\textsuperscript{2}s discussed in the next section (see Definition \ref{def:st2}). 
\end{example}

\begin{example}
	A more simple-minded example is the direct sum of two higher order spectral triples. If $(\mathcal{A}_1,\mathpzc{H}_1,D_1)$ and $(\mathcal{A}_2,\mathpzc{H}_2,D_2)$ are two higher order spectral triples, their direct sum is $(\mathcal{A}_1 \oplus \mathcal{A}_2,\mathpzc{H}_1 \oplus \mathpzc{H}_2, D_1 \oplus D_2)$. Albeit in a somewhat trivial way, $( D_1 \oplus 0, 0 \oplus D_2)$ form a strictly anticommuting collection on the Hilbert space $\mathpzc{H}_1 \oplus \mathpzc{H}_2$. 
\end{example}

\begin{example}
\label{dolcom}
Let $M$ be a compact Kähler manifold. Write $d$ for the complex dimension of $M$. We can consider the Dolbeault complex 
\begin{multline*}
0\to C^\infty(M)\xrightarrow{\bar{\partial}_1}C^\infty(M,\Lambda^1)\xrightarrow{\bar{\partial}_2}C^\infty(M,\Lambda^2)\xrightarrow{\bar{\partial}_3}\cdots\\
 \cdots \xrightarrow{\bar{\partial}_{d-1}}C^\infty(M,\Lambda^{d-1})\xrightarrow{\bar{\partial}_d}C^\infty(M,\Lambda^d)\to 0.
\end{multline*}
Here $\Lambda^k:=\wedge^k_\C T^{0,1}M$ denotes the exterior algebra of the $(0,1)$-forms. We also write $\Lambda:=\bigoplus \Lambda^k$. The operators $\slashed{\partial}_j$ obtained as the closure of $\bar{\partial}_j+\bar{\partial}_j^*$ on $L^2(M;\Lambda)$ satisfy for $j\neq k$
$$\slashed{\partial}_j\slashed{\partial}_k=0=-\slashed{\partial}_k\slashed{\partial}_j.$$
In particular, the collection $(\slashed{\partial}_j)_{j=1}^d$ is a strictly anticommuting collection of operators on $L^2(M;\Lambda)$. The collection of strictly anticommuting operators discussed in this example will fit into the framework of ST\textsuperscript{2}s discussed in the next section (see Definition \ref{def:st2}). 

We can in fact for any partition $\{1,2,\ldots, d-1,d\}=S_1 \sqcup \cdots \sqcup S_n$ form $D_l:=\sum_{j\in S_l} \slashed{\partial}_j$ and the collection $\DC = (D_l)_{l=1}^n$ also forms a strictly anticommuting collection. In both of these constructions
$$\Delta^{\DC}_{(1,1,\ldots,1)}=\sum_{j=1}^d \slashed{\partial}_j^2=\sum_{j=1}^d\bar{\partial}_j^*\bar{\partial}_j+\bar{\partial}_j\bar{\partial}_j^*=\sum_{l=1}^n D_l^2 $$ 
is the Kodaira Laplacian. In later examples arising from complexes, we see that the orders of the differentials affect which partitions we can choose when building an ST\textsuperscript{2}; see in particular Remark \ref{p25rem}.
\end{example}

\begin{lemma}
If $\DC=(D_j)_{j\in I}$ is a strictly anti-commuting collection of self-adjoint operators on $\mathpzc{H}$ and $\VEC{t}\in (0,\infty)^I$, then the operator 
$$\overline{D}_{\VEC{t}}:= \sum_{j\in I} \sign(D_j)|D_j|^{t_j}$$ 
is self-adjoint with $\Dom(\overline{D}_{\VEC{t}})$ coinciding with the form domain of $\Delta^{\VEC{t}}_\DC$. Moreover, as densely defined operators,
$\overline{D}^2_{\VEC{t}}=\Delta^{\VEC{t}}_\DC$.
\end{lemma}

\begin{proof}
Follows from \cite{kaadlesch} and an induction argument.
\end{proof}

Recall from Definition \ref{formandst} that $\cap_{j\in I} \Dom(|D_j|^{t_j})$ carries a natural Hilbert space structure for any $\VEC{t}\in (0,\infty)^I$.

\begin{lemma}
\label{lknlknn}
Let $\DC=(D_j)_{j\in I}$ be a strictly anticommuting collection of self-adjoint operators on $\mathpzc{H}$. Then the following are equivalent for a bounded operator \( a \) on $\mathpzc{H}$:
\begin{itemize}
\item $a(1+\Delta^{\VEC{t}}_\DC)^{-1} \in \mathbb{K}(\mathpzc{H})$ for any $\VEC{t}\in (0,\infty)^I$;
\item $a(1+\Delta^{\VEC{t}}_\DC)^{-1} \in \mathbb{K}(\mathpzc{H})$ for some $\VEC{t}\in (0,\infty)^I$; and 
\item \(a\) as a map \(a : \cap_{j\in I} \Dom(D_j) \to \mathpzc{H}\) is compact.
\end{itemize}
In fact, for  $\VEC{s}, \VEC{t}\in (0,\infty)^I$ and $\sigma, \tau \in (0, \infty)$  such that \(\sigma s_i \leq \tau t_i\) there exists a constant $C>0$ for which 
$$(1 + \Delta^{\VEC{s}}_\DC)^\sigma\leq C(1 + \Delta^{\VEC{t}}_\DC)^\tau$$
in the form sense.
\end{lemma}

\begin{proof}
Noting that $(D_j^2)_{j\in I}$ is a strictly commuting collection of self-adjoint operators, the lemma follows from functional calculus of several commuting operators.
\end{proof}

\subsection{Bounds and tropical combinatorics}

We will encode the orders (relating to commutator properties as in a higher order spectral triple) of the components in a finite collection $\DC=(D_j)_{j\in I}$ of self-adjoint operators in a matrix $\VEC{\epsilon}=(\epsilon_{ij})_{i,j\in I}\in M_I([0,\infty))$. Here we write $M_I$ for the matrices indexed by a finite set $I$. We think of $\VEC{\epsilon}$ pictorially as a weighted directed graph. The weighted directed graph has vertices labelled by $I$ and there is an edge from $i$ to $j$ labelled by $\epsilon_{ij}$ whenever $\epsilon_{ij}> 0$.
For instance, the diagram
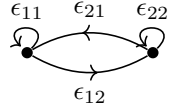
\[ \begin{tikzpicture}
\graph[grow right sep=1.5cm, empty nodes, nodes={fill=black, circle, inner sep=1.5pt}, edges={semithick}]{
    a --["$\epsilon_{11}$"', loop] a --["$\epsilon_{21}$" inner sep=5.5pt, middlearrow={<}, bend left] b --["$\epsilon_{22}$"', loop] b --["$\epsilon_{12}$" inner sep=5.5pt, middlearrow={<}, bend left] a;};
\end{tikzpicture} \]
pictorially describes a $2\times2$ matrix $\VEC{\epsilon}=(\epsilon_{ij})_{i,j=1}^2\in M_2([0,\infty))$. Such diagrams motivate our use of the word `tangled' for the main concept of this paper.

\begin{definition}
\label{deccoccoo}
We say that a matrix $\VEC{\epsilon}=(\epsilon_{ij})_{i,j\in I}\in M_I([0,\infty))$ satisfies the \emph{decreasing cycle condition} if for any $k$ and $\gamma=(\gamma_1,\gamma_2,\ldots, \gamma_k)\in I^k$ with $\gamma_1=\gamma_k$ we have that 
$$\prod_{j=1}^{k-1} \epsilon_{\gamma_j\gamma_{j+1}}<1.$$
\end{definition}

The decreasing cycle condition means that the total weight along any cycle in the weighted directed graph should be $<1$. The condition that $\prod_{j=1}^k \epsilon_{\gamma_j\gamma_{j+1}}<1$ is indeed only a condition appearing along the cycles in the weighted digraph associated with $\epsilon$ since $\gamma=(\gamma_1,\gamma_2,\ldots, \gamma_k)\in I^k$ represents a cycle if and only if $\prod_{j=1}^k \epsilon_{\gamma_j\gamma_{j+1}}>0$. In particular, if the weighted digraph associated with $\epsilon$ has no cycles then $\epsilon$ automatically satisfies the decreasing cycle condition. If $n=1$ then $\VEC{\epsilon}\in [0,\infty)$ satisfies the decreasing cycle condition if and only if $\VEC{\epsilon}<1$.
It follows from \cite[Lemma 3.23]{joswigess} that $\VEC{\epsilon}\in M_I([0,\infty))$ satisfies the decreasing cycle condition if and only if the convex cone
$$\Omega(\VEC{\epsilon}):=\{ \VEC{t} = (t_j)\in (0,\infty)^I: \epsilon_{ij}t_i<t_j\; \forall i,j\}$$
is non-empty.

We can interpret \( \VEC{\epsilon} \) as a matrix valued in the \emph{tropical semiring}, in which context \( \Omega(\VEC{\epsilon}) \) is a well-studied object. The tropical semiring, in the multiplicative convention, is \( [0, \infty) \) with addition \( \oplus \) given by \( x \oplus y = \max\{ x, y \} \) and multiplication \( \times \) defined just as usual. Remark that \( 0 \) is the additive identity, \( 1 \) is the multiplicative identity, and multiplication distributes over addition. 
The reader can find more details on matrices in the tropical semiring and their relationship to weighted directed graphs in \cite{joswigess} (where an additive convention is used for the tropical semiring, related to our multiplicative convention by the logarithm). In this context, \( \Omega(\VEC{\epsilon}) \) is called the \emph{weighted digraph polyhedron} and the decreasing cycle condition is related to the existence of nonzero tropical eigenvalues for \( \VEC{\epsilon} \).
It seems likely that there is more to be gleaned from interpreting \( \VEC{\epsilon} \) as a matrix over the tropical semiring but, for the purposes of this paper, it suffices to remember the nonemptiness of the cone \( \Omega(\VEC{\epsilon}) \) as the antecedent of the decreasing cycle condition.

\subsection{Strictly tangled spectral triples}

We now come to the main definition of this section.

\begin{definition}
\label{def:st2}
A strictly tangled spectral triple (ST\textsuperscript{2}) consists of the data
$$(\mathcal{A},\mathpzc{H},\DC),$$
where $\mathpzc{H}$ is a Hilbert space and $\mathcal{A}\subseteq \mathbb{B}(\mathpzc{H})$ is a $*$-subalgebra, and $\DC=(D_j)_{j\in I}$ is a finite collection of strictly anti-commuting self-adjoint operators on $\mathpzc{H}$ such that for a matrix $\VEC{\epsilon} \in M_I([0,\infty))$ satisfying the decreasing cycle condition (see Definition \ref{deccoccoo}) we have that
\begin{enumerate}
\item $\mathcal{A}$ preserves the domains of each $D_j$;
\item the collection $\DC$ has locally compact resolvent in the sense that, for every \( a \in \mathcal{A} \),
	$a(1+\Delta^{\VEC{t}}_\DC)^{-1}  \in \mathbb{K}(\mathpzc{H})$ for some $\VEC{t}\in (0,\infty)^I$ (or any of the other equivalent conditions of Lemma \ref{lknlknn}); and
\item for any $a\in \mathcal{A}$ and $i\in I$, the densely defined operator
$$[D_i,a] \Biggl(1+\sum_{j\in I}|D_j|^{\epsilon_{ij}} \Biggr)^{-1}$$
is bounded in the norm on $\mathpzc{H}$.
\end{enumerate}
We refer to $\VEC{\epsilon}$ as a bounding matrix.
If additionally $\mathcal{A}$ and \([D_j,\mathcal{A}]\) for each \(j\) preserve $\Dom(\Delta^{\VEC{t}}_\DC)$, for any $\VEC{t}\in [0,\infty)^l$, we say that $(\mathcal{A},\mathpzc{H},\DC)$ is a regular ST\textsuperscript{2}. 

If $\mathpzc{H}$ is graded, \(\mathcal{A}\) consists of even operators and all operators in $\DC$ are odd, we say that $(\mathcal{A},\mathpzc{H},\DC)$ is an even ST\textsuperscript{2}. If $\mathpzc{H}$ carries no grading we say that $(\mathcal{A},\mathpzc{H},\DC)$ is an odd ST\textsuperscript{2}.
\end{definition}

\begin{remark}
A strictly tangled spectral triple $(\mathcal{A},\mathpzc{H},\DC)$ with $n=1$ is the same as a higher order spectral triple as defined in Definition \ref{defhost} above. Recall from above that for $n=1$ the decreasing cycle condition on $\VEC{\epsilon}\in [0,\infty)$ is equivalent to $\VEC{\epsilon}<1$. In this case, if $\VEC{\epsilon}\in (0,1)$ is the bounding matrix then $(\mathcal{A},\mathpzc{H},\DC)$ is a higher order spectral triple of order $m=(1-\VEC{\epsilon})^{-1}$. Furthermore, we point out that there is an implicit lower bound $m\geq 1$ on the order of our higher order spectral triples originating in the requirement on $\VEC{\epsilon}$ to have coefficients in $[0,\infty)$ that in turn is set in order to apply the toolbox of tropical geometry \cite{joswigess}.
\end{remark}

\begin{remark}
\label{remark:possible_extensions_of_st2}
There is a straightforward extension of Definition \ref{def:st2} to strictly tangled unbounded Kasparov modules but, in the interests of exposition, we do not pursue it here.

We note that our definition of a strictly tangled spectral triple is somewhat restrictive in requiring the elements of the collection $\DC=(D_j)_{j\in I}$ to be strictly anti-commuting. We expect that this definition can be relaxed to include collections $\DC=(D_j)_{j\in I}$ on which there is a size constraint on the anti-commutator $D_jD_k+D_kD_j$ along the lines of for instance \cite{bramlesch}. For our applications to complexes, in particular Rockland complexes, we will make do with strictly anti-commuting collections but in order for more general applications to Rockland sequences \cite{davehaller,goffengkuzmin,goffeng24} and more general Kasparov product constructions \cite{GMCK,leschkaad1,meslandbeast} to fit into the framework one needs to extend the notion above to a weaker anticommutation condition. See Remark \ref{relboundgen} for futher comments on where in the proofs this is used.
\end{remark}

\begin{remark}
	\label{remark:operators_with_prescribed_orders}
If the operators \(\DC=(D_j)_{j\in I}\) have a prescribed order \(\VEC{m}=(m_j)_{j\in I}\in [1,\infty)^I\) in an appropriate sense, e.g.~in some pseudodifferential calculus, there is an intuitive guess of bounding matrix \(\VEC{\epsilon}\). Similar to the intuition of a higher order spectral triple, a commutator \([D_i,a]\) should behave like one order lower than \(D_i\) and therefore be controlled by operators of order \(m_i-1\). Therefore, a natural choice that turns out to be correct in examples is the bounding matrix \(\epsilon_{ij}=\frac{m_i - 1}{m_j}\), for \(i,j\in I\), represented by the weighted digraph
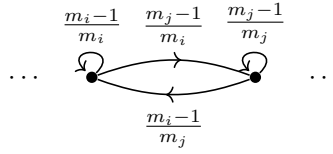
\[\cdots \vcenter{\hbox{
		\begin{tikzpicture}
			\graph[grow right sep=2cm, empty nodes, nodes={fill=black, circle, inner sep=1.5pt}, edges={semithick}]{
			a --["$\frac{m_i - 1}{m_j}$"' inner sep=5.5pt, middlearrow={<}, bend right=20] b;
			b --["$\frac{m_j - 1}{m_i}$"' inner sep=5.5pt, middlearrow={<}, bend right=20] a;
			a --["$\frac{m_i - 1}{m_i}$"', loop] a;
			b --["$\frac{m_j - 1}{m_j}$"', loop] b;
			};
		\end{tikzpicture}
	}} \cdots \]
which also matches the order of a higher order spectral triple in that \(m_i=(1-\epsilon_{ii})^{-1}\). Such an \(\VEC{\epsilon}\) fulfils the decreasing cycle condition since 
\begin{equation}
\label{estforesp}
\prod_{j=1}^k \epsilon_{\gamma_j,\gamma_{j+1}} = \frac{\prod_{j=1}^k (m_{\gamma_j} - 1)}{\prod_{j=1}^k m_{\gamma_{j+1}}} =\frac{\prod_{j=1}^k (m_{\gamma_j} - 1)}{\prod_{j=1}^k m_{\gamma_j}} = \prod_{j=1}^k \left(1 - \frac{1}{m_{\gamma_j}}\right) < 1
\end{equation}
	for any cycle \(\gamma=(\gamma_1,\ldots,\gamma_k)\), where we use the cycle property $\gamma_1=\gamma_k$ in the second equality. In particular, \(\Omega(\VEC{\epsilon})\) contains a ray of the form
	$$\VEC{t}_{\VEC{m}}(\tau):=\left(\frac{\tau}{m_j}\right)_{j\in I}\in \Omega(\VEC{\epsilon}), \quad \tau>0.$$
Indeed, $\VEC{t}_{\VEC{m}}(\tau)\in \Omega(\VEC{\epsilon})$ since \(\epsilon_{ij}t_i=\epsilon_{ii}t_j<t_j\). The operator 
	\[\overline{D}_\tau:=\overline{D}_{\VEC{t}_{\VEC{m}}(\tau)} = \sum_{j\in I} \sign(D_j) |D_j|^{\frac{\tau}{m_j}}\]
	constructed from this ray should then morally be a sum of operators of order \(\tau\), which is discussed further in Remark \ref{remark:order_of_st2_with_prescribed_orders} and placed in a solid mathematical foundation in Proposition \ref{fixorder}.
	
	Operators with prescribed orders \(\VEC{m}\) with this type of bounding matrix \(\VEC{\epsilon}\) will be considered further in Section \ref{sec:complexes} in the context of complexes. For an ST\textsuperscript{2} arising from a complex, we will depart from the preceding discussion by setting \(\epsilon_{ij}=0\) if the operators \(D_i\) and \(D_j\) are ``far apart'' in the complex.
	
	When we consider the $C^*$-algebras of nilpotent groups in Section \ref{section:nilpotent}, we will see that one does not always have a natural prescription of orders. However, in the case of a Carnot group, there will be a natural way of assigning orders, related to conformal equivariance under the dilation action.
\end{remark}

\subsection{Examples}
\label{subsecex1}

Before delving into the general theory and the main examples of the paper, we provide some simpler examples to clarify and justify the structure underlying ST\textsuperscript{2}s. Further examples, generalising these, will be presented in Sections \ref{sec:complexes}, \ref{sec:rumin}, \ref{section:nilpotent}, and \ref{section:parabolic}.

\subsubsection{The Rumin complex in three dimensions on a nilmanifold}
\label{rubsecrumin3d}

The Rumin complex is an example of a Rockland complex. We will consider Rockland complexes in Section \ref{sec:complexes} and return to explain and study Rumin complexes in more generality in Section \ref{sec:rumin}. Let us start with the simplest situation to explain the ideas motivating the notion of ST\textsuperscript{2}s. We consider the 3-dimensional Heisenberg group \( \sfH_3 \). As a manifold, \( \sfH_3 \) coincides with $\R^3$ but is equipped with the product $(x,y,z)(x',y',z')=(x+x',y+y',z+z'+xy')$. We write $\Gamma$ for the cocompact subgroup defined from the integer points $\Z^3$. On the nilmanifold $M= \sfH_3/\Gamma$, the Rumin complex takes the form
\[
     \begin{tikzcd}[cells={nodes={font=\small}}, sep=small]
         & & C^{\infty}(M) \rd x \ar{r}{Z+XY} \ar{rdd}[pos=0.8]{Y^2} &[3em] C^{\infty}(M) \rd x \wedge \theta \ar{rd}{Y} & &\\
         0 \ar{r} & C^{\infty}(M) \ar{ru}{X} \ar{rd}[swap]{Y} & & & C^{\infty}(M) \rd x \wedge \rd y \wedge \theta \ar{r} & 0 \\
         & & C^{\infty}(M) \rd y \ar{r}[swap]{Z-YX} \ar{ruu}[pos=0.9,swap]{-X^2}  & C^{\infty}(M) \rd y \wedge \theta \ar{ru}[swap]{-X} & &
     \end{tikzcd}
\]
where \( X=\partial_x -y\partial_z\), \( Y =\partial_y\), and \( Z=\partial_z \) are the standard basis elements of the Heisenberg Lie algebra with the commutator identity \([X,Y]=Z\), here acting as vector fields on $M$. Here $\theta=y\rd x+\rd z$ denotes the contact form. We equip $M$ with the volume density induced from the Haar measure on $\sfH_3/\Gamma$ and declare $\rd x$, $\rd y$ and $\theta$ to be an orthonormal frame. With these choices, the Rumin complex above is completed into a Hilbert complex, see \cite{leschbruening} or Section \ref{sec:complexes} below.

We shall shorten the notation for the operators in the Rumin complex to $\rd_\bullet^R=(\rd_0^R,\rd_1^R,\rd_2^R)$. It is a mixed order differential complex. Let
$$D_1 := \rd^R_0+(\rd^R_0)^* + \rd^R_2+(\rd^R_2)^* \quad \mbox{and}\quad D_2:= \rd^R_1+(\rd^R_1)^* . $$
We view $D_1$ and $D_2$ as densely defined, self-adjoint operators on $L^2(M;\mathcal{H})$ where $\mathcal{H}\to M$ is the sum of all line bundles appearing in the Rumin complex; so $\mathcal{H}\cong M\times \C^6$. The differential operators $D_1$ and $D_2$ are of order $m_1=1$ and $m_2=2$ respectively. We note that $D_1D_2=D_2D_1=0$, so $D_1$ and $D_2$ are strictly anticommuting. The Rumin Laplacian takes the form
$$\Delta^R=D_1^4+D_2^2.$$
The data $(C^\infty(M),L^2(M;\mathcal{H}),(D_1,D_2))$ constitute an ST\textsuperscript{2} with bounding matrix
\[ \VEC{\epsilon}=\begin{pmatrix} 0& 0\\1& \frac{1}{2}\end{pmatrix} \qquad\qquad \vcenter{\hbox{\begin{tikzpicture}
\graph[grow right sep=1.5cm, empty nodes, nodes={fill=black, circle, inner sep=1.5pt}, edges={semithick}]{
    a --["$1$" inner sep=5.5pt, middlearrow={<}] b --["$\frac{1}{2}$"', loop] b;};
\end{tikzpicture}}} . \]
If one squints, the bounding matrix can be guessed from the structure of the Rumin complex. The diagonal arrows, corresponding to the operators \( -X^2 \) and \( Y^2 \), require the weight-\( 1/2 \) loop and the horizontal arrows, corresponding to the operators \( Z + X Y \) and \( Z - Y X \), require the weight-\( 1 \) edge from \( D_2 \) to \( D_1 \). The other arrows, forming part of \( D_1 \) are all first order, making no contribution to the bounding matrix.
Of course, this is not a rigorous argument; for that we will have to wait until Subsection \ref{subsec:naive}.

Below, in Subsection \ref{subsec:naive}, we will show that the naïvely formed candidate $D_1+D_2$ for a noncommutative geometry on $M$ fails to be a higher order spectral triple, motivating the need for ST\textsuperscript{2}s.
But, for any $\VEC{t}=(t_1,t_2)\in (0,\infty)^2$ with 
$$t_1>t_2,$$
we arrive at a higher order spectral triple with Dirac operator 
$$\overline{D}_{\VEC{t}}=D_1|D_1|^{t_1-1}+D_2|D_2|^{t_2-1}=D_1(\Delta^R)^{\frac{t_1-1}{4}}+D_2(\Delta^R)^{\frac{t_2-1}{2}}.$$
If $\VEC{t}$ lies along the ray spanned by $(1,1/2)$ then $\overline{D}_{\VEC{t}}$ is an $H$-elliptic operator in the Heisenberg calculus and if $\VEC{t}=(2k_1+1,2k_2+1)$ where $k_1>k_2$ are natural numbers then $\overline{D}_{\VEC{t}}$ is a differential operator. The reader can find more details in Section \ref{sec:rumin} below.

\subsubsection{A strictly tangled spectral triple on the Heisenberg group}
\label{subsubheis}
A natural $K$-homology class on the $C^*$-algebra of a connected Lie group with no compact subgroups arises from Baaj--Skandalis duality $KK^G(\C,C_0(G))\cong KK^{\hat{G}}(C^*(G),\C)$ and the dual Dirac element in $KK^G(\C,C_0(G))$. Below in Section \ref{section:nilpotent} we discuss this construction for closed subgroups of nilpotent Lie groups but here we present a miniature version of this example on the $3$-dimensional real Heisenberg group $\sfH_3$. In the $3\times 3$-matrix presentation, we can write
\[\sfH_3=\left\{g\in M_3(\R): g = \begin{pmatrix} 1 & a & c \\ 0& 1 & b \\ 0&0 & 1 \end{pmatrix}\right\} . \]
The group $\sfH_3$ is a central extension of $\R^2$ by $\R$, fitting into an exact sequence
$$\begin{tikzcd}
0 \arrow[r] & \R \arrow[r, "\iota"] & \sfH_3 \arrow[r, "q"] \arrow[l, "r", dashed, bend left] & \R^2 \arrow[r] & 0,
\end{tikzcd}$$
where $r$ is the (topological) retraction \( r : \sfH_3 \to \R \) given by \( r(g) = c \).
	
\begin{definition}
\label{definition:w465ybe57unu467inr4nin6i}
\cite[Definition II.2.2]{AMthesis} cf. \cite[Definitions 2.8,10]{rubinetal}
A \emph{weight} on a locally compact group \( G \) is a continuous function $\ell$ from \( G \) to matrices on a finite-dimensional complex vector space \( V \). If \( V \) is \( \Z/2\Z \)-graded, \( \ell \) is required to be odd. We say that \( \ell \) is
\begin{itemize}
	\item self-adjoint if \( \ell^* = \ell \);
	\item \emph{proper} if \( (1 + \ell^* \ell)^{-1} \in C_0(G, \End V) = C_0(G) \otimes \End V \); and
	\item \emph{translation-bounded} if, for all \( g \in G \),
	\( \sup_{h \in G} \| \ell(g h) - \ell(h) \| < \infty \)
	and there exists a neighbourhood \( U \) of the identity in \( G \) such that
	\( \sup_{g \in U, h \in G} \| \ell(g h) - \ell(h) \| < \infty \).
\end{itemize}
\end{definition}

By \cite[Theorem II.2.24]{AMthesis}, an operator valued length function gives rise to a spectral triple
\[ (C_c(G), L^2(G, E), M_{\ell}) \]
for the group $C^*$-algebra of \( G \). Similar examples have been studied several times in the literature, e.g.~in \cite{connesmetric, GRU}, although often with \( \ell \) valued in \( [0, \infty) \) and so giving rise to trivial $K$-homology.

For notational clarity, write $Z:=\R$ for the central subgroup and use $q$ to identify $\sfH_3/Z= \R^2$. We define the weights
\begin{align*}
\ell_{Z} : Z & \to \C & \ell_{\sfH_3/Z} : \sfH_3/Z & \to \Cl^2 \\
c & \mapsto c & (a, b) & \mapsto a \gamma_1 + b \gamma_2 .
\end{align*}
We can form the unbounded Kasparov $C^*(\sfH_3)$-$C^*(Z)$-module
\[ (C_c(\sfH_3), \overline{(C_c(\sfH_3) \otimes \C^2)_{C^*(Z)}}, M_{q^*(\ell_{\sfH_3/Z})}) \]
representing a generator of \( KK_0(C^*(\sfH_3), C^*(Z)) \). Here $\overline{(C_c(\sfH_3) \otimes \C^2)_{C^*(Z)}}$ denotes Hilbert $C^*(Z)$-module obtained as the completion of $C_c(\sfH_3) \otimes \C^2$ in the $C_c(Z)$-valued inner product $\langle f_1,f_2\rangle:=(f_1^**f_2)|_Z$ where the convolution is over $\sfH_3$. We also form the unbounded Kasparov $C^*(Z)$-$\C$-module
\[ (C_c(Z), L^2(Z), M_{\ell_{Z}}) \]
representing a generator of \( KK_1(C^*(Z), \C) \). In fact, up to a Fourier transform, this is the standard Dirac operator on $Z = \R$. Were we naïvely to form the constructive unbounded Kasparov product of $(C_c(\sfH_3), \overline{(C_c(\sfH_3) \otimes \C^2)_{C^*(Z)}}, M_{q^*(\ell_{\sfH_3/Z})})$ and $(C_c(Z), L^2(Z), M_{\ell_{Z}})$, it would produce the triple
$$(C_c(\sfH_3), L^2(\sfH_3, \C^2), M_{\ell})$$
where the operator valued length function $\ell$ would be given by
\[ \ell(g) = q^*(\ell_{\sfH_3/Z})(g) + r^*(\ell_{z})(g) \gamma_3 = a \gamma_1 + b \gamma_2 + c \gamma_3 . \]
The choice of retraction \( r \) is related to the choice of a connection in the constructive Kasparov product. Alas,
\[ \ell(g h) - \ell(h) = a \gamma_1 + b \gamma_2 + (c + a b') \gamma_3 \]
is not bounded in \( h \). Instead, if we let \( \ell_1 = q^*(\ell_{\sfH_3/Z}) \) and \( \ell_2 = r^*(\ell_{Z}) \gamma_3 \), we see that
\[ \sup_h |\ell_1(g h) - \ell_1(h)| = |\ell_1(g)| < \infty \]
but $\ell_2$ exhibits the parabolic feature that
\[ |\ell_2(g h) - \ell_2(h)| | 1 + \ell_1(h) |^{-1} = |c + a b'| (1 + ({a'}^2 + {b'}^2)^{1/2})^{-1} \leq |\ell_2(g)| + |\ell_1(g)| \]
so that
\[ \sup_h |\ell_2(g h) - \ell_2(h)| | 1 + \ell_1(h) |^{-1} < \infty . \]
We therefore have a strictly tangled spectral triple
\[ \left( C^*(\sfH_3), L^2(\sfH_3, \C^2), (M_{\ell_1}, M_{\ell_2}) \right) \]
with bounding matrix
\[ \VEC{\epsilon}=\begin{pmatrix}0&0\\1&0\end{pmatrix} \qquad \vcenter{\hbox{\begin{tikzpicture}
\graph[grow right sep=1.5cm, empty nodes, nodes={fill=black, circle, inner sep=1.5pt}, edges={semithick}]{
a --["$1$" inner sep=5.5pt, middlearrow={<}] b;
};
\end{tikzpicture}}} . \]
The set $\Omega(\VEC{\epsilon})$ consists of $t_1, t_2 \in (0,\infty)$ such that $t_1>t_2$.

We return to this example below, in two instances. In Example \ref{kaspkuc1}, we interpret the higher order spectral triple associated with $( C^*(\sfH_3), L^2(\sfH_3, \C^2), (M_{\ell_1}, M_{\ell_2}))$ as a Kasparov product using Kucerovsky's theorem. In Section \ref{section:nilpotent}, this example will be generalized to all simply connected nilpotent Lie groups (and their closed subgroups) where there are as many Dirac operators as the step length in the group.

\subsubsection{Parabolic large diffeomorphisms on the torus}

In Section \ref{section:parabolic}, we will see that ST\textsuperscript{2}s allow us to generalise the construction of spectral triples for elliptic dynamical systems by Bellissard, Marcolli, and Reihani \cite{Bellissard_2010} to parabolic dynamical systems, including nilflows, horocycle flows, and large diffeomorphisms of tori, classical and noncommutative. We here give a simple instance of this latter family of examples.

Consider the group of diffeomorphisms \( (\phi_n)_{n \in \Z} \) of \( \T^2 \) given by
\[ \phi_n = \begin{pmatrix} 1 & n \\ 0& 1 \end{pmatrix} \in  SL(2, \Z) . \]
This family of diffeomorphisms is \emph{large}, in the sense that each \( \phi_n \) is in a distinct connected component of the diffeomorphism group of \( \T^2 \).
It induces a \( \Z \)-action \( \alpha \) on \( C(\T^2) \) given by $\alpha_n(a):=\phi_{-n}^*(a)$ for $a\in C(\T^2)$, preserving \( C^{\infty}(\T^2) \).
Let \( (C^\infty(\T^2), L^2(\T^2, \C^2), D) \) be the Dirac spectral triple on the torus. With \( N \) the number operator on \( \ell^2(\Z) \), we write \( (C^\infty(\T^2) \rtimes \Z, \ell^2(\Z) \otimes C(\T^2)_{C(\T^2)}, N \otimes 1) \) for the unbounded Kasparov module associated with the crossed product.

In attempting to form the Kasparov product, we encounter the pointwise-boundedness condition of \cite[\S 1]{Paterson_2014}, reproduced (and generalised) in Definition \ref{definition:pointwisebounded} below. For \( a \in  C^\infty(\T^2) \rtimes \Z\) , we require uniform boundedness of $\| [D, \alpha_n(a)] \|$ in $n$. Let us see how $\| [D, \alpha_n(a)] \|$ behaves as $|n|\to \infty$. For \( a \in  C^\infty(\T^2) \) 
\[ \alpha_n(a)(x, y) = a(x - n y, y), \]
and 
\[ D= \gamma_1 \partial_x + \gamma_2\partial_y, \]
so
\[ [D, \alpha_n(a)] = \gamma_1 \phi_{-n}^*(\partial_x a) + \gamma_2 \left( \phi_{-n}^*(\partial_y a) - n \phi_{-n}^*(\partial_x a) \right). \]
We conclude that there is a constant $C>0$ such that, for any \( a \in  C^\infty(\T^2) \) and $n\in \Z$,
\[  |n| \| \partial_x a \|_{L^\infty}-C\|\nabla a\|_{L^\infty}\leq \| [D, \alpha_n(a)] \| \leq |n| \| \partial_x a \|_{L^\infty}+C\|\nabla a\|_{L^\infty} . \]
We see that the pointwise-boundedness condition is not satisfied, rather we have the growth behaviour $ \| [D, \alpha_n(a)] \|\sim |n| \| \partial_x a \|_{L^\infty}$ as $|n|\to \infty$. Hence
\[ [1 \otimes D, \pi(a)] (1 + |N|)^{-1} \otimes 1 \]
is bounded. In particular, the collection
\[ (C_0(\T^2) \rtimes \Z, \ell^2(\Z) \otimes L^2(\T^2, S), (N \otimes  \gamma, 1 \otimes D )) \]
is a strictly tangled spectral triple with bounding matrix
\[ \VEC{\epsilon}=\begin{pmatrix}0&0\\1&0\end{pmatrix} \qquad \vcenter{\hbox{\begin{tikzpicture}
\graph[grow right sep=1.5cm, empty nodes, nodes={fill=black, circle, inner sep=1.5pt}, edges={semithick}]{
a --["$1$" inner sep=5.5pt, middlearrow={<}] b;
};
\end{tikzpicture}}} \]
and $\Omega(\VEC{\epsilon})=\{(t_1,t_2)\in (0,\infty)^2: t_1>t_2\}$.

\section{Analysis of strictly tangled noncommutative geometries}
\label{secljnakjnad}

We here analyse strictly tangled spectral triples in terms of higher order spectral triples and introduce further structure thereon, with the aim of studying ST\textsuperscript{2}s in their own right.

\subsection{Assembling an ST\textsuperscript{2} into a higher order spectral triple}

Let us study how to construct a higher order spectral triple (HOST) from an ST\textsuperscript{2}. In conjunction with Theorem \ref{thm:host_summability}, we will see that there is a well-defined $K$-homology class associated with an ST\textsuperscript{2}.

\begin{definition}
	\label{definition:rho-preserving}	
	An ST\textsuperscript{2} $(\mathcal{A},\mathpzc{H},\DC)$ with bounding matrix \( \VEC{\epsilon} \) is \emph{\( \VEC{\rho} \)-preserving} for \( \VEC{\rho} \in [1, \infty]^I \) if, for all \( a \in \mathcal{A} \) and \( i \in I \), \( a \Dom |D_i|^{\rho_i} \subseteq \Dom |D_i|^{\rho_i} \). If \( \rho_i = \infty \), the condition should be interpreted as requiring that \( a \Dom |D_i|^t \subseteq \Dom |D_i|^t \) for all \( t \geq 1 \).
\end{definition}

Every ST\textsuperscript{2} is by definition \( \VEC{\rho} \)-preserving for \( \VEC{\rho} = (1, \ldots, 1) \) and, if an ST\textsuperscript{2} is \( \VEC{\rho} \)-preserving, it is \( \VEC{\sigma} \)-preserving for all \( \VEC{\sigma} \leq \VEC{\rho} \) by \cite[Theorem 12.5]{Krasnoselskii_1976}. Recall that
$$\Omega(\VEC{\epsilon})=\{ \VEC{t} = (t_j)\in (0,\infty)^I: \epsilon_{ij}t_i<t_j\; \forall i,j\},$$
For \( \VEC{\rho} \in [1, \infty]^I \), we will define the subset
\[ \Omega(\VEC{\epsilon}, \VEC{\rho}) = \Omega(\VEC{\epsilon}) \cap \prod_{j \in I} (0, 1] \cup (1, \rho_j) . \]
Here, the interval \( (0, 1] \cup (1, \rho_j) = (0, \rho_j) \cup \{1\} \) is simply the half-open interval \( (0, 1] \) if \( \rho_j = 1 \) and the open interval \( (0, \rho_j) \) if \( \rho_j > 1 \). We remark that \( \Omega(\VEC{\epsilon}, \VEC{\rho}) \) is a convex set and also that \( \Omega(\VEC{\epsilon}, (\infty, \ldots, \infty))  = \Omega(\VEC{\epsilon}) \).

\begin{thm}
\label{thm:st2_give_host}
Let $(\mathcal{A},\mathpzc{H}, \DC=(D_j)_{j\in I})$ be an ST\textsuperscript{2} with bounding matrix \( \VEC{\epsilon} \) and \( \VEC{\rho} \)-preserving. For $\VEC{t}\in (0, \infty)^I $, we define the operator 
$$\overline{D}_{\VEC{t}} = \sum_{j\in I} \sign(D_j) |D_j|^{t_j}.$$ 
If $\VEC{t}\in \Omega(\VEC{\epsilon}, \VEC{\rho}) $, then the triple $(\mathcal{A},\mathpzc{H},\overline{D}_{\VEC{t}})$ defines an order-\( m \) spectral triple for any
\begin{equation}
	\label{eq:st2_give_host_order_req}
	m > \max_{i, j \in I} \ \max\Big\{ 1, \frac{\rho_i - 1}{\rho_i - t_i} t_i \Big\} \Big(1-\frac{\epsilon_{ij}t_i}{t_j} \Big)^{-1} .
\end{equation}
(If \( \rho_i = \infty \) for some \( i \in I \), we interpret \( \frac{\rho_i - 1}{\rho_i - t_i} \) as \( 1 \). If \( \rho_i = t_i = 1 \) for some \( i \in I \), we also interpret \( \frac{\rho_i - 1}{\rho_i - t_i} \) as \( 1 \).)
\end{thm}

We remark that it is impossible for \( 1 \neq \rho_i = t_i \). To prove Theorem~\ref{thm:st2_give_host}, we use results from Appendix \ref{section:nearly-convex}.

\begin{proof}
	Let \( \VEC{t} \in \Omega(\VEC{\epsilon}, \VEC{\rho}) \). The local compactness of the resolvent follows immediately from Lemma \ref{lknlknn}. We now proceed to show that, for all $i \in I$ and $a\in \mathcal{A}$, 
\begin{equation}
	\label{eq:st2_give_host_thing_to_prove_bounded}
	[\sign(D_i) |D_i|^{t_i},a] \Biggl(1+\sum_{j\in I} |D_j|^{t_j} \Biggr)^{-1 + \frac{1}{m}}
\end{equation}
is bounded. If \(t_i=1\), again using Lemma \ref{lknlknn} we see that \eqref{eq:st2_give_host_thing_to_prove_bounded} is bounded if \( (1 - 1/m) t_j \geq \epsilon_{ij}t_i\) for all \(j\), which is equivalent to \(m \geq (1 - \epsilon_{ij} t_i/t_j)^{-1}\). In the context of Theorem \ref{theorem:nearly-convex}, let
\[ A = D_i \qquad B = 1+\sum_{j\in I} |D_j|^{t_j} \]
and
\[ \alpha_1 = 1 \qquad \beta_1 = \max_{j \in I} \frac{\epsilon_{i j}}{t_j} \qquad \alpha_2 = t_i \qquad \beta_2 = 1 - \frac{1}{m} . \]
We see that \eqref{eq:st2_give_host_thing_to_prove_bounded} is bounded if \( 1 - 1/m > \max_{j \in I} \epsilon_{ij}t_i /t_j \), equivalent to \(m > \max_{j \in I} (1 - \epsilon_{ij} t_i/t_j)^{-1}\). If \( 1 < \rho_i < \infty \) and \(t_i \in (1, \rho_i) \), still in the context of Theorem \ref{theorem:nearly-convex}, let \( \alpha_3 = \rho_i \) and \( \beta_3 = \rho_i/t_i \). We see that \eqref{eq:st2_give_host_thing_to_prove_bounded} is bounded if 
\[ 1 - \frac{1}{m} > \frac{1}{\rho_i - 1} \Big( (\rho_i - t_i) \max_{j \in I} \frac{\epsilon_{i j}}{t_j} + (t_i - 1) \frac{\rho_i}{t_i} \Big) , \]
equivalent to
\[ m  > \max_{j \in I} \frac{\rho_i - 1}{\rho_i - t_i} t_i \Big(1-\frac{\epsilon_{ij}t_i}{t_j} \Big)^{-1} . \]
If \( \rho_i = \infty \) and \(t_i \in (1, \rho_i) \), we see by taking the limit that \eqref{eq:st2_give_host_thing_to_prove_bounded} is bounded if 
\[ m  > t_i \Big(1- \max_{j \in I} \frac{\epsilon_{ij}t_i}{t_j} \Big)^{-1} . \]
Noting that \( \frac{\rho_i - 1}{\rho_i - t_i} t_i > 1 \) if and only if \( \rho_i, t_i > 1 \), we thus obtain the claimed order estimate.
\end{proof}

\begin{remark}
\label{relboundgen}
Theorem \ref{thm:st2_give_host} is proven under strong assumptions on the anticommutators $D_jD_k+D_kD_j$, namely that they vanish for $j\neq k$. We expect that Theorem \ref{thm:st2_give_host} holds under much milder assumptions on the anticommutators $D_jD_k+D_kD_j$. In the proof of Theorem \ref{thm:st2_give_host}, we rely heavily on Theorem \ref{theorem:nearly-convex}
for \( A = D_i \) and $B=\Delta_{\VEC{t}/2}^\DC$. Assumptions such as those in \cite{leschkaad1,kaadlesch,bramlesch}, modified according to an $\VEC{\epsilon}$-power of $\DC$, may allow one to extend Theorem \ref{thm:st2_give_host}.

Let us discuss a prototypical example to which Theorem \ref{thm:st2_give_host} extends, despite a lack of vanishing anticommutators. In \cite[§1.1–2]{Connes_1995}, an order-2 spectral triple
\begin{equation}
\label{cmhost} 
\left( C_c^{\infty}(M) \rtimes \Gamma, L^2(M, \Lambda^* V^* \otimes \Lambda^* N^*), (\rd_L \rd_L^* - \rd_L^* \rd_L) (-1)^{\partial_N} + \rd_H + \rd_H^* \right)
\end{equation}
is built from the data of a manifold \( M \) with triangular structure preserved by a group of diffeomorphisms \( \Gamma \). To arrive at this higher order spectral triple, the longitudinal signature operator \( \rd_L + \rd_L^* \) is first found to be homotopic to \( \Delta_L^{-1/2} (\rd_L \rd_L^* - \rd_L^* \rd_L) \). At this point, we can consider the collection
\begin{equation}
\label{cmst} 
\left(C_c^{\infty}(M) \rtimes \Gamma, L^2(M, \Lambda^* V^* \otimes \Lambda^* N^*), \left( \Delta_L^{-1/2} (\rd_L \rd_L^* - \rd_L^* \rd_L) (-1)^{\partial_N}, \rd_H + \rd_H^* \right) \right). 
\end{equation}
The operators in the collection \eqref{cmst} are not strictly anticommmuting but the anticommutators are of lower order in the pseudodifferential calculus of \cite{Connes_1995}. The pseudodifferential calculus allow us to think of \eqref{cmst} as a ``tangled spectral triple'' with bounding matrix 
\[ \VEC{\epsilon}=\begin{pmatrix} 0& 0\\1& \frac{1}{2}\end{pmatrix} \qquad\qquad \vcenter{\hbox{\begin{tikzpicture}
\graph[grow right sep=1.5cm, empty nodes, nodes={fill=black, circle, inner sep=1.5pt}, edges={semithick}]{
    a --["$1$" inner sep=5.5pt, middlearrow={<}] b --["$\frac{1}{2}$"', loop] b;};
\end{tikzpicture}}} \]
so that taking \( \VEC{t} = (2, 1) \) produces the order-2 spectral triple \eqref{cmhost}.
\end{remark}

\begin{remark}
	Let us consider the consequences of Theorem \ref{thm:st2_give_host} in the special case when the collection \( \VEC{D} \) has only one element. Let \( (\mathcal{A}, \mathpzc{H}, D) \) be an order-\( m \) spectral triple which is \( \rho \)-preserving for \( \rho \in [1, \infty] \). For \( t \in (0, 1] \cup (1, \rho) \), \( (\mathcal{A}, \mathpzc{H}, \sign(D) |D|^t) \) is an order-\( m' \) cycle for
	\[ m' > m \, \max\Big\{ 1, \frac{\rho - 1}{\rho - t} t \Big\} \]
	If \( \rho = \infty \), this means \( m' > m \, \max\{ 1, t\} \).
\end{remark}

In examples, it is frequently the case that the requirement in \eqref{eq:st2_give_host_order_req} may be taken as an equality.

\begin{remark}
	\label{remark:order_of_st2_with_prescribed_orders}
	For an ST\textsuperscript{2} \((\mathcal{A},\mathpzc{H},\DC)\) such that the operators \(\DC\) have prescribed orders \(\VEC{m}\in [1,\infty)^I\) with bounding matrix on the form \(\epsilon_{ij}=\frac{m_i - 1}{m_j}\), as in Remark \ref{remark:operators_with_prescribed_orders}, then for any \(\tau>0\) we would like the spectral triple \((\mathcal{A},\mathpzc{H},\overline{D}_\tau)\) to be of order \(\tau\). Abstractly, Theorem \ref{thm:st2_give_host} guarantees the order to be at most \( \tau + \delta \) for any \( \delta > 0 \). In the examples below coming from Rockland complexes (see Corollary \ref{thm:rockland_H_ellipitc_host}), the pseudodifferential calculus ensures that the order can be taken to be \(\tau\) on the nose.
\end{remark}

In light of Theorem \ref{thm:host_summability}, Theorem \ref{thm:st2_give_host} implies the following.

\begin{cor}
\label{khomfromadk}
Let $(\mathcal{A},\mathpzc{H},\DC)$ be an ST\textsuperscript{2} with the bounding matrix $\VEC{\epsilon}$. We write $A$ for the $C^*$-algebra closure of $\mathcal{A}$. There is a well defined $K$-homology class
$$[(\mathcal{A},\mathpzc{H},\DC)]:=[(\mathcal{A},\mathpzc{H},F_{\overline{D}_{\VEC{t}}})]\in K^*(A)$$
for any $\VEC{t}\in \Omega(\VEC{\epsilon})\cap (0,1]^I$ (or \(\VEC{t}\in \Omega(\VEC{\epsilon})\) if the ST\textsuperscript{2} is regular) with the same parity as $(\mathcal{A},\mathpzc{H},\DC)$. The class $[(\mathcal{A},\mathpzc{H},\DC)]$ depends only on $(\mathcal{A},\mathpzc{H},\DC)$ and not on $\VEC{t}$.
\end{cor}
\begin{proof}
	That a class in $K$-homology is obtained for any $\VEC{t}\in \Omega(\VEC{\epsilon}, \VEC{\rho}) $  follows immediately from the bounded transform for HOSTs, Theorem \ref{thm:host_summability}. 
	Consider distinct  $\VEC{s}, \VEC{t}\in \Omega(\VEC{\epsilon}, \VEC{\rho}) $.  Since \( \Omega(\VEC{\epsilon}, \VEC{\rho}) \) is a convex set, \( x \VEC{s} + (1 - x) \VEC{t} \in \Omega(\VEC{\epsilon}, \VEC{\rho}) \) for all \( x \in [0, 1] \). That \( (A, \mathpzc{H},F_{\overline{D}_{\VEC{s}}}) \) and \( (A, \mathpzc{H},F_{\overline{D}_{\VEC{t}}}) \) are equivalent can then be shown by taking the straight line homotopy.
\end{proof}

We note that, in Corollary \ref{khomfromadk}, the fact that we retain the sign of each \( D_j \) in the combined operator
$\overline{D}_{\VEC{t}}= \sum_{j\in I} \sign(D_j) |D_j|^{t_j} $
ensures that the $K$-homology classes can be non-trivial, even when a component of $\VEC{t}$ is an even integer.

\begin{example}
\label{kaspkuc1}
Let us return to the ST\textsuperscript{2} for the Heisenberg group in Subsubsection \ref{subsubheis}. Using Kucerovsky's theorem \cite{danthedan} (and in particular its extension to HOSTs in \cite[Theorem A.7]{GMCK}), we see that, for any $\VEC{t}\in \Omega(\VEC{\epsilon})=\{(t_1,t_2)\in (0,\infty)^2: t_1>t_2\}$ and for $\DC=(M_{\ell_1}, M_{\ell_2})$, the HOST $(C_c(\sfH_3), L^2(\sfH_3,\C^2),\overline{D}_{\VEC{t}})$ represents the class in $K^1(C^*(\sfH_3))=KK_1(C^*(\sfH_3), \C)$ of the Kasparov product of 
$$[(C_c(\sfH_3), \overline{(C_c(\sfH_3) \otimes \C^2)_{C^*(Z)}}, M_{q^*(\ell_{\sfH_3/Z})})]\in KK_0(C^*(\sfH_3), C^*(Z)),$$ 
with 
$$[(C_c(Z), L^2(Z), M_{\ell_{Z}})]\in KK_1(C^*(Z), \C).$$
The take home message from this example is that ST\textsuperscript{2}s can be used to represent bad Kasparov products by encoding the directional properties separately.
\end{example}

\begin{thm}
\label{thm:extprod}
Let $(\mathcal{A}_1,\mathpzc{H}_1, \VEC{D_1})$ and $(\mathcal{A}_2,\mathpzc{H}_2,\DC_2)$ be two even ST\textsuperscript{2}s with bounding matrices $\VEC{\epsilon_1}$ and $\VEC{\epsilon_2}$ respectively.
Here we write $\VEC{D_1}=(D_{1, j})_{j\in I_1}$ and $\VEC{D_2}=(D_{2, k})_{k\in I_2}$.
Denote by $\mathpzc{H}_1\tildeotimes \mathpzc{H}_2$ the graded tensor product.
Then, with the collection $ \VEC{D_1} \tildeotimes 1 \, \sqcup \, 1 \tildeotimes \VEC{D_2} =(\hat{D}_l)_{l\in I_1 \sqcup I_2}$, given by
\[ \hat{D}_l := \begin{cases} D_{1, l} \tildeotimes 1 & l \in I_1 \\ 1\tildeotimes D_{2, l} & l \in I_2 \end{cases} , \]
the data
$$(\mathcal{A}_1\otimes \mathcal{A}_2, \mathpzc{H}_1\tildeotimes \mathpzc{H}_2, \VEC{D_1} \tildeotimes 1 \, \sqcup \, 1 \tildeotimes \VEC{D_2})$$
constitute a strictly tangled cycle with bounding matrix the direct sum $\VEC{\epsilon_1} \oplus \VEC{\epsilon_2}$.
Moreover, the exterior Kasparov product of the associated $K$-homology classes can be written as 
$$[(\mathcal{A}_1,\mathpzc{H}_1,\VEC{D_1})]\otimes [(\mathcal{A}_2,\mathpzc{H}_2,\VEC{D_2})]:=[(\mathcal{A}_1\otimes \mathcal{A}_2,\mathpzc{H}_1\tildeotimes \mathpzc{H}_2,\VEC{D_1} \tildeotimes 1 \, \sqcup \, 1 \tildeotimes \VEC{D_2})]\in K^*(A_1\otimes A_2) , $$
where $A_1$ and $A_2$ are the $C^*$-algebra closures of $\mathcal{A}_1$ and $\mathcal{A}_2$ respectively.
\end{thm}

For the sake of brevity, we have stated Theorem \ref{thm:extprod} only for even ST\textsuperscript{2}s but an analogous result holds for all parities. 

\begin{proof}
It is straightforward to verify that $(\mathcal{A}_1\otimes \mathcal{A}_2,\mathpzc{H}_1\tildeotimes \mathpzc{H}_2, \VEC{D_1} \tildeotimes 1 \, \sqcup \, 1 \tildeotimes \VEC{D_2})$ is an ST\textsuperscript{2} with bounding matrix $\VEC{\epsilon_1} \oplus \VEC{\epsilon_2}$. It is also clear that $\Omega( \VEC{\epsilon_1} \oplus \VEC{\epsilon_2})=\Omega( \VEC{\epsilon_1})\times\Omega( \VEC{\epsilon_2})$. For $\VEC{t}=(\VEC{t_1},\VEC{t_2})$ we have that 
$$\overline{\hat{D}}_{\VEC{t}}=\overline{D}_{\VEC{t_1}}\tildeotimes 1+1\tildeotimes \overline{D}_{\VEC{t_2}} , $$
which is the form of the product operator for the external product of (higher order) spectral triples \cite[Theorem A.7]{GMCK}. Hence, any higher order spectral triple assembled from 
$(\mathcal{A}_1\otimes \mathcal{A}_2,\mathpzc{H}_1\tildeotimes \mathpzc{H}_2,\VEC{D_1} \tildeotimes 1 \, \sqcup \, 1 \tildeotimes \VEC{D_2})$ represents the exterior Kasparov product of the higher order spectral triples assembled from $(\mathcal{A}_1,\mathpzc{H}_1,\VEC{D_1})$ and $(\mathcal{A}_2,\mathpzc{H}_2,\VEC{D_2})$. The theorem follows.
\end{proof}

In a simpler way, we obtain

\begin{thm}
\label{thm:directsum}
	Let $(\mathcal{A},\mathpzc{H}_1, \VEC{D_1})$ and $(\mathcal{A},\mathpzc{H}_2, \VEC{D_2})$ be two ST\textsuperscript{2}s (of the same parity) with bounding matrices $\VEC{\epsilon_1}$ and $\VEC{\epsilon_2}$ respectively.
	Here we write $\VEC{D_1}=(D_{1, j})_{j\in I_1}$ and $\VEC{D_2}=(D_{2, k})_{k\in I_2}$. Then, with the collection $ \VEC{D_1} \oplus \VEC{D_2} =(\hat{D}_l)_{l\in I_1 \sqcup I_2}$, given by
	\[ \hat{D}_l := \begin{cases} D_{1, l} \oplus 0 & l \in I_1 \\ 0 \oplus D_{2, l} & l \in I_2 \end{cases} , \]
	the data
	$$(\mathcal{A}, \mathpzc{H}_1 \oplus \mathpzc{H}_2, \VEC{D_1} \oplus \VEC{D_2})$$
	constitute an ST\textsuperscript{2} with bounding matrix the direct sum $\VEC{\epsilon_1} \oplus \VEC{\epsilon_2}$.
	Moreover, the direct sum of the associated $K$-homology classes can be written as 
	$$[(\mathcal{A},\mathpzc{H}_1,\VEC{D_1})] \oplus [(\mathcal{A}_2,\mathpzc{H}_2,\VEC{D_2})]:=[(\mathcal{A}, \mathpzc{H}_1 \oplus \mathpzc{H}_2, \VEC{D_1} \oplus \VEC{D_2})]\in K^*(A) ,$$
	where $A$ is the $C^*$-algebra closure of $\mathcal{A}$.
\end{thm}

It unclear whether it is an advantage or a disadvantage of the framework of ST\textsuperscript{2}s that products and sums are treated in the same way, in the sense that they have the same effect on the bounding matrix. The difference is in the support of the operators: for an external product, every operator is supported on the entire Hilbert module whereas, for the direct sum, the operators have disjoint support. In this respect, when we come to consider complexes, we will see they behave more like sums than products; on the other hand, examples coming from the constructive unbounded Kasparov product will behave more like products than sums.

\subsection{Finite summability of strictly tangled spectral triples}

The natural notion of dimension in noncommutative geometry is determined from spectral properties in analogy with the Weyl law. We introduce a notion of summability of an ST\textsuperscript{2} that takes into account the different directions by means of a function. To simplify the description, we restrict our discussion of summability to the Schatten ideals with exponent $p>0$.

\begin{definition}
\label{deffsum}
Assume that $f:(0,\infty)^n\to (0,\infty)$ is a function decreasing in each argument. An ST\textsuperscript{2} $(\mathcal{A},\mathpzc{H},\DC)$, with \( \mathcal{A} \) unital, is said to be \emph{$f$-summable} if, for $\VEC{t}=(t_1,\ldots, t_n)\in (0,\infty)^n$, the domain inclusion 
$$\cap_j \Dom(|D_j|^{t_j})\hookrightarrow \mathpzc{H}$$
belongs to the Schatten class $\mathcal{L}^{f(\VEC{t})}(\cap_j \Dom(|D_j|^{t_j}),\mathpzc{H})$, where the left hand side is given the Hilbert space topology from the intersection of graph topologies.
\end{definition}

\begin{example}
The notion of $f$-summability is for $n=1$ compatible with the notion of summability for spectral triples or, more generally, higher order spectral triples as in Definition \ref{defhost}. Indeed, if $(\mathcal{A},\mathpzc{H},D)$ is a $p$-summable higher order spectral triple then it is an $f$-summable ST\textsuperscript{2} with $n=1$ for $f(t)=p/t$. Below in Section \ref{sec:complexes}, we consider ST\textsuperscript{2}s arising from Hilbert complexes defined from mixed order operators in which case the function $f$ plays a role of controlling different orders of summability in the different directions.
\end{example}

\begin{example}
Let us return to the exterior Kasparov product of Theorem \ref{thm:extprod}. Assume that $(\mathcal{A}_1,\mathpzc{H}_1,D_1)$ and $(\mathcal{A}_2,\mathpzc{H}_2,D_2)$ are two even higher order spectral triples that are summable of order $p_1$ and $p_2$ respectively. Their external Kasparov product is represented by the ST\textsuperscript{2} $(\mathcal{A}_1\otimes\mathcal{A}_2,\mathpzc{H}_1\tildeotimes \mathpzc{H}_2,(D_1\tildeotimes 1,1 \tildeotimes D_2))$. The ST\textsuperscript{2} $(\mathcal{A}_1\otimes\mathcal{A}_2,\mathpzc{H}_1\tildeotimes \mathpzc{H}_2,(D_1\tildeotimes 1,1 \tildeotimes D_2))$ will then be $f$-summable for any $f:(0,\infty)^2\to (0,\infty)$ such that 
$$(1+|D_1|^{t_1}\tildeotimes 1+1 \tildeotimes |D_2|^{t_2})^{-1}\in \mathcal{L}^{f(t_1,t_2)}(\mathpzc{H}_1\tildeotimes \mathpzc{H}_2).$$
For instance, we could take 
$$f(t_1,t_2):=\frac{p_1}{t_1}+\frac{p_2}{t_2}.$$
\end{example}

\begin{example}
We return to the direct sum of Theorem \ref{thm:directsum}. Let us assume that $(\mathcal{A},\mathpzc{H}_1,D_1)$ and $(\mathcal{A},\mathpzc{H}_2,D_2)$ are two higher order spectral triples that are summable of order $p_1$ and $p_2$ respectively. Their direct sum is represented by the ST\textsuperscript{2} $(\mathcal{A},\mathpzc{H}_1 \oplus \mathpzc{H}_2,(D_1\oplus 0,0 \oplus D_2))$. The ST\textsuperscript{2} $(\mathcal{A}_1,\mathpzc{H}_1 \oplus \mathpzc{H}_2,(D_1\oplus 0,0 \oplus D_2))$ will be $f$-summable for any $f:(0,\infty)^2\to (0,\infty)$ such that
$$(1+|D_1|^{t_1})^{-1} \oplus (1 + |D_2|^{t_2})^{-1} \in \mathcal{L}^{f(t_1,t_2)}(\mathpzc{H}_1) \oplus \mathcal{L}^{f(t_1,t_2)}(\mathpzc{H}_2) . $$
For instance, we could take
$$f(t_1,t_2):= \max\left\{ \frac{p_1}{t_1}, \frac{p_2}{t_2} \right\} .$$
\end{example}

If $(\mathcal{A},\mathpzc{H},\DC)$ is $f_1$-summable and $f_2\geq f_1$, then $(\mathcal{A},\mathpzc{H},\DC)$ is also $f_2$-summable. The reader should note that if $(\mathcal{A},\mathpzc{H},\DC)$ is $f$-summable then by complex interpolation it is also $\tilde{f}$-summable for any $\tilde{f}>f_0$ where $f_0$ is the homogeneous function of degree $-1$ given by 
$$f_0(t):=\frac{\inf_{s>0} sf(st|t|^{-1})}{|t|}.$$  
Here $|\cdot|$ is an arbitrary norm on $\R^n$. If the infimum is attained, $(\mathcal{A},\mathpzc{H},\DC)$ is $f_0$-summable.

The following is immediate from the fact that $\Dom(D_{\VEC{t}})=\cap_{j\in I} \Dom(|D_j|^{t_j})$ for a strictly anticommuting $n$-tuple $(D_j)_{j\in I}$.

\begin{prop}
Let $(\mathcal{A},\mathpzc{H},\DC)$ be an $f$-summable ST\textsuperscript{2}. For $\VEC{t}\in \Omega(\VEC{\epsilon})$, $(\mathcal{A},\mathpzc{H},D_{\VEC{t}})$ is an $f(\VEC{t})$-summable higher order spectral triple.
\end{prop}

\subsection{Equivariance of strictly tangled spectral triples}
\label{subsecvonf}

We now come to defining equivariance in strictly tangled spectral triples and, with the applications to parabolic geometry and dynamics in mind, we allow for conformal actions. The latter notion is less studied in noncommutative geometry, with some work in the last decade \cite{pongewang1,pongewang2} and recent work \cite{AMsomewhere} by the third listed author with Adam Rennie. 

In the non-conformal case, there are no additional technical issues arising in the equivariant setting. This follows from the same method of proof as that leading up to Theorem \ref{thm:st2_give_host}. We state this fact in a definition and proposition, restricting to the unital case for ease of exposition.

\begin{definition}
	\label{def:uniformly-equivariant-st2}
	Let $G$ be a locally compact group acting on the unital algebra $\mathcal{A}$ by $*$-automorphisms. An ST\textsuperscript{2} $(\mathcal{A},\mathpzc{H},\DC)$ is \emph{\( G \)-equivariant} if there is a unitary action of \( G \) on \( \mathpzc{H} \), implementing the action on \( \mathcal{A} \), such that, for each $i\in I$, $U_g$ preserves $\Dom D_i$ for all $g\in G$, and the map
	\[ g \mapsto (U_g D_i U_g^* - D_i) \Biggl(1+\sum_{j\in I}|D_j|^{\epsilon_{ij}} \Biggr)^{-1} \]
	is $*$-strongly continuous from \( G \) into the space of bounded operators on $\mathpzc{H}$.
\end{definition}

\begin{prop}
\label{lknlknlkandadlad}
If $(\mathcal{A},\mathpzc{H},\DC)$ is a $G$-equivariant ST\textsuperscript{2}, the higher order spectral triple $(\mathcal{A},\mathpzc{H},\overline{D}_{\VEC{t}})$ is $G$-equivariant for all $ \VEC{t} \in \Omega(\VEC{\epsilon}, (1, \ldots, 1))$.
\end{prop}

Naïvely, the right way of applying the idea of conformal equivariance to ST\textsuperscript{2}s would seem to be to have a collection of conformal factors, one for each operator in the collection \( \DC = (D_j)_{j \in I} \). Alas, this idea falls apart already in the simple example of the exterior product of two real line Dirac spectral triples,
\[ \left( C_c^{\infty}(\R^2), L^2(\R^2) \otimes \C^2, (\partial_{x_1} \otimes \gamma_1, \partial_{x_2} \otimes \gamma_2) \right),\]
whose bounding matrix is $\VEC{\epsilon}=0$. In this simple example the action of \( \R^2 \) by dilation in each direction, \( (r_1, r_2) : (x_1, x_2) \mapsto (r_1 x_1, r_2 x_2) \), makes any resulting HOST fail to be conformally \( \R^2 \)-equivariant. The source of this problem is actually deeper, however, because the bounded transform of any resulting HOST also cannot be \( \R^2 \)-equivariant.

However, under some circumstances, it may be possible to align the conformal factors so that the resulting higher order spectral triple is conformally equivariant. We will see such a phenomenon for Carnot groups in Proposition \ref{prop:carnot-equivariant}. In the example above with its dilation action, this is possible by restricting to a subgroup where \( r_1 = r_2^{\alpha} \), for some fixed \( \alpha \neq 0 \). If we choose \( \VEC{t} \in \R_+ (1, \alpha) \subset \Omega(\VEC{\epsilon}) = \R_+^2 \), the higher order spectral triple
\[ \left( C_c^{\infty}(\R^2), L^2(\R^2), \sign(\partial_{x_1})|\partial_{x_1}|^{t_1} \otimes \gamma_1 + \sign(\partial_{x_2})|\partial_{x_2}|^{t_2} \otimes \gamma_2 \right) \]
is conformally equivariant, with conformal factor \( r_1^{-t_1/2} = r_2^{-t_2/2} \).

Another example to consider is the direct sum of two real line Dirac spectral triples,
\[ \left( C_c^{\infty}(\R \sqcup \R) , L^2(\R) \oplus L^2(\R), (\partial_{x_1} \oplus 0, 0 \oplus \partial_{x_2}) \right) \qquad \VEC{\epsilon} = 0 \]
with an action of \( \R^2 \) by dilation on each corresponding copy of \( \R \), 
\[ (r_1, r_2) : x_1 \mapsto r_1 x_1 \qquad x_2 \mapsto r_2 x_2 \qquad (x_1 \in \R \sqcup \emptyset, x_2 \in \emptyset \sqcup \R) . \]
Here there is no restriction on \( \VEC{t} \in \Omega(\VEC{\epsilon}) = \R_+^2 \), as we may take the conformal factor to be \( r_1^{-t_1/2} \oplus r_2^{-t_2/2} \) on the higher order spectral triple
\[ \left( C_c^{\infty}(\R \sqcup \R) , L^2(\R) \oplus L^2(\R),  \sign(\partial_{x_1})|\partial_{x_1}|^{t_1} \oplus  \sign(\partial_{x_2})|\partial_{x_2}|^{t_2} \right) . \]

Unfortunately, the development of an abstract framework for conformal equivariance of ST\textsuperscript{2}s seems elusive. The main technical problem is to find conditions guaranteeing that, if $UDU^*-\mu D \mu^*$ is of ``lower order'', $UD^tU^*-\mu ^tD^t (\mu^*)^t$ is also of ``lower order''. For natural candidate conditions, we have been able neither to prove such a result in the abstract nor to find a counterexample.

The approach we take in the examples below is to take the following Proposition as giving an ad hoc notion of a conformally equivariant ST\textsuperscript{2}. Here, we fix \( \VEC{t} \) and give sufficient conditions for a single conformal factor \( (\mu_g)_{g \in G} \) to give rise to a conformally equivariant HOST at \( \VEC{t} \). A more general statement would be possible but this will suffice for our needs. One could view this approach as similar to the ``guess-and-check'' method of computing Kasparov products via either the bounded picture or Kucerovsky's theorem \cite{danthedan}.

\begin{prop}
	\label{prop:confst2guess}
	Let \( (\mathcal{A},\mathpzc{H},\DC) \) be an ST\textsuperscript{2} with a unitary action of \( G \) on \( \mathpzc{H} \), implementing the action on \( \mathcal{A} \). Suppose there exists a family \( (\mu_g)_{g \in G} \) of invertible bounded operators such that, for all $g\in G$, $\mu_g$, $\mu_g^*$, and \( U_g \) preserve $\Dom D_i$ for all $i$, with 
	\begin{align*}
		g &\mapsto [D_i, \mu_g] \Biggl(1+\sum_{j\in I}|D_j|^{\epsilon_{ij}}\Biggr)^{-1} & \text{and} &&
		g &\mapsto [D_i, \mu_g^*] \Biggl(1+\sum_{j\in I}|D_j|^{\epsilon_{ij}}\Biggr)^{-1}
	\end{align*}
	defining $*$-strongly continuous maps from \( G \) into the space of bounded operators on $\mathpzc{H}$.
	Suppose furthermore that, for some \( \VEC{t} \in \Omega(\VEC{\epsilon}) \cap (0, 1]^I \), the maps
	\begin{align*}
		g & \mapsto (U_g \sign(D_i)|D_i|^{t_i} U_g^* - \mu_g \sign(D_i)|D_i|^{t_i} \mu_g^*) \Biggl(1+\sum_{j\in I}|D_j|^{\epsilon_{ij}}\Biggr)^{-t_i} \quad \text{and} \\
		g & \mapsto U_g \Biggl(1+\sum_{j\in I}|D_j|^{\epsilon_{ij}}\Biggr)^{-t_i} U_g^* (U_g \sign(D_i)|D_i|^{t_i} U_g^* - \mu_g \sign(D_i)|D_i|^{t_i} \mu_g^*)
	\end{align*}
	are $*$-strongly continuous from \( G \) into the space of bounded operators on $\mathpzc{H}$. Then \( (\mathcal{A},\mathpzc{H},\overline{D}_{\VEC{t}}) \) is a conformally \( G \)-equivariant HOST with conformal factor \( \mu \).
\end{prop}

The proof is a straightforward extension of the proof of Theorem \ref{thm:st2_give_host}.
We give in Proposition \ref{prop:definiednoconf} a statement in the generality of Hilbert complexes. In this case, we will naturally begin with a collection of conformal factors which will need to be cajoled into cooperating with one another and so into giving a single conformal factor \( \mu \) for the HOST.

\section{Strictly tangled spectral triples from differential complexes}
\label{sec:complexes}

The main application of strictly tangled spectral triples that we study in this paper comes from Hilbert complexes and, more concretely, Rockland complexes on filtered manifolds. We first present an abstract framework for Hilbert complexes and proceed to describe it in detail for Rockland complexes.

\subsection{Hilbert complexes}
\label{subsechilbert}

We first recall the notion of a Hilbert complex. We follow the presentation of \cite{leschbruening} and refer the reader there for further details.

\begin{definition}
A Hilbert complex
$$0\to \mathcal{H}_0\xrightarrow{\rd_0}\mathcal{H}_1\xrightarrow{\rd_1}\cdots  \xrightarrow{\rd_{n-2}}\mathcal{H}_{n-1}\xrightarrow{\rd_{n-1}}\mathcal{H}_n\to 0 ,$$
abbreviated as $(\mathcal{H}_\bullet,\rd_\bullet)$,
consists of Hilbert spaces $\mathcal{H}_0,\mathcal{H}_1,\ldots, \mathcal{H}_n$ and closed densely defined maps $\rd_i:\mathcal{H}_{i}\dashrightarrow \mathcal{H}_{i+1}$ with the property that 
$$\mathrm{Ran}(\rd_{i-1})\subseteq \ker(\rd_i).$$
We say that $(\mathcal{H}_\bullet,\rd_\bullet)$ is Fredholm if the cohomology groups 
$$H^i(\mathcal{H}_\bullet,\rd_\bullet):=\ker(\rd_i)/\mathrm{Ran}(\rd_{i-1})$$
are finite-dimensional. We say that $(\mathcal{H}_\bullet,\rd_\bullet)$ has discrete spectrum if, for each $i$, the self-adjoint Laplacian
$\rd_i^*\rd_i+\rd_{i-1}\rd_{i-1}^*$,
densely defined on $\mathcal{H}_i$,
has discrete spectrum, i.e.~the spectrum consists of isolated eigenvalues of finite multiplicity.
\end{definition}

By \cite[Theorem 2.4]{leschbruening}, $(\mathcal{H}_\bullet,\rd_\bullet)$ is Fredholm if and only if $0$ is not in the essential spectrum of all the Laplacians $\rd_i^*\rd_i+\rd_{i-1}\rd_{i-1}^*$. In particular, $(\mathcal{H}_\bullet,\rd_\bullet)$ is Fredholm if it has discrete spectrum. We shall make use of a construction analogous to Rumin--Seshadri's construction of Laplacians in the Rumin complex \cite{ruminsesh} (see also \cite{davehaller}). Given parameters $\VEC{m}=(m_0,\ldots,m_{n-1})\in [1,\infty)^n$ that we refer to as an order and a Hilbert complex $(\mathcal{H}_\bullet,\rd_\bullet)$ we define the Rumin Laplacians
$$\Delta_{\VEC{m},i}^R:=(\rd_i^*\rd_i)^{a_i}+(\rd_{i-1}\rd_{i-1}^*)^{a_{i-1}},$$
where $a_i=\prod_{l\neq i} m_l=m/m_i$ for $m =  \prod_{l=1}^n m_l$. Clearly, $(\mathcal{H}_\bullet,\rd_\bullet)$ has discrete spectrum if and only if all the self-adjoint operators 
$\Delta_{\VEC{m},i}^R$ have compact resolvent.
We also introduce, for $s\geq 0$, the abstract Sobolev spaces
$$\mathcal{H}_{i,\VEC{m}}^s:=\Dom((\Delta_{\VEC{m},i}^R)^{s/2m})\subseteq \mathcal{H}_i.$$

\begin{definition}
Let $\mathcal{A}$ be a $*$-algebra. A Hilbert complex over $\mathcal{A}$ of order $\VEC{m}=(m_0,\ldots,m_{n-1})\in [1,\infty)^n$ is a Hilbert complex 
$$0\to \mathcal{H}_0\xrightarrow{\rd_0}\mathcal{H}_1\xrightarrow{\rd_1}\cdots \xrightarrow{\rd_{n-2}} \mathcal{H}_{n-1}\xrightarrow{\rd_{n-1}}\mathcal{H}_n\to 0,$$
where each $\mathcal{H}_i$ is a left $\mathcal{A}$-module under $*$-representations 
$$\pi_i:\mathcal{A}\to \mathbb{B}(\mathcal{H}_i)$$
such that, for any $a\in \mathcal{A}$, $\pi_i(a)$ preserves $\Dom(\rd_i)$ and the densely defined operators
\begin{align*}
\left(\rd_i \pi_i(a)-\pi_{i+1}(a)\rd_i\right)&(1+\Delta_{\VEC{m},i})^{\frac{1-m_i}{2m}} \quad\mbox{and}\\
& (1+\Delta_{\VEC{m},i+1})^{\frac{1-m_i}{2m}}\left(\rd_i \pi_i(a)-\pi_{i+1}(a)\rd_i\right)
\end{align*}
are norm bounded. 

If, for all $s\geq 0$, $\pi_i(a)$ preserves the domain of $\rd_i$ as an operator on the Sobolev spaces $\mathcal{H}_{i,\VEC{m}}^s$ and the densely defined operator $\left(\rd_i \pi_i(a)-\pi_{i+1}(a)\rd_i\right)(1+\Delta_{\VEC{m},i})^{\frac{1-m_i}{2m}}$ is continuous in norm $\mathcal{H}_{i,\VEC{m}}^s\to \mathcal{H}_{i+1,\VEC{m}}^s$ then we say that $(\mathcal{H}_\bullet,\rd_\bullet)$ is a regular Hilbert complex over $\mathcal{A}$ of order $\VEC{m}$.
\end{definition}

To ease the notation, we drop the representations $\pi_i$ when they are clear from the context, writing $[\rd_i,a]$ instead of $\rd_i \pi_i(a)-\pi_{i+1}(a)\rd_i$ for $a\in \mathcal{A}$. 

\begin{lemma}
\label{alknaljdna}
Let $(\mathcal{H}_\bullet,\rd_\bullet)$ be a Hilbert complex which is Fredholm and of order $\VEC{m}=(m_0,m_1,\ldots, m_{n-1})\in [1,\infty)^n$. Write $a_i=\prod_{l\neq i} m_l=m/m_i$ for $m=\prod_{l=0}^{n-1} m_l$. Then, setting $\mathpzc{H}=\bigoplus_i \mathcal{H}_i$ and defining the operators
$$D_i=\rd_{i}+\rd_{i}^*,\quad\mbox{with}\quad \Dom(D_j):=\Dom(\rd_i)\cap \Dom(\rd_i^*),$$
the collection $\DC=(D_i)_{i=0}^{n-1}$ is a strictly anticommuting collection of selfadjoint operators on $\mathpzc{H}$. Morever, for any $\alpha$ we have that 
$$D_i|D_i|^{\alpha}=D_i((\Delta_{\VEC{m},i}^R)^{\alpha/2a_{i-1}}+(\Delta_{\VEC{m},i-1}^R)^{\alpha/2a_{i-1}}).$$
\end{lemma}

\begin{proof}
We start by noting that $\cap_k\cap_{i=0}^{n-1} \Dom(D_i^{2k})=\cap_{\VEC{t}}\Dom(\Delta_\DC^{\VEC{t}})$ is a common core for the collection $\DC=(D_i)_{i=0}^{n-1}$, and on there we have for $i\neq j$ that 
$$D_iD_j=\rd_{i}\rd_{j}+\rd_{i}\rd_{j}^*+\rd_{i}^*\rd_{j}+\rd_{i}^*\rd_{j}^*=0.$$
Here the first and fourth term vanishes since $\rd^2=0$. The second vanishes since $\rd_j^*$ take values in the $j$:th Hilbert space whereas $\rd_i$ vanishes outside the $i$:th Hilbert space, and the third term vanishes for similar reasons. In particular,  $\DC=(D_i)_{i=0}^{n-1}$ is a strictly anticommuting collection of selfadjoint operators.

We remark that \( a_i \frac{1-m_i}{2m} = \frac{1-m_i}{2m_i} = \frac{1}{2} (-1 + \frac{1}{m_i}) \).
Since the Hilbert complex is Fredholm, 
$$(\Delta_{\VEC{m},i}^R)^{\beta}=(\rd_i^*\rd_i)^{\beta a_i}+(\rd_{i-1}\rd_{i-1}^*)^{\beta a_{i-1}}.$$
In particular, 
$$D_i(\Delta_{\VEC{m},i}^R)^{\beta}=\rd_{i-1}^*(\rd_{i-1}\rd_{i-1}^*)^{\beta a_{i-1}}\quad \mbox{and}\quad D_i(\Delta_{\VEC{m},i-1}^R)^{\beta}=\rd_{i-1}(\rd_{i-1}^*\rd_{i-1})^{\beta a_{i-1}} . $$
On the other hand, 
$$|D_i|^{\alpha}=(\rd_{i-1}^*\rd_{i-1})^{\alpha/2}+(\rd_{i-1}\rd_{i-1}^*)^{\alpha/2} $$
so 
$$D_i|D_i|^{\alpha}=\rd_{i-1}(\rd_{i-1}^*\rd_{i-1})^{\alpha/2}+\rd_{i-1}^*(\rd_{i-1}\rd_{i-1}^*)^{\alpha/2} $$
and the lemma follows.
\end{proof}

\begin{thm}
\label{st2fromhilb}
Let $(\mathcal{H}_\bullet,\rd_\bullet)$ be a Hilbert complex over $\mathcal{A}$ of order $\VEC{m}$ with discrete spectrum. We set
$\mathpzc{H}=\bigoplus_i \mathcal{H}_i$
and write $\DC=(D_i)_{i=0}^{n-1}$ for the collection
$D_i=\rd_{i}+\rd_{i}^*$.
It then holds that the collection $(\mathcal{A},\mathpzc{H},\DC)$ is a ST\textsuperscript{2} with bounding matrix $\VEC{\epsilon}=(\epsilon_{ij})_{i,j=0}^{n-1}$ where
\begin{equation}
\label{alknalkjdnajlkdn}
\epsilon_{ij}=\begin{cases} 
\frac{m_{i}-1}{m_{j}} \; &j=i-1,i,i+1\\
0 \; &\mbox{otherwise}. \end{cases}
\end{equation}
Furthermore, if $(\mathcal{H}_\bullet,\rd_\bullet)$ is regular, $(\mathcal{A},\mathpzc{H},\DC)$ is \( \VEC{\rho} \)-preserving for \( \VEC{\rho} = (\infty, \ldots, \infty) \). 
\end{thm}

\begin{proof}
We have that the bounding matrix \eqref{alknalkjdnajlkdn} satisfies the decreasing cocycle condition by the same argument as in \eqref{estforesp} (with the first equality of \eqref{estforesp} replaced by an upper bound).
Since $(\mathcal{H}_\bullet,\rd_\bullet)$ has discrete spectrum, what remains to prove is the commutator condition. And $(\mathcal{H}_\bullet,\rd_\bullet)$ is a Hilbert complex over $\mathcal{A}$ of order $\VEC{m}$, so 
\begin{align*}
[\rd_i,a](1+\Delta_{\VEC{m},i})^{\frac{1-m_i}{2m}} & = [\rd_i,a](1+\Delta_{\VEC{m},i})^{\frac{1}{2m}-\frac{1}{2a_i}} \quad\mbox{and}\\
(1+\Delta_{\VEC{m},i+1})^{\frac{1-m_i}{2m}}[\rd_i,a] & =(1+\Delta_{\VEC{m},i+1})^{\frac{1}{2m}-\frac{1}{2a_i}}[\rd_i,a]
\end{align*}
are bounded. Since $\Delta_{\VEC{m},i}^R=(\rd_i^*\rd_i)^{a_i}+(\rd_{i-1}\rd_{i-1}^*)^{ a_{i-1}}$, we conclude from the boundedness of the first operator that 
$$[\rd_{i},a]\left(1+|D_i|^{1-\frac{1}{m_{i-1}}}+|D_{i+1}|^{\frac{1-m_{i-1}}{m_{i}}}\right)^{-1}$$
is bounded and from the boundedness of the second operator that 
$$[\rd_{i}^*,a]\left(1+|D_i|^{1-\frac{1}{m_{i-1}}}+|D_{i-1}|^{\frac{1-m_{i-1}}{m_{i-2}}}\right)^{-1}$$
is bounded.
\end{proof}

For instance, for a complex with \( n = 5 \), the graph corresponding to the bounding matrix would be
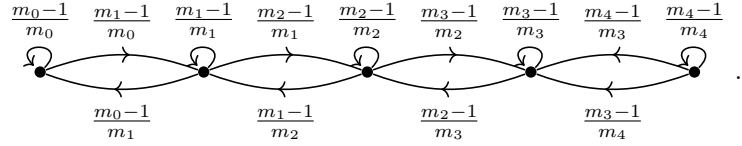
\[ \vcenter{\hbox{\begin{tikzpicture}
	\graph[grow right sep=2cm, empty nodes, nodes={fill=black, circle, inner sep=1.5pt}, edges={semithick}]{
	a --["$\frac{m_0 - 1}{m_1}$"' inner sep=5.5pt, middlearrow={<}, bend right=20] b --["$\frac{m_1 - 1}{m_2}$"' inner sep=5.5pt, middlearrow={<}, bend right=20] c --["$\frac{m_2 - 1}{m_3}$"' inner sep=5.5pt, middlearrow={<}, bend right=20] d --["$\frac{m_3 - 1}{m_4}$"' inner sep=5.5pt, middlearrow={<}, bend right=20] e;
	e --["$\frac{m_4 - 1}{m_3}$"' inner sep=5.5pt, middlearrow={>}, bend right=20] d --["$\frac{m_3 - 1}{m_2}$"' inner sep=5.5pt, middlearrow={<}, bend right=20] c --["$\frac{m_2 - 1}{m_1}$"' inner sep=5.5pt, middlearrow={<}, bend right=20] b --["$\frac{m_1 - 1}{m_0}$"' inner sep=5.5pt, middlearrow={<}, bend right=20] a;
	a --["$\frac{m_0 - 1}{m_0}$"', loop] a;
	b --["$\frac{m_1 - 1}{m_1}$"', loop] b;
	c --["$\frac{m_2 - 1}{m_2}$"', loop] c;
	d --["$\frac{m_3 - 1}{m_3}$"', loop] d;
	e --["$\frac{m_4 - 1}{m_4}$"', loop] e;
	};
\end{tikzpicture}}} . \]

\begin{remark}
If $(\mathcal{H}_\bullet,\rd_\bullet)$ is a Hilbert complex with discrete spectrum over $\mathcal{A}$, there are multiple ways of grading the ST\textsuperscript{2} $(\mathcal{A},\mathpzc{H},\DC)$. The first option is to use the grading coming from the complex in which 
$$\mathpzc{H}_+=\bigoplus_i \mathcal{H}_{2i},\quad\mbox{and}\quad \mathpzc{H}_-=\bigoplus_i \mathcal{H}_{2i+1}.$$
Another option arises if $(\mathcal{H}_\bullet,\rd_\bullet)$ satisfies a mild strengthening of Poincaré duality; see \cite[Lemma 2.16]{leschbruening}. Assume that we have $\mathcal{A}$-linear unitaries $\gamma_i:\mathcal{H}_i\to \mathcal{H}_{n-i}$ such that 
$$\rd_{n-i-1}^*\gamma_i=-\gamma_{i+1}\rd_i, \quad\mbox{and}\quad \gamma_{n-i}\gamma_i=1_{\mathcal{H}_i}.$$
We can then define a symmetry $\gamma=\bigoplus \gamma_j$ on $\mathpzc{H}$ that anticommutes with $D_j$, for $j=1,\ldots, n$. In particular, $\gamma$ grades $\mathpzc{H}$ in such a way that the ST\textsuperscript{2} constructed in Theorem \ref{st2fromhilb} forms an even ST\textsuperscript{2}. This construction is analogous to the grading induced from the Hodge star on differential forms defining the signature operator from the Hodge--de Rham operator described in Example \ref{hodgeandsign}.
\end{remark}

\begin{prop}
\label{fixorder}
Assume that $(\mathcal{A},\mathpzc{H},\DC)$ is an ST\textsuperscript{2} defined from a Hilbert complex with discrete spectrum $(\mathcal{H}_\bullet,\rd_\bullet)$ over $\mathcal{A}$ of order $\VEC{m}$ and bounding matrix $\VEC{\epsilon}$ as in \eqref{alknalkjdnajlkdn}. Then, for any $\tau>0$, 
$$\VEC{t}_{\VEC{m}}(\tau):=\left(\frac{\tau}{m_0},\frac{\tau}{m_1},\ldots, \frac{\tau}{m_{n-1}}\right)\in \Omega(\VEC{\epsilon})$$
and 
$$D_{\VEC{t}_{\VEC{m}}(\tau)}=\sum_{i=0}^{n-1} \rd_i(\Delta_{\VEC{m},i}^R)^{\frac{\tau-m_i}{2m}}+\rd_i^*(\Delta_{\VEC{m},i+1}^R)^{\frac{\tau-m_i}{2m}}.$$
\end{prop}

\begin{proof}
We see that \(\VEC{t}_{\VEC{m}}(\tau)\in \Omega(\VEC{\epsilon})\) since $\frac{\tau}{m_j} > \frac{m_i-1}{m_j} \frac{\tau}{m_i}$ and the expression for \(D_{\VEC{t}_{\VEC{m}}(\tau)}\) follows from Lemma \ref{alknaljdna}.
\end{proof}

Recall the weak Hodge decomposition of \cite[Lemma 2.1]{leschbruening}, 
\[ \mathcal{H}_i = \mathscr{H}_i \oplus \overline{\Ran \rd_{i-1}}  \oplus \overline{\Ran \rd_i^*} \]
where \( \mathscr{H}_i = \ker \rd_i  \cap \ker \rd_{i-1}^* = \ker \Delta_{\VEC{m},i}^R \). Note that if $(\mathcal{H}_\bullet,\rd_\bullet)$ is Fredholm, e.g.~if it has discrete spectrum, the ranges are automatically closed with 
$$\Ran \rd_{i-1}=(\ker\rd_{i-1}^*)^\perp\quad\mbox{and}\quad \Ran \rd_i^*=(\ker\rd_i)^\perp.$$ 

We can in fact build up a conformally equivariant higher order spectral triple by specifying conformal factors on each part of the decomposition. To do so, we require some setup. Consider a unital $*$-algebra $\mathcal{A}$ with an action of a locally compact group $G$, and a Hilbert complex
$$0\to \mathcal{H}_0\xrightarrow{\rd_0}\mathcal{H}_1\xrightarrow{\rd_1}\cdots \mathcal{H}_{n-1}\xrightarrow{\rd_{n-1}}\mathcal{H}_n\to 0$$
over $\mathcal{A}$ of order $\VEC{m}$ with a unitary action \( U_i \) of \( G \) on each \(\mathcal{H}_i \) intertwining the representation of \( \mathcal{A} \) and preserving the domains of \( \rd_{\bullet} \). We can then specify conformal factors on each part of the decomposition by means of families of invertible operators $ (\nu_i)_{g \in G} \subset \mathbb{B}(\overline{\Ran \rd_i^*}) $ and $ (\tilde{\nu}_i)_{g \in G} \subset \mathbb{B}(\overline{\Ran \rd_{i-1}}) $, and the conditions we impose are that all of them and their adjoints preserving the domains of \( \rd_{\bullet} \), as well as that 
\begin{enumerate}
\item the densely defined operators
\[
	\left(\tilde{\nu}_{i+1,g}\rd_i -\rd_i\nu_{i,g}\right) (1+\Delta_{\VEC{m},i})^{\frac{1-m_i}{2m}}\quad\mbox{and} \quad
	(1+\Delta_{\VEC{m},i+1})^{\frac{1-m_i}{2m}}\left(\tilde{\nu}_{i+1,g}\rd_i -\rd_i\nu_{i,g}\right)
\]
are in fact bounded and define $*$-strongly continuous functions $G\to \mathbb{B}(\mathcal{H}_i,\mathcal{H}_{i+1})$;
\item for some \( \VEC{t} \in \Omega(\VEC{\epsilon}) \), the densely defined operators
\begin{multline*}
	\left(U_{i+1,g} \rd_i (\Delta_{\VEC{m},i}^R)^{\frac{-1 + t_i}{2 m} m_i} U_{i,g}^*-\tilde{\nu}_{i+1,g} \rd_i(\Delta_{\VEC{m},i}^R)^{\frac{-1 + t_i}{2 m} m_i} \nu_{i,g}\right) (1+\Delta_{\VEC{m},i})^{\frac{1 - m_i}{2m} t_i} \quad\mbox{and} \\
	(1+\Delta_{\VEC{m},i})^{\frac{1 - m_i}{2m} t_i} \left(U_{i+1,g} \rd_i (\Delta_{\VEC{m},i}^R)^{\frac{-1 + t_i}{2 m} m_i} U_{i,g}^*-\tilde{\nu}_{i+1,g} \rd_i(\Delta_{\VEC{m},i}^R)^{\frac{-1 + t_i}{2 m} m_i} \nu_{i,g}\right)
\end{multline*}
are in fact bounded and define $*$-strongly continuous functions $G\to\mathbb{B}(\mathcal{H}_i,\mathcal{H}_{i+1})$.
\end{enumerate}

\begin{prop}
\label{prop:definiednoconf}
Let $\mathcal{A}$ be a unital $*$-algebra with an action of a locally compact group $G$. Let 
$$0\to \mathcal{H}_0\xrightarrow{\rd_0}\mathcal{H}_1\xrightarrow{\rd_1}\cdots \mathcal{H}_{n-1}\xrightarrow{\rd_{n-1}}\mathcal{H}_n\to 0$$
be a Hilbert complex over $\mathcal{A}$ of order $\VEC{m}$ with a unitary action \( U_i \) of \( G \) on each \(\mathcal{H}_i \) intertwining the representation of \( \mathcal{A} \) and preserving the domains of \( \rd_{\bullet} \). Assume that  $ (\nu_i)_{g \in G} \subset \mathbb{B}(\overline{\Ran \rd_i^*}) $ and $ (\tilde{\nu}_i)_{g \in G} \subset \mathbb{B}(\overline{\Ran \rd_{i-1}}) $ are families of invertible operators satisfying the assumptions preceding the proposition for some \( \VEC{t} \in \Omega(\VEC{\epsilon}) \), then \( (\mathcal{A},\mathpzc{H},\overline{D}_{\VEC{t}}) \) is conformally equivariant with conformal factor
\[ \mu = \bigoplus_j \nu_j + \tilde{\nu_j} + P_{\mathscr{H}_j} . \]
\end{prop}

\subsection{Rockland complexes}
\label{subsecrock}

We now turn to Rockland sequences on filtered manifolds. Rockland sequences were studied in detail in Dave--Haller's work \cite{davehaller}. The associated analysis relies heavily on van Erp--Yuncken's Heisenberg calculus \cite{vanerpyunck} on a filtered manifold. Filtered manifolds are known also as Carnot manifolds, and relate to the equiregular differential systems of sub-Riemannian geometry. We tacitly assume all manifolds to be compact throughout this section.

\subsubsection{The Heisenberg calculus} Let us briskly recall the geometry of filtered manifolds and their Heisenberg calculus. We refer the details to the literature \cite{davehaller,goffengkuzmin,vanerpyunck}. A filtered manifold is a manifold $X$ equipped with a filtering 
$$TX=T^{-r}X\supsetneq T^{-r+1}X\supsetneq \ldots\supsetneq T^{-2}X\supsetneq T^{-1}X\supsetneq 0$$
of sub-bundles such that $[T^{-j}X,T^{-k}X]\subseteq T^{-j-k}X$ for any $j,k$. We call $r$ the depth of $X$. We write 
$$\mathfrak{t}_H X=\bigoplus_j T^{-j}X/T^{-j+1}X$$
for the associated graded bundle. Taking commutators of vector fields induces a fibrewise Lie bracket on $\mathfrak{t}_HX$, making $\mathfrak{t}_HX\to X$ a Lie algebroid. The fibres are nilpotent of step length at most $r$, so the Baker--Campbell--Hausdorff formula implies that $\mathfrak{t}_HX$ integrates to a Lie groupoid $T_HX\rightrightarrows X$ (with the same range and source map). Concretely, as a fibre bundle, $T_HX=\mathfrak{t}_HX$. However, $T_HX$ carries a fibrewise polynomial group operation defined from the Baker--Campbell--Hausdorff formula and the commutator of vector fields modulo lower order terms in the filtration. We call $T_HX\rightrightarrows X$ the osculating Lie groupoid. The osculating Lie groupoid carries an $\R_+$-action $\delta$ defined from integrating the $\R_+^{\times}$-action on $\mathfrak{t}_HX$ defined from its grading.

The Heisenberg calculus on a filtered manifold introduced by van Erp--Yuncken \cite{vanerpyunck} is built from operators whose Schwartz kernels in appropriate exponential coordinates are defined from $r$-fibred distributions on $T_HX$ that expand asymptotically into a sum of almost homogeneous fibrewise convolution operators. A way to formalize this statement uses van Erp--Yuncken's parabolic tangent groupoid \cite{vanerpyuncktang}, a Lie groupoid $\mathbb{T}_HX\rightrightarrows X\times [0,\infty)$. As a set, 
$$\mathbb{T}_HX=T_HX\times\{0\} \sqcup X\times X\times (0,\infty),$$
with the groupoid structure of $T_HX$ on the first component and the pair groupoid structure on the second component. The Lie groupoid structure on $\mathbb{T}_HX\rightrightarrows X\times [0,\infty)$ is defined using a blowup in exponential coordinates defined from a graded connection. The parabolic tangent groupoid carries an $\R_+$-action called the zoom action, which by an abuse of notation we also denote by $\delta$, acting by 
$$ \delta_\lambda : (x,v,0) \mapsto (x,\delta_\lambda(v),0) \qquad (x,y,t) \mapsto (x,y,\lambda^{-1}t) . $$
A Heisenberg pseudodifferential operator $T$ of order $m$ is defined to be an operator on $C^\infty(X)$ whose Schwartz kernel $k_T\in \mathcal{D}'(X\times X)$ can be written as the evaluation at $t=1$ of a properly supported, $r$-fibred distribution $K\in \mathcal{D}'_ r(\mathbb{T}_HX)$ which is homogeneous of order $m$ modulo properly supported elements under the zoom action. In exponential coordinates, we can Taylor expand such a $K$ at $t=0$ and arrive at an asymptotic sum
\begin{equation}
\label{askadlknad}
K(x,v,t)\sim \sum_{j=0}^\infty t^jk_j(x,v),
\end{equation}
where $k_j\in \mathcal{E}'_r(T_HX)$ is homogenenous modulo $C^\infty_c(T_HX)$ of degree $m-j$. Here $K$ and the collection $(k_j)_{j=0}^\infty$ are uniquely determined by $k_T$ modulo respectively properly and compactly supported smooth elements. Writing $\Psi^m_H(X)$ for the space of Heisenberg pseudodifferential operators of order $m$, we arrive at a short exact sequence 
$$0\to \Psi^{m-1}_H(X)\to \Psi^{m}_H(X)\xrightarrow{\sigma_H^m} \Sigma^{m}_H(X)\to 0,$$
where $\Sigma^{m}_H(X)\subseteq \mathcal{E}'_r(T_HX)/C^\infty_c(T_HX)$ consists of elements homogenenous of degree $m$. The map $\sigma_H^m$ is called the principal symbol and is defined by $\sigma_H^m(T):=[k_0]$ for $k_0$ the leading term in \eqref{askadlknad}. A composition of Heisenberg pseudodifferential operators of order $m$ and $m'$ respectively as operators on $C^\infty(X)$ is again a Heisenberg pseudodifferential operator but of order $m+m'$. The principal symbol respects products in the sense that 
$$\sigma_H^{m+m'}(TT')=\sigma_H^{m}(T)*\sigma_H^{m'}(T'), \quad T\in \Psi^{m}_H(X), \; T'\in \Psi^{m'}_H(X)$$
where $*$ denotes groupoid convolution on $T_HX$.

We can realize the principal symbol algebra in a more concrete way. Write $\mathcal{S}(T_HX)\subseteq C^\infty(T_HX)$ for the space of fibrewise Schwarz functions: functions that together with their derivatives decay faster than the reciprocal of any polynomial in the fibre. We define $\mathcal{S}_0(T_HX)$ to consist of those functions $f\in \mathcal{S}(T_HX)$ such that for any $x\in X$ and any polynomial $p$ on $T_xX$ we have 
$$\int_{T_xX} p(v)f(x,v)\rd v=0.$$
The space $\mathcal{S}_0(T_HX)$ is closed under convolution and is dense in the ideal of $C^*(T_HX)$ of elements vanishing in the fibrewise trivial representations. We embed $\Sigma^{m}_H(X)$ in the multipliers of $\mathcal{S}_0(T_HX)$ as follows. Any element $k\in \Sigma^{m}_H(X)$ can be represented near the zero section $X\subseteq T_HX$ by an $r$-fibred distribution $\hat{k}\in \mathcal{D}'_r(T_HX)$ of the form
$$\hat{k}=\hat{k}_0+p\log|\cdot|,$$
where $\hat{k}_0$ is homogeneous of degree $m$, $p$ is fibrewise polynomial and where $|\cdot|$ is a fibrewise gauge (smooth outside the zero section and homogeneous of degree $1$). Upon fixing $|\cdot|$, the distribution $\hat{k}$ is unique up to a fibrewise polynomial. In particular, the muliplier on $\mathcal{S}_0(T_HX)$ defined by convolution by $\hat{k}$ depends only on $k\in \Sigma^{m}_H(X)$.

To understand further the principal symbol, we study its action in localizations of $\mathcal{S}_0(T_HX)$ in its $*$-representations. Whenever $(\pi,\mathcal{H})$ is a unitary representation of a nilpotent group $G$, we write $\mathcal{S}_0(\pi):=\pi(\mathcal{S}_0(G))\mathcal{H}$. If $\pi$ does not weakly contain the trivial representation, $\mathcal{S}_0(\pi)=\pi(\mathcal{S}(G))\mathcal{H}$ and is dense in $\mathcal{H}$. Moreover, any multiplier $k$ of $\mathcal{S}_0(G)$ localizes to an operator $\pi(k)$ on $\mathcal{H}$ with domain $\mathcal{S}_0(\pi)$ defined by $\pi(k)(\pi(a)\xi)=\pi(k*a)\xi$ for $a\in \mathcal{S}_0(G)$ and $\xi\in \mathcal{H}$. We can therefore for a Heisenberg pseudodifferential operator $T$ of order $m$, $x\in X$ and a unitary representation $\pi$ of $(T_HX)_x$, define the represented symbol 
$$\sigma_H^m(T,\pi) =\pi(\sigma_H^m(T)):\mathcal{S}_0(\pi)\to \mathcal{S}_0(\pi).$$
The discussion above readily extends to operators on vector bundles. We denote the space of Heisenberg pseudodifferential operators of order $m$ from the vector bundle $E_1$ to $E_2$ by $\Psi^{m}_{H}(X; E_1,E_2)$. We recall the following important definition.

\begin{definition}
Let $X$ be a filtered manifold and $E_1,E_2\to X$ two vector bundles. Assume that $T:C^\infty(X,E_1)\to C^\infty(X,E_2)$ is a Heisenberg pseudodifferential operator of order $m$. We say that $T$ satisfies the \emph{Rockland condition} if, for any $x\in X$ and any irreducible, non-trivial, unitary representation $\pi$ of $(T_HX)_x$, the represented symbol 
$$\sigma_H^m(T,\pi):=\pi(\sigma_H^m(T)):\mathcal{S}_0(\pi)\otimes E_{1,x}\to \mathcal{S}_0(\pi)\otimes E_{2,x}$$
is injective. If the represented symbol in all points and all irreducible, non-trivial, unitary representations is bijective then we say that $T$ is $H$-elliptic.
\end{definition}

Operators in the Heisenberg calculus act continuously in a scale of Sobolev spaces adapted to the filtering. Fix a volume density on $X$. Following \cite{davehallerheat,davehaller}, we know that there exists a family of $H$-elliptic operators $(A^t)_{t\in \R}$ (in fact the complex powers of a single $H$-elliptic operator) that we can assume satisfies $A^0=1$ and $A^t A^{-t}=1$. We define $W^s_H(X):=A^{-s}L^2(X)\subseteq \mathcal{D}'(X)$ with inner product defined by declaring $A^s:W^s_H(X)\to L^2(X)$ unitary. A similar definition can be made also for vector bundles. Any $T\in \Psi^m_H(X;E_1,E_2)$ extends by density to a continuous operator
$$T:W^{s_1}_H(X;E_1)\to W^{s_2}_H(X;E_2)$$
as soon as $s_1+m\geq s_2$ and a compact operator when $s_1+m> s_2$.

\begin{thm}
Let $X$ be a closed filtered manifold equipped with a volume density, let $E_1,E_2\to X$ be two hermitian vector bundles, and let $T:C^\infty(X,E_1)\to C^\infty(X,E_2)$ be a Heisenberg pseudodifferential operator of order $m$. Then the following are equivalent:
\begin{enumerate}
\item $T$ and $T^*$ satisfy the Rockland condition;
\item $T$ is $H$-elliptic;
\item $T:W^s_H(X;E_1)\to W^{s-m}_H(X;E_2)$ is Fredholm for some $s$; and
\item  $T:W^s_H(X;E_1)\to W^{s-m}_H(X;E_2)$ is Fredholm for all $s$.
\end{enumerate}
Moreover, $H$-elliptic operators are hypoelliptic and admit parametrices in the Heisenberg calculus.
\end{thm}

Here it is clear that 4) implies 3) and 2) implies 1). That 3) implies 2) is proven in \cite{AMY} and that 1) implies 4) is proven in \cite{davehaller}.

For summability results of spectral triples and ST\textsuperscript{2}s on filtered manifolds, we will use Dave--Haller's Weyl law in the Heisenberg calculus \cite{davehallerheat}. Its statement gives a leading term in the eigenvalue of positive, even-order, $H$-elliptic, differential operators in the Heisenberg calculus. For a filtered manifold $X$, we define its homogeneous dimension as
\begin{equation}
\label{homodim}
\dim_h(X)=\sum_j j ~\mathrm{rk}(T^{-j}X/T^{-j+1}X).
\end{equation}
Dave--Haller's Weyl law \cite{davehallerheat} implies that if $T\in \Psi^m_H(X;E_1,E_2)$ for an $m<0$ then 
\begin{equation}
\label{weyldave}
\mu_k(T)=O(k^{\dim_h(X)/m}).
\end{equation}
In particular, for $m<0$,
$$\Psi^m_H(X;E_1,E_2)\subseteq \mathcal{L}^p(L^2(X,E_1),L^2(X,E_2))\qquad (p>-\dim_h(X)/m).$$ 

\subsubsection{Rockland complexes}

We now turn to studying Rockland complexes. They play the role of elliptic complexes on filtered manifolds. We start by recalling the definition and proceed to place it in the context of the preceding subsection by building ST\textsuperscript{2}s for filtered manifolds.

\begin{definition}
Consider a collection $E_\bullet=(E_0,E_1,\ldots, E_n)$ of hermitian vector bundles $E_j\to X$ and numbers $\VEC{m}=(m_0,\ldots, m_{n-1})\in (0,\infty)^n$. We let
\begin{multline}
\label{lkankland}
\rd_\bullet: \quad 0\to C^\infty(X;E_0)\xrightarrow{\rd_0} C^\infty(X;E_1)\xrightarrow{\rd_1}\cdots \\
\cdots \xrightarrow{\rd_{n-2}} C^\infty(X;E_{n-1})\xrightarrow{\rd_{n-1}}C^\infty(X;E_n)\to 0
\end{multline}
be a complex with maps $\rd_j\in \Psi^{m_j}_{H}(X; E_{j},E_{j+1})$. We say that the complex $\rd_\bullet$ in Equation \eqref{lkankland} is a \emph{Rockland complex} if the symbol sequence $\sigma_H(\rd_\bullet)$ defined by
\begin{multline}
\label{eq:symbol_seq_of_rockland_complex}
\sigma_H(\rd_\bullet): \quad 0\to \mathcal{S}_0(T_HX;E_0)\xrightarrow{\sigma_H^{m_0}(\rd_0)} \mathcal{S}_0(T_HX;E_1)\xrightarrow{\sigma_H^{m_1}(\rd_1)}\cdots\\
\cdots \xrightarrow{\sigma_H^{m_{n-2}}(\rd_{n-2})} \mathcal{S}_0(T_HX;E_{n-1})\xrightarrow{\sigma_H^{m_{n-1}}(\rd_{n-1})}\mathcal{S}_0(T_HX;E_n)\to 0
\end{multline}
is localized to an exact sequence by any non-trivial, irreducible, unitary representation of the osculating Lie groupoid $T_HX$. We say that $\VEC{m}$ is the order of $\rd_\bullet$.
\end{definition}

There are many interesting examples of Rockland sequences. As shown in \cite{davehaller}, and further discussed in \cite{goffeng24}, there is a general procedure for producing (graded) Rockland complexes via Čap--Slovák--Souček's \cite{capslovaksoucek} (curved) BGG-complexes. The notion of a graded Rockland complex is more general than that of a Rockland complex and arises from internal gradings in the bundles $E_j$. For a curved BGG-complex to be Rockland, and not just graded Rockland, all the bundles $E_0,E_1,\ldots, E_n$ need to be constantly graded, corresponding to \( \mathfrak{t}_H X \) having pure cohomology groups \cite[Section 3.7]{hallergenfive}. This is known to hold for trivially filtered manifolds (where Rockland means elliptic), for contact manifolds, for generic rank-two distributions in dimension five, and for parabolic geometries of the same type as the full complex flag manifold of $SL(3, \C)$ (as implicitly used in \cite{yunckensl3}). We discuss contact manifolds in more detail below in Section \ref{sec:rumin} and generic rank two distributions in dimension five in Example \ref{g2ex}.

For the purpose of completing a Rockland complex into a Hilbert complex, we will henceforth fix a volume density on $X$ and hermitian metrics on all the vector bundles $E_0,E_1,\ldots, E_n\to X$, giving us Hilbert spaces \( L^2(X;E_0), \ldots, L^2(X;E_n) \). By an abuse of notation, we write also $\rd_j$ for the closure of $\rd_j$ as a densely defined operator 
$$\rd_j:L^2(X;E_{j})\dashrightarrow L^2(X;E_{j+1}).$$
The Hilbert complex associated with a Rockland complex $(C^\infty(X;E_\bullet),\rd_\bullet)$ is given by 
$$
 \quad 0\to L^2(X;E_0)\xrightarrow{\rd_0}L^2(X;E_1)\xrightarrow{\rd_1}\cdots \xrightarrow{\rd_{n-1}}L^2(X;E_n)\to 0.$$
 
\begin{thm}
\label{ljnknknbgs}
Assume that $(C^\infty(X;E_\bullet),\rd_\bullet)$ is a Rockland complex where all differentials are differential operators and \(X\) is compact. The Hilbert complex associated with a Rockland complex $(C^\infty(X;E_\bullet),\rd_\bullet)$ of order $\VEC{m}=(m_1,\ldots, m_n)$ is a regular Hilbert complex with discrete spectrum of order $\VEC{m}$ over $C^\infty(X)$. In particular, with a Rockland complex we can associate the $f$-summable ST\textsuperscript{2} $(C^\infty(X),L^2(X;\oplus_j E_j),\DC)$ where $\DC=(\rd_j+\rd_j^*)_{j=0}^{n-1}$ and 
$$f(\VEC{t}) > \mathrm{min}_j \frac{m_j\dim_h(X)}{t_j} . $$
\end{thm}

\begin{proof}
The result will, upon checking the definition, follow from Theorem \ref{st2fromhilb}. Chasing through the definitions, we see that the Hilbert complex associated with a Rockland complex is a regular Hilbert complex over $C^\infty(X)$ as soon as the Rumin Laplacians are hypoelliptic of order $2m$. Indeed, if this is the case then, since the Rumin Laplacians additionally are even order differential operators, \cite[Theorem 2]{davehallerheat} implies that $(\Delta_{\VEC{m},i}^R)^{\frac{\beta}{2m}}\in  \Psi^\beta_H(X,E_i)$. The Theorem follows from order considerations in the Heisenberg calculus. The Rumin Laplacians are hypoelliptic by \cite[Lemma 2.14]{davehaller}.

We note the finite summability statement follows from \eqref{weyldave}. Indeed, if we set $\delta(\VEC{t}):=\mathrm{min}_j \frac{m_j\dim_h(X)}{t_j}$ the interpolation as in Lemma \ref{lknlknn} and  \eqref{weyldave} implies that $\mu_k((1+\Delta^{\VEC{t}}_\DC)^{-1})=O(k^{-\delta(\VEC{t})})$ as $k\to +\infty$. In particular, $(1+\Delta^{\VEC{t}}_\DC)^{-1}\in \mathcal{L}^p$ for any $p>\delta(\VEC{t})$.
\end{proof}

\begin{remark}
\label{p25rem}
In the construction of Theorem \ref{ljnknknbgs} we group together the differentials in the easiest way possible, following Theorem \ref{st2fromhilb}. We can in general group together the differentials more efficiently, e.g.~below in Section \ref{sec:rumin} when studying the Rumin complex on a contact manifold we will group together the differentials into only two self-adjoint operators. If $(C^\infty(X;E_\bullet),\rd_\bullet)$ is a Rockland complex of order $\VEC{m}=(m_0,\ldots, m_{n-1})$, we can consider a partition 
$$\{0,\ldots, n-1\}= \bigsqcup_{l=1}^{n_0} S_l,$$
such that $m_i=m_{j}$ whenever $i$ and $j$ belong to the same set $S_l$. Then the collection $\tilde{\DC}:=(\sum_{j\in S_l}\rd_j+\rd_j^*)_{l=1}^{n_0}$ also fits into an ST\textsuperscript{2} $(C^\infty(X),L^2(X;\oplus_j E_j),\tilde{\DC})$. The bounding matrix $\tilde{\VEC{\epsilon}}=(\tilde{\epsilon}_{lk})_{l,k=1}^{n_0}$ for $\tilde{\DC}$ is similar to \eqref{alknalkjdnajlkdn} and is given by
$$
\tilde{\epsilon}_{lk}:=\begin{cases} 
\frac{m_{i}-1}{m_{j}} \; &\mbox{if there are $i\in S_l$ and $j\in S_k$ with $|j-i|\leq1$},\\
0 \; &\mbox{otherwise}. \end{cases}
$$
\end{remark}

\begin{example}
\label{g2ex}
Let us describe the Rockland complex constructed from the BGG-complex on a generic rank two distribution in dimension five, i.e.~a parabolic geometry of type $(G_2,P)$ where $G_2$ is the split real form of the indicated exceptional Lie group and $P$ the maximal parabolic subgroup corresponding to the shorter simple root. We aim only at describing the overall structure and refer the details to \cite[Example 4.21]{davehaller} (see also Example 4.24 in the arXiv version of \cite{davehaller} and further computational details in its appendix). Let $X$ be a five dimensional manifold filtered by a generic rank two distribution throughout the example. We also fix a finite-dimensional representation $V$ of $G_2$. The BGG-complex of $X$ looks like 
\begin{multline*}
0\to C^\infty(X;E_0)\xrightarrow{\rd_0}C^\infty(X;E_1)\xrightarrow{\rd_1}C^\infty(X;E_2) \\ \xrightarrow{\rd_2} C^\infty(X;E_3)
\xrightarrow{\rd_3} C^\infty(X;E_4)\xrightarrow{\rd_4}C^\infty(X;E_5)\to 0,
\end{multline*}
where $E_j\to X$ is a bundle induced from the parabolic structure and the cohomology group $H^j(\mathfrak{p}_+,V)$. The BGG-complex is by \cite{davehaller} a Rockland sequence of order 
$$\VEC{m}=(1,3,2,3,1).$$
To understand the principal symbol structure of the BGG-complex of $X$, one uses the fact that $X$ locally admits filtered charts modelled on the nilpotent chart $\overline{N}\subseteq G_2/P$ arising from the open, dense Bruhat cell $\overline{N}MAN\subseteq G_2$. In these charts, \cite[Example 4.21]{davehaller} explicitly describes $\sigma^{m_j}_H(\rd_j)$ in terms of elements of the universal enveloping Lie algebra of $\overline{N}$.

In this example, we have the bounding matrix 
$$\VEC{\epsilon}=
\begin{pmatrix}
0&0&0&0&0\\
2&\frac{2}{3}&1&0&0\\
0&\frac{1}{3}&\frac{1}{2}&\frac{1}{3}&0\\
0&0&1&\frac{2}{3}&2\\
0&0&0&0&0
\end{pmatrix}$$
and the associated weighted digraph takes the form 
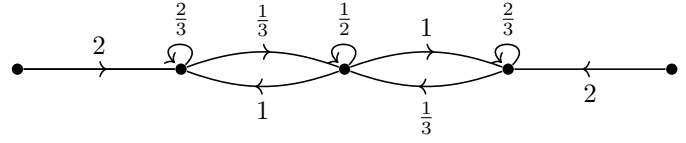
\[ \vcenter{\hbox{\begin{tikzpicture}
	\graph[grow right sep=2cm, empty nodes, nodes={fill=black, circle, inner sep=1.5pt}, edges={semithick}]{
	a -- b --["$1$"' inner sep=5.5pt, middlearrow={<}, bend right=20] 
	c --["$\frac{1}{3}$"' inner sep=5.5pt, middlearrow={<}, bend right=20] 
	d --["$2$"' inner sep=5.5pt, middlearrow={<}] e;
	d --["$1$"' inner sep=5.5pt, middlearrow={<}, bend right=20] 
	c --["$\frac{1}{3}$"' inner sep=5.5pt, middlearrow={<}, bend right=20] 
	b --["$2$"' inner sep=5.5pt, middlearrow={<}] a;
	b --["$\frac{2}{3}$"', loop] b;
	c --["$\frac{1}{2}$"', loop] c;
	d --["$\frac{2}{3}$"', loop] d;
	};
\end{tikzpicture}}} \; . \]
\end{example}

We now turn to discussing two special cases of Theorem \ref{ljnknknbgs}. We produce two higher order spectral triples from Rockland complexes, the first with an $H$-elliptic Heisenberg pseudodifferential operator and the second with a differential operator.

\begin{cor}
	\label{thm:rockland_H_ellipitc_host}
	Let $(C^\infty(X;E_\bullet),\rd_\bullet)$ be a Rockland complex where all differentials are differential operators.
	For any order $\tau>0$, we can form the $(\frac{\dim_h(X)}{\tau},\infty)$-summable higher order spectral triple $(C^\infty(X),L^2(X;\oplus_j E_j),\overline{D}_\tau)$ from the $H$-elliptic Heisenberg operator 
$$\overline{D}_\tau:=\overline{D}_{\VEC{t}_{\VEC{m}}(\tau)}\equiv\sum_{i=0}^{n-1} \rd_i(\Delta_{\VEC{m},i}^R)^{\frac{\tau-m_i}{2m}}+\rd_i^*(\Delta_{\VEC{m},i+1}^R)^{\frac{\tau-m_i}{2m}}\in \Psi^\tau_H(X,\oplus_j E_j).$$
\end{cor}

Any $H$-elliptic Heisenberg operator of an order $\tau>0$ defines a $(\frac{\dim_h(X)}{\tau},\infty)$-summable higher order spectral triple, so Corollary \ref{thm:rockland_H_ellipitc_host} follows from the construction implying that $\overline{D}_\tau$ is $H$-elliptic of order $\tau>0$. On the other hand, Theorem \ref{thm:st2_give_host} together with Theorem \ref{ljnknknbgs} implies the next Corollary which allows us to construct from a Rockland complex a higher order spectral triple with a differential operator as its Dirac operator.

\begin{cor}
If $(C^\infty(X;E_\bullet),\rd_\bullet)$ is a Rockland complex where all differentials are differential operators, there exist odd integers $\VEC{k}=(2k_j+1)_j \in \Omega(\VEC{\epsilon}) \cap (2\N+1)^n$ so that the differential operator 
$$\overline{D}_{\VEC{k}}:=\sum_{i=0}^{n-1} D_i^{2k_i+1}:C^\infty(X;\oplus_jE_j)\to C^\infty(X;\oplus_jE_j)$$
defines a higher order spectral triple $(C^\infty(X),L^2(X;\oplus_j E_j),\overline{D}_{\VEC{k}})$.
\end{cor}

We end this subsubsection by describing the $K$-homology class associated with a Rockland complex via Corollary \ref{khomfromadk} and Theorem \ref{ljnknknbgs}.

\begin{thm}
\label{Kclass}
Assume that $(C^\infty(X;E_\bullet),\rd_\bullet)$ is a Rockland complex where all differentials are differential operators and write $(C^\infty(X),L^2(X;\oplus_j E_j),\DC)$ for its associated ST\textsuperscript{2} graded by $L^2(X;\oplus_j E_j)=L^2(X;\oplus_j E_{2j})\oplus L^2(X;\oplus_j E_{2j+1})$. Take a $\VEC{t}\in \Omega(\VEC{\epsilon})$. The class of the higher order spectral triple $(C^\infty(X),L^2(X;\oplus_j E_j),\overline{D}_{\VEC{t}})$ in $K_0(X)$ coincides with the class $[\rd_\bullet]\in K_0(X)$ as defined in \cite{goffengkuzmin}.
\end{thm}

\begin{proof}
The class $[\rd_\bullet]\in K_0(X)$ as defined in \cite{goffengkuzmin} was defined by order reduction. If we use the Rumin--Seshadri Laplacians to define order reduction, a short algebraic manipulation shows that $|\overline{D}_{\tau=1}|$ lifts the Fredholm module defining the class  $[\rd_\bullet]\in K_0(X)$ to a bounded perturbation of $\overline{D}_{\tau=1}$. 
\end{proof}

\subsubsection{Equivariance in Rockland complexes}

We now turn to studying conformal equivariance of Rockland complexes. 

\begin{definition}
\label{equirock}
Assume that $(C^\infty(X;E_\bullet),\rd_\bullet)$ is a Rockland complex and that $G$ is a locally compact group acting by filtered diffeomorphisms on $X$ and that $E_0,\ldots, E_n$ are $G$-equivariant. We say that $(C^\infty(X;E_\bullet),\rd_\bullet)$ is a $G$-equivariant Rockland complex if the symbol complex $\sigma_H(\rd_\bullet)$ (see \eqref{eq:symbol_seq_of_rockland_complex}) is $G$-equivariant. 

If $(C^\infty(X;E_\bullet),\rd_\bullet)$ is a $G$-equivariant Rockland complex with each $E_j$ an hermitian vector bundle, we say that the $G$-action is a conformal $G$-action on $(C^\infty(X;E_\bullet),\rd_\bullet)$ if for any $j$ the $G$-representation $V_j:G\to \mathrm{GL}(L^2(X;E_j))$ defined from the $G$-action on $E_j\to X$ there is a function $\lambda_{j,g}\in C^\infty(X,\R_{>0})$ such that
$$V_{j,g}V_{j,g}^* = \lambda_{j,g}^2 .$$
The associated unitary representations are
$$U_j:G\to U(L^2(X;E_j)) \quad U_{j,g} = \lambda_{j,g}^{-1} V_{j,g} $$
and we observe that \( U_{j+1, g} \rd_j U_{j, g}^* = \lambda_{j+1,g}^{-1} \rd_j \lambda_{j,g} \).
\end{definition}

\begin{prop}
	Assume that $(C^\infty(X;E_\bullet),\rd_\bullet)$ is a Rockland complex of order $\VEC{m}=(m_1,\ldots, m_n)$, where all differentials are differential operators and \(X\) is compact, with a conformal action of $G$. For \( \VEC{t} \in \Omega(\VEC{\epsilon}) \), the higher order spectral triple
	\[ (C^\infty(X),L^2(X;\oplus_j E_j), \overline{D}_{\VEC{t}}) \]
	is conformally \( G \)-equivariant with conformal factor
	\[ \mu_g = \bigoplus_{j=0}^n P_{\mathscr{H}_i} + P_{\Ran \rd_{j-1}} (\lambda_{j,g}^{-1} \lambda_{j-1,g})^{t_{j-1}/2} P_{\Ran \rd_{j-1}} + P_{\Ran \rd_j^*} (\lambda_{j+1,g}^{-1} \lambda_{j,g})^{t_j/2} P_{\Ran \rd_j^*} . \]
\end{prop}
\begin{proof}
	Because \( \lambda_{j,g} \) is nonvanishing, bounded, and positive, \( \mu_g \) is invertible and positive. Indeed,
	\[ \mu_g \geq \bigoplus_{j=0}^n P_{\mathscr{H}_i} + \| \lambda_{j,g} \lambda_{j-1,g}^{-1} \|_{\infty}^{-t_{j-1}/2} P_{\Ran \rd_{j-1}} + \| \lambda_{j+1,g} \lambda_{j,g}^{-1} \|^{-t_j/2} P_{\Ran \rd_j^*} . \]
	One can check that the difference
	\begin{align*}
		& U_g \rd_j (\rd_j^* \rd_j)^{-1+t_j} U_g^* - \mu_g \rd_j (\rd_j^* \rd_j)^{-1+t_j} \mu_g^* \\
		& = U_g \rd_j (\rd_j^* \rd_j)^{-1+t_j} U_g^* - P_{\Ran \rd_j} (\lambda_{j+1,g}^{-1} \lambda_{j,g})^{t_j/2} \rd_j (\rd_j^* \rd_j)^{-1+t_j} (\lambda_{j+1,g}^{-1} \lambda_{j,g})^{t_j/2} P_{\Ran \rd_j^*} \\
		& = P_{\Ran \rd_j} \left( U_g \rd_j (\rd_j^* \rd_j)^{-1+t_j} U_g^* - (\lambda_{j+1,g}^{-1} \lambda_{j,g})^{t_j/2} \rd_j (\rd_j^* \rd_j)^{-1+t_j} (\lambda_{j+1,g}^{-1} \lambda_{j,g})^{t_j/2} \right) P_{\Ran \rd_j^*}
	\end{align*}
	and the commutator
	\[ [\rd_j (\rd_j^* \rd_j)^{-1+t_j}, \mu_g] =  P_{\im \rd_j} [\rd_j (\rd_j^* \rd_j)^{-1+t_j}, (\lambda_{j+1,g}^{-1} \lambda_{j,g})^{t_j/2}] P_{\Ran \rd_j^*} \]
	are of lower order, since $\rd_j (\rd_j^* \rd_j)^{-1+t}$ belongs to the Heisenberg calculus by Lemma \ref{alknaljdna}, as required by Definition \ref{def:conf_geq_host}.
\end{proof}

An undesirable feature of the above construction is that the conformal factors are not functions on \( X \). Under some circumstances, this can be remedied but then only for certain \( \VEC{t} \in \Omega(\VEC{\epsilon}) \).

\begin{prop}
\label{conaofnadin}
	Assume that $(C^\infty(X;E_\bullet),\rd_\bullet)$ is a Rockland complex of order $\VEC{m}=(m_0,m_1,\ldots, m_{n-1})$, where all differentials are differential operators and \(X\) is compact, with a conformal action of $G$. Suppose that, for some \( \VEC{s} \in \Omega(\VEC{\epsilon}) \), 
	\[  \lambda_{j-1,g}^{s_{j-1}} \lambda_{j+1,g}^{s_j} = \lambda_{j,g}^{s_j + s_{j-1}} \]
	for all \( j = 1, \ldots, n \). Then, for all \( \tau > 0 \), the higher order spectral triple
	\[ (C^\infty(X),L^2(X;\oplus_j E_j), \overline{D}_{\tau \VEC{s}}) \]
	is conformally \( G \)-equivariant with conformal factor
	\[ \mu_g = (\lambda_{1,g}^{-1} \lambda_{0,g})^{\tau s_0} = \cdots = (\lambda_{n,g}^{-1} \lambda_{n-1,g})^{\tau s_{n-1}} . \]
\end{prop}

\begin{remark}
If we can take $\VEC{s}=\VEC{m}$, in the situation of Proposition \ref{conaofnadin}, i.e.~if 
	\[ \lambda_{j-1,g}^{m_j} \lambda_{j+1,g}^{m_{j-1}} = \lambda_{j,g}^{m_{j-1} + m_j} , \]
	the higher order spectral triple $(C^\infty(X),L^2(X;\oplus_j E_j),\overline{D}_\tau)$ of order $\tau>0$ defined from the $H$-elliptic Heisenberg operator $\overline{D}_\tau$ (as in Corollary \ref{thm:rockland_H_ellipitc_host}), is a conformally $G$-equivariant higher order spectral triple with conformal factor
	\[ \mu_g = (\lambda_{1,g}^{-1} \lambda_{0,g})^{\tau/m_0} = \cdots = (\lambda_{n,g}^{-1} \lambda_{n-1,g})^{\tau/m_{n-1}} . \]
	We will see that this in fact does occur for the Rumin complex on a CR-manifold, in Theorem \ref{ruminconaofnadin}.
\end{remark}

\begin{remark}
\label{hosthighrank2}
In the next section we provide further context for conformally equivariant Rockland complexes by studying the Rumin complex on a contact manifold. It would be interesting to include further examples of Rockland complexes, especially in higher rank parabolic geometries. As work by Yuncken \cite{yunckensl3} and Voigt--Yuncken \cite{voigtyuncken} showcases, the interesting aspect lies in the equivariance properties. However, the approach above cannot produce conformally equivariant noncommutative geometries with nontrivial index theory, or even equivariant Fredholm modules, for a semisimple Lie group $G$ of real rank $>1$. Indeed, if $G$ is a higher rank semisimple Lie group and and $T$ is an $H$-elliptic operator on $G/P$ (for some parabolic subgroup $P\subseteq G$) of order $m\geq 0$ commuting with $G$ up to lower order terms then Puschnigg rigidity \cite{puschhigh} implies that $\sigma_H^m(T)$ is positive and that $T$ defines the trivial equivariant $K$-homology class; cf. Remark \ref{hosthighrank}. For $SL(3,\C)$, as studied in \cite{voigtyuncken,yunckensl3}, the BGG-complex is an equivariant Rockland complex (in the sense of Definition \ref{equirock}) but it is not conformally equivariant. The same statement holds for the BGG-complex of $G_2/P$, see Example \ref{g2ex} above.

A separate but equally serious issue at play, as discussed in \cite{davehaller,hallergenfive}, is that a BGG-complex is frequently not a Rockland complex but only a graded Rockland complex. The BGG-complex of a parabolic geometry is Rockland in the usual sense only when the cohomology of the osculating nilpotent group in each fibre has pure cohomology; see \cite[\S 3.7]{hallergenfive} for more details. For index theory purposes \cite{goffeng24}, the graded Rockland situation works well but it is less clear how to do spectral noncommutative geometry with graded Rockland complexes. The BGG-complex arising from the quaternionic contact structure on \( S^{4n-1} \) \cite[\S 3]{Julg_1995a} \cite[(66--67)]{Rumin_2005} is an example which fails to be ungraded Rockland but for which the action of \( Sp(n, 1) \) is conformal, in a sense made clear in \cite{julghow}. In particular, the two issues of conformally equivariant geometries and representing geometries by ungraded Rockland complexes are quite distinct.
\end{remark}

\section{The Rumin complex on contact manifolds}
\label{sec:rumin}

In this section we will look at an explicit example of a Rockland complex, namely the Rumin complex on a contact manifold. We will show that the naïve way of constructing a spectral triple from the Rumin complex does not work. However, using our construction with tangled spectral triples we obtain higher order spectral triples as in Theorem~\ref{thm:st2_give_host}. Lastly, we will look at conformal equivariance under CR-automorphisms when the manifold has an almost CR-structure.

Let $X$ be a $(2n+1)$-dimensional contact manifold with contact structure $H\subseteq TX$, that is, a filtered manifold of depth \(2\) such that \(H:=  T^{-1}X = \ker \theta\) for a \(1\)-form \(\theta\) satisfying that \(\rd\theta|_H\) is non-degenerate. 
Note that for simplicity we are assuming that the contact structure is cooriented, i.e.~that there exists a \emph{global} $1$-form $\theta$ such that $H=\ker\theta$ and not just a locally defined $1$-form. The nondegeneracy of \(\rd\theta|_H\) is equivalent to the top-degree form $\theta\wedge (\rd \theta)^{n+1}$ being a volume form. 

The Rumin complex \cite{rumincomp} is a special case of BGG-complexes \cite{capslovaksoucek,davehaller,goffeng24} that we explain in some more detail. The Rumin complex can be explicitly constructed by splicing together the quotient complex \(\Omega^*/\II^*\) and the subcomplex \(\JJ^*\) of the de Rham complex \(\Omega^*=C^\infty(X;\bigwedge^*T^*X)\) of \(X\), where $\II^*$ is the differential ideal of forms of the form $\theta \wedge \alpha + \rd\theta \wedge \beta$ and $\JJ^*$  is the differential ideal of forms $\alpha$ such that $\theta \wedge \alpha = \rd\theta \wedge \alpha = 0$. These complexes are spliced together using the map \(D_R\colon \Omega^n /\II^n \to \JJ^{n+1}\) defined as \(\alpha\mapsto \rd \tilde{\alpha}\) where \(\tilde{\alpha}\) is any lift of \(\alpha\) such that \(\rd \tilde{\alpha}\in \theta \wedge \Omega^*\). In particular, the cohomology of the Rumin complex
\[
    0 \to \Omega^0/\II^0  \xrightarrow{\rd} \cdots  \xrightarrow{\rd} \Omega^n/\II^n \xrightarrow{D_R} \JJ^{n+1} \xrightarrow{\rd} \cdots  \xrightarrow{\rd} \JJ^{2n+1} \to 0
\]
coincides with the de Rham cohomology \cite{rumincomp}. We call \(D_R\) the Rumin differential, which in fact is a second order differential operator as will be evident later in \eqref{eq:rumin_differential}.

Following \cite{Rumin_2000} we can obtain a different description of the Rumin complex as follows. Let us fix a contact form \( \theta \) and choose a Riemannian metric \( \mathbf{g} \) on $X$. We require that these be compatible, in the sense that \( H \) is orthogonal to the Reeb field, the (unique) vector field $Z$ such that \( \theta(Z) = 1 \) and \( \iota_Z(d \theta) = 0 \).
With our choice of metric, we have an orthogonal splitting
\[
    T^*X=H^*\oplus H^\perp
\]
defined from the contact coorientation $\theta$ spanning $H^\perp$. The exterior derivative takes the form 
\[
	\rd = \begin{pmatrix} \rd_H & L \\ \LL_Z & -\rd_H \end{pmatrix} \text{ in the splitting } \wedge^* T^* X = \wedge^* H^* \bigoplus H^\perp\otimes  \wedge^* H^*.
\]
Here \(\LL_Z\) denotes the Lie derivative along the Reeb field \(Z\) and \(L\) denotes exterior multiplication with \(\rd\theta\). We note that \(\JJ^{k+1} = C^\infty(X; H^\perp \otimes F_k)\) where \(F_k=\ker L \cap \bigwedge^k H^*\) and each element in \(\Omega^k/\II^k\) has a unique representative in \(C^\infty(X; E_k)\) where \(E_k=(\Ran L)^\perp \cap \bigwedge^k H^*\). With this, the Rumin complex takes the form
\begin{multline*}
        0 \xrightarrow{~} C^\infty(X;E_0) \xrightarrow{P_{E_1}\rd_H} C^\infty(X;E_1) \xrightarrow{P_{E_2}\rd_H} \cdots \\ 
        \cdots \xrightarrow{P_{E_n}\rd_H} C^\infty(X;E_n) \xrightarrow{D_R} C^\infty(X;H^\perp \otimes F_{n}) \xrightarrow{-\rd_H} \cdots  \\
        \cdots \xrightarrow{\!\!-\rd_H} C^\infty(X;H^\perp \otimes F_{2n-1}) \xrightarrow{\!\!-\rd_H} C^\infty(X;H^\perp \otimes F_{2n}) \xrightarrow{~} 0 
\end{multline*}
where the Rumin differential $D_R$ can be expressed as the {\bf second} order differential operator
\begin{equation}
    \label{eq:rumin_differential}
    D_R = \theta \wedge (\LL_Z + \rd_H L^{-1} \rd_H).
\end{equation}
Note that \(L\colon C^\infty(X; \bigwedge^k H^*) \to C^\infty(X; \bigwedge^{k+1} H^*) \) is injective for \(k \leq n-1\) and surjective for \(k \geq n-1\) \cite{rumincomp}, which is utilized to show that \(D_R\) is well defined.

We shall write the Rumin complex as $\rd_\bullet^R$. This is a mixed order differential complex. 

\begin{lemma}
The Rumin complex $(C^\infty(X;E_\bullet),\rd_\bullet^R)$ on a cooriented contact manifold $X$ is a Rockland complex.
\end{lemma}

It is well known that $\rd_\bullet^R$ is a Rockland complex \cite{julgkasparov,rumincomp} and a detailed discussion thereof can be found in Example 4.21 of the arXiv version of \cite{davehaller}. Let us nevertheless hint at how the argument goes. 

\begin{proof}[Sketch of proof]
	When localizing the Rumin complex to its symbol complex, it will be the same as the Rumin complex on the Heisenberg group \(\sfH_{2n+1}\) but with sections in $\SS_0$ instead. It therefore suffices to show that this complex is exact. To do this one can by hand prove the following Poincaré lemma: if \(\rd \alpha = 0\) for \(\alpha\in \SS_0(\sfH_{2n+1};\bigwedge^k \R)\) there is an \(\beta\in \SS_0(\sfH_{2n+1};\bigwedge^{k-1} \R)\) such that \(\rd \beta = \alpha\). Let  \(\Omega_{\SS}^* = \SS_0(\sfH_{2n+1};\bigwedge^* \R)\) and define \(\JJ_\SS^*\) and \(\II_\SS^*\) analogously as when we constructed the Rumin complex.

    To show exactness at \(\Omega_\SS^k/\II_\SS^k\) for \(k<n\), take any \(\alpha\in \Omega_\SS^k/\II_\SS^k\) such that \(\rd \alpha = 0\). What we mean with this is that for any representative \(\tilde{\alpha}\) of \(\alpha\) we have that \(\rd\tilde{\alpha} \in \II_\SS^{k+1}\). Fix a representative \(\tilde{\alpha}\in \Omega_\SS^k\) of \(\alpha\) and let \(\gamma\in \SS(\sfH_{2n+1};\bigwedge^{k-1} H^*)\) and \(\delta\in \SS(\sfH_{2n+1};\bigwedge^k H^*)\) be such that \(\rd\tilde{\alpha} = \rd\theta \wedge \gamma + \theta\wedge \delta\). Now, \(\tilde{\alpha}'\coloneq \alpha - \theta \wedge \gamma\) is also a representative of \(\alpha\) and \(\rd \tilde{\alpha}' = \theta\wedge(\delta + \rd\gamma)\in \theta\wedge \Omega_\SS^*\), so \(\rd\theta\wedge \rd\tilde{\alpha}' = \rd\:(\theta\wedge \rd\tilde{\alpha}') = 0\). Noting that \(\delta + \rd\gamma = \delta + \rd_H\gamma \in \SS(\sfH_{2n+1};\bigwedge^k H^*)\) we obtain by injectivity of \(L\) that \(\rd\tilde{\alpha}'= 0\). Hence, there is an \(\beta\in \Omega^{k-1}\) such that \(\rd\beta=\tilde{\alpha}'\) and the image of \(\beta\) in \(\Omega_\SS^{k-1}/\II_\SS^{k-1}\) is mapped to \(\alpha\) by \(\rd\).

    To show exactness at \(\Omega_\SS^n/\II_\SS^n\), take any \(\alpha\in \Omega_\SS^n/\II_\SS^n\) such that \(D_R \alpha = 0\). Then there is a representative \(\tilde{\alpha}\in \Omega_\SS^n\) of \(\alpha\) such that \(\rd\tilde{\alpha} = 0\) and hence there is a \(\beta \in \Omega_\SS^{n-1}\) such that \(\rd\beta = \tilde{\alpha}\). Now, the image of \(\beta\) in \(\Omega_\SS^{n-1}/\II_\SS^{n-1}\) is mapped to \(\alpha\) by \(\rd\). 
	
	To show exactness at \(\JJ_\SS^{n+1}\), take any \(\alpha\in \JJ_\SS^{n+1}\) such that \(\rd \alpha = 0\). Then there is a \(\beta \in \Omega_\SS^n\) such that \(\rd\beta = \alpha\). Since \(\alpha \in \theta \wedge \Omega_\SS^*\), the image of \(\beta\) in \(\Omega_\SS^n/\II_\SS^n\) is mapped to \(\alpha\) by \(D_R\), and we obtain exactness at \(\JJ_\SS^{n+1}\). To show exactness at \(\JJ_\SS^k\) for \(k>n+1\), take any \(\alpha\in \JJ_\SS^k\) such that \(\rd \alpha = 0\), then there is a \(\beta \in \Omega_\SS^{k-1}\) such that \(\rd\beta = \alpha\). By surjectivity of \(L\), we can find \(\gamma\in \SS(\sfH_{2n+1}; \bigwedge^{k-3} H^*)\) and \(\delta\in \SS(\sfH_{2n+1};\bigwedge^{k-2} H^*)\) such that \(\beta = \rd\theta \wedge \gamma + \theta\wedge \delta\). Now, \(\beta' \coloneq \theta \wedge(\delta + \rd\gamma)\) satisfies that \(\rd \beta' = \alpha\) and since \(\alpha \in \theta \wedge \Omega_\SS^*\) we have that \(\rd\theta\wedge \beta' = \rd\:(\theta \wedge \beta') - \theta \wedge \rd\beta' = 0\), so \(\beta'\in \JJ_\SS^{k-1}\) and we obtain exactness at \(\JJ_\SS^k\).
\end{proof}

Let us describe the symbol complex of the Rumin complex in some more detail. We do the same procedure as for the Rockland sequences in \eqref{eq:symbol_seq_of_rockland_complex} and identify $T_HX_x\cong \sfH_{2n+1}$ with the Heisenberg group via Darboux coordinates for each point \(x\in X\). Write $X_1,\ldots, X_n,Y_1,\ldots, Y_n,Z$ for the standard generators of $\mathfrak{h}_{2n+1}$ with $[X_i,Y_j]=\delta_{ij}Z$ corresponding to the Darboux coordinates near $x$. We will identify the fibres \(E_{k,x}\) and \(H^\perp_x\otimes F_{k,x}\) with subspaces of \(\bigwedge^kH^*_x=\bigwedge^k \R^{2n}\) and \(H_x^\perp\otimes \bigwedge^kH^*_x=\bigwedge^k \R^{2n}\) respectively. Consider the $\mathfrak{h}_{2n+1}$-valued vector 
\[
    \omega_1 =
\begin{pmatrix} 
X_1 & \ldots & X_n &
 Y_1 & \ldots & Y_n\end{pmatrix}^T\in \mathfrak{h}_{2n+1}\otimes H_x^*.
\]
We can express the principal symbols of the differentials in the Rumin complex as
\begin{align*}
\sigma_H^1(\rd_j^R)_x & = \omega_1\wedge: \mathcal{S}_0(\sfH_{2n+1},E_{j,x})\to \mathcal{S}_0(\sfH_{2n+1},E_{j+1,x}) & (j<n), \\
\sigma_H^2(\rd_j^R)_x & = \sigma_H^2(D_R)_x \\
& = Z + (\omega_1\wedge) L^{-1} (\omega_1\wedge): \mathcal{S}_0(\sfH_{2n+1},E_{n,x})\to \mathcal{S}_0(\sfH_{2n+1},F_{n,x}) \; &(j=n), \\
\sigma_H^1(\rd_j^R)_x & = -\omega_1\wedge: \mathcal{S}_0(\sfH_{2n+1},F_{j-1,x})\to \mathcal{S}_0(\sfH_{2n+1},F_{j,x}) \; &(j>n).
\end{align*}
In the case of \(n=1\), we can identify 
$$E_{0,x}=\C, \; E_{1,x}=F_{1,x}=\C^2, \; \mbox{and}\;F_{2,x}=\C.$$
Under these identifications, we have that \(\sigma_H^2(D_R) = Z + \omega_1 \omega_2^*\) where
\[ J = \begin{pmatrix} 0 & -1 \\ 1 & 0 \end{pmatrix} \quad \text{and} \quad \omega_2 = J\omega_1 =
\begin{pmatrix} -Y \\ X \end{pmatrix} . \]
The symbol complex over \(x\) takes the form
\[
0 \xrightarrow{} \SS_0(\sfH_3) \xrightarrow{\mbox{\footnotesize\( \begin{pmatrix} X \\ Y \end{pmatrix}\)}} \SS_0(\sfH_3)\otimes \C^2 \xrightarrow{\mbox{\footnotesize\( \begin{pmatrix} Z + XY & \!\!\!\! -X^2 \\ Y^2 & \!\!\!\! Z - YX \end{pmatrix} \)}} \SS_0(\sfH_3)\otimes \C^2 \xrightarrow{\mbox{\footnotesize\( \begin{pmatrix} Y & \!\!\!\! -X \end{pmatrix}\)}} \SS_0(\sfH_3) \xrightarrow{} 0 .
\]
The reader can compare this to the BGG-complex studied by Yuncken \cite{yunckensl3} and the example in Subsection \ref{rubsecrumin3d}.

\subsection{A naïve approach to higher order spectral triples}
\label{subsec:naive}

A first approach to study the noncommutative geometry of the Rumin complex is to naïvely roll up the complex as
$$\slashed{D}^R:=\rd_\bullet^R+(\rd_\bullet^R)^*.$$
Rolling up a complex in this way is how one produces the Hodge--de Rham Dirac operator from the de Rham complex. We shall see that this approach fails to produce a higher order spectral triple, thereby justifying the approach of Section \ref{sec:complexes} and the decreasing cycle condition. 

By discreteness of the spectrum, $\slashed{D}^R$ has compact resolvent. However, taking commutators with $C^\infty(X)$ does not improve the order. We will show this in the case of three-dimensional contact manifolds. In Darboux coordinates, the Rumin complex is up to lower order terms given by
\begin{align*}
\rd^R_\bullet=
\begin{pmatrix}
0& 0& 0& 0& \\
\omega_1 & 0&0& 0 \\
0& Z + \omega_1 \omega_2^*& 0& 0 \\
0& 0& \omega_2^* & 0
\end{pmatrix}.
\end{align*}
Therefore $\slashed{D}^R$ takes the form
\begin{align*}
\slashed{D}^R=
\begin{pmatrix}
0& \omega_1^* & 0& 0& \\
\omega_1 & 0& -Z + \omega_2 \omega_1^* & 0 \\
0& Z + \omega_1 \omega_2^*& 0& \omega_2 \\
0& 0& \omega_2^* & 0
\end{pmatrix}.
\end{align*}

\begin{prop}
\label{thm:commutator_rumin_3d}
Let $X$ be a compact contact manifold of dimension $3$ and $a\in C^\infty(X)$. Then $[\slashed{D}^R,a]$ is up to a vector bundle endomorphism of the form 
\begin{align*}
\begin{pmatrix}
0& 0& 0& 0& \\
0& 0&  \omega_2 \omega_1^*(a)+\omega_2(a)\omega_1^* & 0 \\
0& \omega_1 \omega_2^*(a)+\omega_1(a)\omega_2^* & 0&0 \\
0& 0&0 & 0
\end{pmatrix}.
\end{align*}
in local Darboux coordinates.
\end{prop}

\begin{proof}
Follows from direct computation with the Leibniz rule.
\end{proof}

\begin{prop}
\label{hostissue}
Let $X$ be a compact contact manifold of dimension $3$. If $\alpha\in \R$ satisfies that $[\slashed{D}^R,a](1+(\slashed{D}^R)^2)^{-1/2+\alpha}$ is a bounded operator on $L^2(X;\mathcal{H})$ for any $a\in C^\infty(X)$ then $\alpha\leq 0$.
\end{prop}

In consequence, \( \slashed{D}^R \) does not define a higher order spectral triple for \( C^{\infty}(X) \).

\begin{proof}
We need to show that, for all $\alpha> 0$, $[\slashed{D}^R,a](1+(\slashed{D}^R)^2)^{-1/2+\alpha}$ fails to be bounded on $L^2(X;\mathcal{H})$ for some $a\in C^\infty(X)$. By Proposition \ref{thm:commutator_rumin_3d} and the computations above, $[\slashed{D}^R,a](1+(\slashed{D}^R)^2)^{-1/2+\alpha}$ is bounded if and only if $(\omega_1 \omega_2^*(a)+\omega_1(a)\omega_2^*)(1+T)^{-1/2+\alpha}$ is bounded where \(T = \omega_1\omega_1^* + (-Z + \omega_2 \omega_1^*)(Z + \omega_1 \omega_2^*)\).

Were $(\omega_1 \omega_2^*(a)+\omega_1(a)\omega_2^*)(1+T)^{-1/2+\alpha}$ to be bounded, we could freeze coefficients in a point $x$ and represent this operator in a non-trivial character $\xi\in \R^2 \subseteq \widehat{\sfH}_{2n+1}$ and obtain a uniformly bounded function in $\xi$. For notational simplicity, write 
$$v:=\omega_1(a)_x.$$
In this notation, $\omega_2(a)_x=Jv$. In a character $\xi$, $\omega_1$ is represented as $\xi$, $\omega_2$ is represented as $J\xi$ and \(Z\) is represented as \(0\). Hence, $T$ is represented in the character $\xi\neq 0$ as the matrix valued function 
$$F(\xi)=\xi\xi^*+|\xi|^2(J\xi)(J\xi)^*=|\xi|^2e_1(\xi)+|\xi|^4e_2(\xi),$$
where $e_1(\xi)=|\xi|^{-2}\xi\xi^*$ and $e_2(\xi)=Je_1(\xi)J$ are the orthogonal projections onto the span of $\xi$ and $J\xi$ respectively. Since $J$ is anti-symmetric, $e_1(\xi)$ and $e_2(\xi)$ have orthogonal ranges and $F(\xi)=|\xi|^2 e_1(\xi)+|\xi|^4 e_2(\xi)$ is the eigenvalue decomposition of $F(\xi)$. By the discussion above, we need to show that for $\alpha>0$, boundedness fails for the matrix valued function 
$$A_\alpha(\xi) :=  (\xi (Jv)^*+v (J\xi)^*)(1+F(\xi))^{-1/2+\alpha}.$$
By orthogonality of $e_1(\xi)$ and $e_2(\xi)$, we compute that 
\begin{align*}
A_\alpha(\xi) 
& = (1+|\xi|^2)^{-\frac{1}{2}+\alpha}(\xi (Jv)^*+v (J\xi)^*)e_1(\xi) + (1+|\xi|^4)^{-\frac{1}{2}+\alpha}(\xi (Jv)^*+v (J\xi)^*) e_2(\xi) \\
& = (1+|\xi|^2)^{-\frac{1}{2}+\alpha} ((Jv)^* \xi) e_1(\xi) + O(|\xi|^{-1+4\alpha})
\end{align*}
and see that for $t>0$
$$A_\alpha(tJv)= t(1+t^2)^{-\frac{1}{2}+\alpha} e_1(Jv)+O(t^{-1+4\alpha}).$$
In particular, $A_\alpha$ is bounded if and only if $\alpha\leq 0$ so in particular boundedness fails for $\alpha>0$.
\end{proof}

\subsection{Spectral triples from the Rumin complex}

Let us place the Rumin complex $\rd_\bullet^R$ of a contact manifold in a spectral triple. We have a somewhat simpler structure than seen in Subsection \ref{subsechilbert}, since all but one of the differentials are order one, see Remark \ref{p25rem}. We consider the two self-adjoint operators 
$$D_1:=\sum_{j\neq n}\rd^R_j+(\rd^R_j)^*\quad \mbox{and}\quad D_2:=D_R+(D_R)^*.$$
These are differential operators of order $m_1=1$ and $m_2=2$ respectively. We note that $D_1D_2=D_2D_1=0$ on the common core $C^\infty(X;\oplus_j E_j)$, so $D_1$ and $D_2$ are strictly anticommuting. We compute that  
$$D_1^2=\sum_{j\neq n}\rd^R_j(\rd^R_j)^*+(\rd^R_j)^*\rd^R_j\quad \mbox{and}\quad D_2^2:=D_R(D_R)^*+(D_R)^*D_R.$$
The Rumin Laplacian takes the form
$$\Delta^R=D_1^4+D_2^2.$$
We can proceed as in Subsection \ref{subsecrock} to prove the following.

\begin{prop}
\label{rumst2}
Consider the Rumin complex $\rd_\bullet^R$ on a $2n+1$-dimensional compact contact manifold $X$. Then with $D_1$ and $D_2$ as in the preceding paragraph, the data $(C^\infty(X),L^2(X;\mathcal{H}),(D_1,D_2))$ constitute an ST\textsuperscript{2} with bounding matrix
\[ \VEC{\epsilon}=\begin{pmatrix} 0& 0\\1& \frac{1}{2}\end{pmatrix} \qquad\qquad \vcenter{\hbox{\begin{tikzpicture}
\graph[grow right sep=1.5cm, empty nodes, nodes={fill=black, circle, inner sep=1.5pt}, edges={semithick}]{
    a --["$1$" inner sep=5.5pt, middlearrow={<}] b --["$\frac{1}{2}$"', loop] b;};
\end{tikzpicture}}} , \]
which is $f$-summable for any function $f$ with 
$$f(t_1,t_2)>2\min\left(\frac{n+1}{t_1},\frac{2(n+1)}{t_2}\right).$$
In particular, for any $\VEC{t}=(t_1,t_2)\in (0,\infty)^2$ with 
$t_1>t_2$,
we arrive at a higher order spectral triple defined from the operator 
$$\overline{D}_{\VEC{t}}=D_1|D_1|^{t_1-1}+D_2|D_2|^{t_2-1}=D_1(\Delta^R)^{\frac{t_1-1}{4}}+D_2(\Delta^R)^{\frac{t_2-1}{2}} . $$
If $\VEC{t}$ lies along the ray spanned by $(1,1/2)$ then $\overline{D}_{\VEC{t}}$ is an $H$-elliptic operator in the Heisenberg calculus and if $\VEC{t}=(2k_1+1,2k_2+1)$ where $k_1>k_2$ are natural numbers then $\overline{D}_{\VEC{t}}$ is a differential operator.
\end{prop}

\begin{remark}
We note that
$$d(x,y):= \sup\{|a(x)-a(y)|: \|[D_1,a]\|\leq 1\} = \sup\{|a(x)-a(y)|: \frac{1}{2}\|[[D_2,a],a]\|\leq 1\} $$
and coincides with the Carnot--Carath{\'e}odory distance of $X$. In \cite[\S 3.3]{Hasselmann_2014}, compact quantum metric spaces are built from Carnot manifolds using a `horizontal Dirac operator' similar to \( D_1 \). Related results are found in \cite{GGJNCG}. There is a potential for interesting metric aspects of ST\textsuperscript{2}s to be considered. In this connection, we mention also the work \cite{Kaad_2020, kaadkyed} of Kaad and Kyed which uses a collection of operators for constructing quantum metric spaces.
\end{remark}

\subsection{CR-equivariance}

Recall the setup above: $X$ is a cooriented contact manifold with a fixed contact form \( \theta \) and a Riemannian metric \( \mathbf{g} \) in which the orthogonal complement of $H = \ker \theta$ is the Reeb field. The two-form $\rd\theta$ defines a symplectic form on $H$. An almost CR-structure is the additional datum of a complex structure \( J \) on \( H \) such that $\mathbf{g}(v,w) = \rd\theta(v, Jw)$. Note that the complex structure $J$ and the metric $\mathbf{g}$ uniquely determine one another.

We let $\mathrm{Aut}_{\rm CR}(X)$ denote the group of CR-automorphisms of $X$. That is, the group of diffeomorphisms $g:X\to X$ such that $Dg$ preserves $H$ (i.e.~$(Dg)_xH_x\subseteq H_{g(x)}$ for all $x$) and acts complex linearly on $H$ (i.e.~$(Dg)_x:(H_x,J_x)\to (H_{g(x)},J_{g(x)})$ is complex linear for all $x$). The group of CR-automorphisms is generically a compact subgroup as the following result of Schoen \cite{schoen95} proves.

\begin{thm}[Schoen's Ferrand--Obata theorem]
Let $X$ be a compact cooriented contact manifold with a choice of Riemannian metric as above. The group $\mathrm{Aut}_{\rm CR}(X)$ can equivalently be topologized by its compact-open topology, $C^0$- or $C^\infty$-topology. The group $\mathrm{Aut}_{\rm CR}(X)$ is compact unless $X$ is an odd-dimensional sphere with its round contact structure and metric and in this case $X=SU(n,1)/P$ for the standard parabolic subgroup $P\subseteq SU(n,1)$ and $\mathrm{Aut}_{\rm CR}(X)\cong SU(n,1)$.
\end{thm}

The action of a CR-automorphism has features similar to being conformal. In  \cite{julgkasparov} \cite[\S 3.3.3]{AMsomewhere}, this similarity to conformality has been utilized for $\mathrm{Aut}_{\rm CR}(X)\cong SU(n,1)$ when $X$ is the round odd-dimensional sphere. A contact form for a given contact structure is unique up to multiplication by a nonvanishing smooth function on \( X \). Because the contact structure is preserved by a CR-automorphism \( g \), the pullback \( g^*(\theta) \) of the contact form must be equal to \( f \theta \) for some nonvanishing smooth function \( f \). Hence
\[ g^*(\mathbf{g})(X, Y) = g^*(\rd\theta)(X, J Y) = (f \rd \theta + \rd f \wedge \theta)(X, J Y) = f \rd \theta (X, J Y) = f \mathbf{g}(X, Y) \]
for all \( X, Y \in H \). Moreover, the induced metric on \( TX/H \) is multiplied by \( f^2 \).

We conclude that, if $g\in \mathrm{Aut}_{\rm CR}(X)$, the differential of $g$ lifts to a graded vector-bundle action
$$v_g \colon E_\bullet\to E_\bullet,$$
with 
$$v_g^*v_g = \bigoplus_{k=0}^n \lambda_g^{2k} \oplus \bigoplus_{k=n}^{2n} \lambda_g^{2(k+2)} , $$
in accordance with the grading of \(E_\bullet\). We define an action
$$V:\mathrm{Aut}_{\rm CR}(X)\to \mathrm{GL}(L^2(X;E_\bullet)),\quad V(g)f(x):=v_gf(g^{-1}x).$$ 
Since the volume form belongs to $\wedge^n H^*\otimes H^\perp$ it rescales with $\lambda_g^{2n+2}$ under $g\in \mathrm{Aut}_{\rm CR}(X)$. Therefore
$$V(g)^*V(g)=\lambda_g^{2n+2} v_g^*v_g= \bigoplus_{k=0}^n \lambda_g^{2(k + n + 1)} \oplus \bigoplus_{k=n}^{2n} \lambda_g^{2(k + n + 3)}.$$

The Rumin complex of $X$ is defined from a quotient complex and a subcomplex of the de Rham complex spliced with the Rumin differential. As such, the Rumin complex is invariant under $\mathrm{Aut}_{\rm CR}(X)$. In other words,
$$V(g)\rd_\bullet^RV(g)^{-1}=\rd_\bullet^R.$$
Set $\Lambda_g=V(g)^*V(g)$, so $V(g)^*=\Lambda_gV(g^{-1})$. We conclude that 
$$V(g)^*\rd_\bullet^R V(g)=\Lambda_g \rd_\bullet^R.$$
If we pass to the unitarized action
$$U:\mathrm{Aut}_{\rm CR}(X)\to \mathcal{U}(L^2(X;E_\bullet)),\quad U(g):=V(g)\Lambda_g^{-1/2},$$
we see that
$$U(g)^*\rd_\bullet^RU(g)= \Lambda_g^{1/2} \rd_\bullet^R\Lambda_g^{-1/2}.$$
From Proposition \ref{conaofnadin} we conclude the following.

\begin{thm}
\label{ruminconaofnadin}
Let $(C^\infty(X;E_\bullet),\rd_\bullet^R)$ denote a Rumin complex on a $(2n+1)$-dimensional almost CR-manifold $X$ with its conformal action of $G=\Aut_{\rm CR}(X)$. For $\tau>0$, the $H$-elliptic Heisenberg operator 
$$D_{\tau}=D_1|D_1|^{\tau-1}+D_2|D_2|^{\frac{\tau}{2}-1}=D_1(\Delta^R)^{\frac{\tau-1}{4}}+D_2(\Delta^R)^{\frac{\tau-2}{4}}$$
defines a conformally $G$-equivariant, $(\frac{\tau}{2n+2},\infty)$-summable, order-\( \tau \) spectral triple
$$(C^\infty(X),L^2(X;\oplus_j E_j),D_\tau)$$
with the conformal factor $\mu=\lambda^{-\tau}$.
\end{thm}

\section{Group \texorpdfstring{$C^*$}{C*}-algebras of nilpotent groups}
\label{section:nilpotent}

Let \( G \) be a simply connected, nilpotent Lie group \( G \). As a manifold $G$ is diffeomorphic to \( \R^n \) for some \( n \) and its maximal compact subgroup is trivial. Hence the dual Dirac element, as defined by Kasparov \cite[Section 5]{Kasparov_1988}, is an element of \( KK^G_*(\C, C_0(G)) \). By Baaj–Skandalis duality, there is an isomorphism of the $KK$-groups \( KK^G_*(\C, C_0(G)) \) and \( KK^{\hat{G}}_*(C^*(G), \C) \), where \( \hat{G} \) is the dual quantum group. For a more detailed discussion we refer to \cite[Chapter II]{AMthesis} in which, among other things, spectral triples are built for $C^*$-algebras of groups acting properly on CAT(0) spaces. Nilpotent groups are generally not CAT(0), and we turn to the framework of ST\textsuperscript{2}s.

Recall from Definition \ref{definition:w465ybe57unu467inr4nin6i} that a weight on a locally compact group \( G \) is a continuous function \( \ell \) from \( G \) to matrices on a finite-dimensional complex vector space \( V \). If \( V \) is \( \Z/2\Z \)-graded, \( \ell \) is required to be odd.

\begin{definition}
	\label{definition:collofweights}
	Fixing a finite-dimensional complex vector space \( V \), we will say that a finite collection of weights \( \VEC{\ell} = (\ell_j)_{j\in I} : G \to \End V \) on a locally compact group \( G \) is
	\begin{enumerate}
		\item \emph{self-adjoint} if \( \ell_j^* = \ell_j \) for all \( j \);
		\item \emph{proper} if \( (\ell_j)_{j\in I} \) mutually anticommute and \( \prod_j (1 + |\ell_j|)^{-1} \in C_0(G, \End V) \); and
		\item \emph{translation bounded} with bounding matrix \( \VEC{\epsilon} \in M_n([0,\infty)) \) if, for all \( g \in G \),
			\[ \sup_{h \in G} \Big\| (\ell_i(g h) - \ell_i(h)) \Big(1 + \sum_{j\in I} |\ell_j(h)|^{\epsilon_{i j}} \Big)^{-1} \Big\| < \infty \]
			and there exists a neighbourhood \( U \) of the identity in \( G \) such that, for all \( s \in G \),
			\[ \sup_{g \in U, h \in G} \Big\| (\ell_i(g h) - \ell_i(h)) \Big(1 + \sum_{j\in I} |\ell_j(s h)|^{\epsilon_{i j}} \Big)^{-1} \Big\| < \infty . \]
	\end{enumerate} 
\end{definition}

Note that the second part of the translation-boundedness condition is automatic for a discrete group.
Whether this condition can be simplified in general is unclear. For the time being, we content ourselves with giving two equivalent conditions.

\begin{lemma}
	\label{lemma:srt7n56ne756e5n6-copy}
	Let \( G \) be a locally compact group, \( V \) a finite-dimensional complex vector space, and \( (\ell_j)_{j \in I} : G \to \End V \) a collection of weights. The following are equivalent:
	\begin{enumerate}
		\item For all \( g \in G \),
			\[ \sup_{i \in I, h \in G} \Big\| (\ell_i(g h) - \ell_i(h)) \Big(1 + \sum_{j\in I} |\ell_j(h)|^{\epsilon_{i j}} \Big)^{-1} \Big\| < \infty \]
			and there exists a neighbourhood \( U \) of the identity in \( G \) such that, for all \( s \in G \),
			\[ \sup_{i \in I, g \in U, h \in G} \Big\| (\ell_i(g h) - \ell_i(h)) \Big(1 + \sum_{j\in I} |\ell_j(s h)|^{\epsilon_{i j}} \Big)^{-1} \Big\| < \infty . \]
		\item For every compact subset \( K \subseteq G \),
			\[ \sup_{i \in I, g \in K, h \in G} \Big\| (\ell_i(g h) - \ell_i(h)) \Big(1 + \sum_{j\in I} |\ell_j(h)|^{\epsilon_{i j}} \Big)^{-1} \Big\| < \infty . \]
		\item The functions \( (\zeta_i)_{i \in I} : G \to C(G, \End V) \) given by
			\[ \zeta_i(g)(h) = (\ell_i(g h) - \ell_i(h)) \Big(1 + \sum_{j\in I} |\ell_j(h)|^{\epsilon_{i j}} \Big)^{-1} \]
			are elements of \( C(G, C_b(G, \End V)_{\beta}) \), where \( \beta \) is the strict topology.
	\end{enumerate}
\end{lemma}

\begin{proof}
	Suppose that 1.~holds and let \( K \) be a compact subset of \( G \). The open sets \( (U g)_{g\in K} \) cover \( K \). Let \( U g_1, \ldots, U g_k \) be a finite subcover. We have
	\begin{align*}
		& \sup_{i \in I, g \in K, h \in G} \Big\| (\ell_i(g h) - \ell_i(h)) \Big(1 + \sum_{j\in I} |\ell_j(h)|^{\epsilon_{i j}} \Big)^{-1} \Big\| \\
		& \qquad \leq \sup_{1 \leq r \leq k, i \in I, g \in U g_r, h \in G} \Big\| (\ell_i(g h) - \ell_i(h)) \Big(1 + \sum_{j\in I} |\ell_j(h)|^{\epsilon_{i j}} \Big)^{-1} \Big\| \\
		& \qquad = \sup_{1 \leq r \leq k, i \in I,  g \in U, h \in G} \Big\| (\ell_i(g g_r^{-1} h) - \ell_i(h)) \Big(1 + \sum_{j\in I} |\ell_j(h)|^{\epsilon_{i j}} \Big)^{-1} \Big\| \\
		& \qquad \leq \sup_{1 \leq r \leq k, i \in I,  g \in U, h \in G} \Big( \Big\| (\ell_i(g g_r^{-1} h) - \ell_i(g_r^{-1} h)) \Big(1 + \sum_{j\in I} |\ell_j(h)|^{\epsilon_{i j}} \Big)^{-1} \Big\| \\
		& \qquad\qquad + \sup_{1 \leq r \leq k, i \in I, h \in G} \Big\| (\ell_i(g_r^{-1} h) - \ell_i(h)) \Big(1 + \sum_{j\in I} |\ell_j(h)|^{\epsilon_{i j}} \Big)^{-1} \Big\| \\
%		& \qquad \leq \max_{1 \leq r \leq k} \sup_{g \in U, h \in G} \big( \| \ell(g g_r^{-1} h) - \ell(g_r^{-1} h) \| + \| \ell(g_r^{-1} h) - \ell(h) \| \big) \\
		& \qquad \leq \sup_{1 \leq r \leq k, i \in I, g \in U, h \in G} \Big\| (\ell_i(g h) - \ell_i(h)) \Big(1 + \sum_{j\in I} |\ell_j(g_r h)|^{\epsilon_{i j}} \Big)^{-1} \Big\| \\
		& \qquad\qquad + \sup_{1 \leq r \leq k, i \in I, h \in G} \Big\| (\ell_i(g_r^{-1} h) - \ell_i(h)) \Big(1 + \sum_{j\in I} |\ell_j(h)|^{\epsilon_{i j}} \Big)^{-1} \Big\| \\
%		& \qquad \leq \sup_{g \in U, h \in G} \| \ell(g h) - \ell(h) \| + \max_{1 \leq r \leq k} \sup_{h \in G} \| \ell(g_r^{-1} h) - \ell(h) \| \\
		& \qquad < \infty ,
	\end{align*}
	that is, 2.~is satisfied.
	
	Suppose that 2.~holds and, by the local compactness of \( G \), take an open neighbourhood \( U \) of the identity in \( G \) contained in a compact set \( K \). Then
	\begin{align*}
		& \sup_{i \in I, g \in U, h \in G} \Big\| (\ell_i(g h) - \ell_i(h)) \Big(1 + \sum_{j\in I} |\ell_j(s h)|^{\epsilon_{i j}} \Big)^{-1} \Big\| \\
		& \qquad = \sup_{i \in I, g \in U, h \in G} \Big\| (\ell_i(g s^{-1} h) - \ell_i(s^{-1} h)) \Big(1 + \sum_{j\in I} |\ell_j(h)|^{\epsilon_{i j}} \Big)^{-1} \Big\| \\
		& \qquad \leq \sup_{i \in I, g \in U, h \in G} \Big\| (\ell_i(g s^{-1} h) - \ell_i(h)) \Big(1 + \sum_{j\in I} |\ell_j(h)|^{\epsilon_{i j}} \Big)^{-1} \Big\| \\
		& \qquad\qquad + \sup_{i \in I, h \in G} \Big\| (\ell_i(s^{-1} h) - \ell_i(h)) \Big(1 + \sum_{j\in I} |\ell_j(h)|^{\epsilon_{i j}} \Big)^{-1} \Big\| \\
		& \qquad \leq \sup_{i \in I, g \in K s^{-1}, h \in G} \Big\| (\ell_i(g h) - \ell_i(h)) \Big(1 + \sum_{j\in I} |\ell_j(h)|^{\epsilon_{i j}} \Big)^{-1} \Big\| \\
		& \qquad\qquad + \sup_{i \in I, h \in G} \Big\| (\ell_i(s^{-1} h) - \ell_i(h)) \Big(1 + \sum_{j\in I} |\ell_j(h)|^{\epsilon_{i j}} \Big)^{-1} \Big\| \\
		& \qquad < \infty ,
	\end{align*}
	so 1.~is satisfied.
	
	That 3.~is equivalent to 2.~follows from \cite[Lemma II.2.3]{AMthesis}.	
	\end{proof}

\begin{remark}
\label{remark:567n457buev56y65vy365ye56y}
In our construction of weights for nilpotent groups, we will have a bound of the form
\[ \sup_{i \in I, h \in G} \Big\| (\ell_i(g h) - \ell_i(h)) \Big(1 + \sum_{j\in I} |\ell_j(h)|^{\epsilon_{i j}} \Big)^{-1} \Big\| \leq f((\| \ell_i(g) \|)_{i \in I}) , \]
for some continuous function \( f : [0, \infty)^I \to [0, \infty) \), which will imply the translation-boundedness of \( (\ell_i)_{i \in I} \) for a bounding matrix \( \VEC{\epsilon} \).
\end{remark}

\begin{thm}
	Let \( G \) be a locally compact group and \( V \) be a finite-dimensional vector space. Let \( (\ell_j)_{j\in I} : G \to \End V \) be a finite collection of weights which is self-adjoint, proper, and translation bounded with \( \VEC{\epsilon} \in M_I([0,\infty)) \). Let \( (M_{\ell_j})_{j\in I} \) be the operators densely defined on \( L^2(G, V) \) given by multiplication by \( (\ell_j)_{j\in I} \) respectively. Then, provided \( \VEC{\epsilon} \) satisfies the decreasing cycle condition,
	\[ \big( C_c(G), L^2(G, V), (M_{\ell_j})_{j\in I} \big) \]
	is a strictly tangled spectral triple with bounding matrix \(\VEC{\epsilon}\). If \( V \) is \(\Z/2\Z\)-graded, the ST\textsuperscript{2} is even, otherwise it is odd.
\end{thm}

\begin{proof}
	The local compactness of the resolvent is a consequence of the properness of \( (\ell_j)_{j\in I} \) and the isomorphism \( C_0(G, \End V) \rtimes G \cong K(L^2(G, V)) \). For the commutator bounds, fix an element \( f \in C_c(G) \) and let
	\[ T_i = [M_{\ell_i}, f] \Big( 1 + \sum_{j\in I}|M_{\ell_j}|^{\epsilon_{ij}} \Big)^{-1} . \]
	On a vector \( \xi \in C_c(G, V) \),
	\[ (T_i \xi)(h) = \int_G \left(\ell_i(h) - \ell_i(s^{-1} h) \right) \Big(1 + \sum_{j\in I} |\ell_j(s^{-1} h)|^{\epsilon_{i j}} \Big)^{-1} f(s) \xi(s^{-1} h) d\mu(s) . \]
	A computation similar to that in \cite[Proof of Theorem II.2.24]{AMthesis} shows that
	\[ \| T_i \xi \| \leq \| \xi \| \| f \|_{L^1} \sup_{h \in G, s \in \supp f} \Big\| \left(\ell_i(h) - \ell_i(s^{-1} h) \right) \Big(1 + \sum_{j\in I} |\ell_j(s^{-1} h)|^{\epsilon_{i j}} \Big)^{-1} \Big\| \]
	Using the compactness of the support of \( f \) and Lemma \ref{lemma:srt7n56ne756e5n6-copy}, it follows that \( T_i \) is bounded.
\end{proof}

Let \( G \) be a simply connected nilpotent Lie group. Recall the lower central series
\[ \mathfrak{g}_1 = \mathfrak{g} \qquad \mathfrak{g}_n = [\mathfrak{g}_1, \mathfrak{g}_{n-1}] . \]
The successive quotients \( \mathfrak{g}_n / \mathfrak{g}_{n+1} \) are abelian Lie algebras. The largest \( n \) for which \( \mathfrak{g}_n \) is nonzero is the step size \( s \) of \( \mathfrak{g} \). In the Baker–Campbell–Hausdorff formula for \( \log(\exp X \exp Y) \), we will call the \( n \)th-order term \( z_n(X, Y) \), so that, e.g.
\[ z_1(X, Y) = X + Y \qquad z_2(X, Y) = \frac{1}{2} [X, Y] \qquad z_3(X, Y) = \frac{1}{12}([[X, Y], Y] + [[Y, X], X]) . \]
Because \( z_{s+1}(X, Y) = 0 \), the exponential map from \( \mathfrak{g} \) to \( G \) is a diffeomorphism.

A \emph{Malcev basis} of \( \mathfrak{g} \) through the lower central series is a basis \( ((e_{j, k})_{k=1}^{\dim \mathfrak{g}_j / \mathfrak{g}_{j+1}})_{j=1}^s \) of \( \mathfrak{g} \) such that \( ((e_{j, k})_{k=1}^{\dim \mathfrak{g}_j / \mathfrak{g}_{j+1}})_{j=n}^s \) is a basis of \( \mathfrak{g}_n \) \cite[Theorem 1.1.13]{Corwin_1990}. Remark that the span of
\[ \{ e_{j, k}, \ldots, e_{j, \dim \mathfrak{g}_j / \mathfrak{g}_{j+1}}, e_{j+1, 1}, \ldots, e_{s, \dim \mathfrak{g}_s} \} \]
(in other words, the basis with some number of elements dropped from the beginning) is automatically an ideal of \( \mathfrak{g} \).
Using the Malcev basis, we may write an arbitrary element of \( X \in \mathfrak{g} \) as a tuple \( (x_1, \ldots, x_s) \in \bigoplus_{n=1}^s \R^{\dim \mathfrak{g}_n / \mathfrak{g}_{n+1}} \).

\begin{prop}
\label{lnaljdnaljn}
	Let \( G \) be a simply connected \( s \)-step nilpotent Lie group and choose a Malcev basis \( ((e_{j, k})_{k=1}^{\dim \mathfrak{g}_j / \mathfrak{g}_{j+1}})_{j=1}^s \) of \( \mathfrak{g} \) through the lower central series. Let \( V \) be Clifford module for \( \Cl_{\dim \mathfrak{g}} \), whose generators we label \( ((\gamma_{j, k})_{k=1}^{\dim \mathfrak{g}_j / \mathfrak{g}_{j+1}})_{j=1}^s \). Then the collection \( (\ell_j)_{j=1}^s : G \to \End V \) of weights given by
	\[ \ell_j(\exp_{\mathfrak{g}} (x_1, \ldots, x_s)) = \sum_{k=1}^{\dim \mathfrak{g}_j / \mathfrak{g}_{j+1}} x_{j, k} \gamma_{j, k} \]
	is self-adjoint, proper, and translation bounded with the strictly lower triangular bounding matrix \( \epsilon_{i j} = \max\{i - j, 0\} \).
\end{prop}

\begin{proof}
	Self-adjointness is by construction. For properness, observe that
	\[ \ell_j(\exp_{\mathfrak{g}} (x_1, \ldots, x_s))^2 = \sum_{k=1}^{\dim \mathfrak{g}_j / \mathfrak{g}_{j+1}} x_{j, k}^2 . \]
	For translation-boundedness, observe that \( \ell_j \) is well-defined on the quotient \( G/G_{j+1} \) and \( \epsilon_{i j} = 0 \) for \( j \geq i \). Without loss of generality, then, we consider only the translation-boundedness of \( \ell_s \). For any \( 1 \leq m \leq s \), the map
	\[ \| \cdot \|_{m+1} : X \mapsto \sqrt{\sum_{j=1}^m \ell_j(\exp_{\mathfrak{g}} X)^2} \]
	defines a norm on the finite-dimensional vector space \( \mathfrak{g}/\mathfrak{g}_{m+1} \). By the necessary continuity of the Lie bracket in this norm, there exists a constant \( C_{m+1} \) such that
	\[ \| [X, Y] \|_{m+1} \leq C_{m+1} \| X \|_{m+1} \| Y \|_{m+1} \]
	for all \( X, Y \in \mathfrak{g}/\mathfrak{g}_{m+1} \). Actually, since \( [X + \mathfrak{g}_m, Y + \mathfrak{g}_m] = [X, Y] \),
	\[ \| [X, Y] \|_{m+1} \leq C_{m+1} \| X \|_m \| Y \|_m . \]
	In the term \( z_n(X, Y) \) of the Baker–Campbell–Hausdorff formula, there are no more than \( n - 1 \) instances of \( X \) or of \( Y \). (Actually, for even \( n \geq 4 \), the vanishing of the Bernoulli number \( B_{n-1} \) means that there are no more than \( n - 2 \) instances of \( X \) or \( Y \).) Because \( \mathfrak{g} \) is \( s \)-step nilpotent,
	\[ z_n(X + \mathfrak{g}_{s-n+2}, Y + \mathfrak{g}_{s-n+2}) = z_n(X, Y) \]
	for \( X, Y \in \mathfrak{g} \), and so
	\[ \| z_n(X, Y) \|_{s+1} \leq C_{s-n+2}^{n-1} \| X \|_{s-n+2}^{n-1} \| Y \|_{s-n+2}^{n-1} \]
	By the linearity of \( \ell_s \circ \exp_{\mathfrak{g}} \),
	\[ \ell_s(\exp_{\mathfrak{g}} X \exp_{\mathfrak{g}} Y) - \ell_s(\exp_{\mathfrak{g}} Y) = \ell_s(\exp_{\mathfrak{g}} X) + \sum_{n=2}^s \ell_s(\exp_{\mathfrak{g}} z_n(X, Y)) . \]
	We obtain a bound
	\begin{align*}
		& \| \ell_s(\exp_{\mathfrak{g}} X \exp_{\mathfrak{g}} Y) - \ell_s(\exp_{\mathfrak{g}} Y) \| \\
		& \qquad \leq \| \ell_s(\exp_{\mathfrak{g}} X) \| + \sum_{n=2}^s \| \ell_s(\exp_{\mathfrak{g}} z_n(X, Y)) \| \\
		& \qquad \leq \| \ell_s(\exp_{\mathfrak{g}} X) \| + \sum_{n=2}^s \| z_n(X, Y) \|_{s+1} \\
		& \qquad \leq \| \ell_s(\exp_{\mathfrak{g}} X) \| + \sum_{n=2}^s C_{s-n+2}^{n-1} \| X \|_{s-n+2}^{n-1} \| Y \|_{s-n+2}^{n-1} \\
		& \qquad = \| \ell_s(\exp_{\mathfrak{g}} X) \| + \sum_{m=1}^{s-1} C_{m+1}^{s-m} \| X \|_{m+1}^{s-m} \| Y \|_{m+1}^{s-m} \\
		& \qquad = \| \ell_s(\exp_{\mathfrak{g}} X) \| + \sum_{m=1}^{s-1} C_{m+1}^{s-m} \| X \|_{m+1}^{s-m} \Biggl( \sum_{j=1}^m \ell_j(\exp_{\mathfrak{g}} Y)^2 \Biggr)^{\frac{s-m}{2}} \\
		& \qquad \leq \| \ell_s(\exp_{\mathfrak{g}} X) \| + \sum_{m=1}^{s-1} C_{m+1}^{s-m} \| X \|_{m+1}^{s-m} m^{s-m-1} \sum_{j=1}^m \| \ell_j(\exp_{\mathfrak{g}} Y) \|^{s-m} \\
		& \qquad = \| \ell_s(\exp_{\mathfrak{g}} X) \| + \sum_{j=1}^{s-1} \sum_{m=j}^{s-1} C_{m+1}^{s-m} \| X \|_{m+1}^{s-m} m^{s-m-1} \| \ell_j(\exp_{\mathfrak{g}} Y) \|^{s-m} .
	\end{align*}
	We conclude from this and Remark \ref{remark:567n457buev56y65vy365ye56y} that \( \epsilon_{s j} = s - j \) is sufficient to give translation-boundedness.
\end{proof}

The reader can note that a lower triangular bounding matrix automatically satisfies the decreasing cycle condition since its associated weighted directed graph has no cycles. For instance, for a 5-step nilpotent group, the bounding graph is
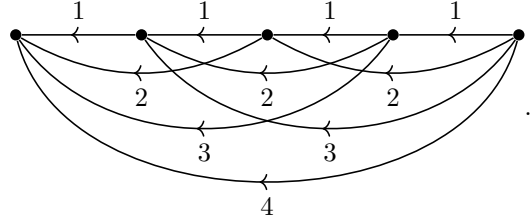
\[ \vcenter{\hbox{\begin{tikzpicture}
	\graph[grow right sep=1.5cm, empty nodes, nodes={fill=black, circle, inner sep=1.5pt}, edges={semithick}]{
	a --["$1$" inner sep=5.5pt, middlearrow={<}] b --["$1$" inner sep=5.5pt, middlearrow={<}] c --["$1$" inner sep=5.5pt, middlearrow={<}] d --["$1$" inner sep=5.5pt, middlearrow={<}] e;
	a --["$2$"' inner sep=5.5pt, middlearrow={<}, bend right] c;
	b --["$2$"' inner sep=5.5pt, middlearrow={<}, bend right] d;
	c --["$2$"' inner sep=5.5pt, middlearrow={<}, bend right] e;
	a --["$3$"' inner sep=5.5pt, middlearrow={<}, bend right=55] d;
	b --["$3$"' inner sep=5.5pt, middlearrow={<}, bend right=55] e;
	a --["$4$"' inner sep=5.5pt, middlearrow={<}, bend right=75] e;
	};
\end{tikzpicture}}} . \]

\begin{prop}
	Let \( G \) be a simply connected \( s \)-step nilpotent Lie group. Then, for an irreducible Clifford module \( V \) for \( \Cl_{\dim \mathfrak{g}} \),
	\[ \left( C_c(G), L^2(G, V), (M_{\ell_n})_{n=1}^s \right) \qquad \epsilon_{i j} = \max\{i - j, 0\} \]
	is an ST\textsuperscript{2} with nontrivial class in \( KK_{\dim \mathfrak{g}}(C^*(G), \C) \). This ST\textsuperscript{2} represents the Kasparov product
	\begin{multline*}
		[(C_c(G_1), \overline{(C_c(G_1, V_1))_{C^*(G_2)}}, M_{\ell_1})] \\
		\otimes_{C^*(G_2)} [(C_c(G_2), \overline{(C_c(G_2, V_2))_{C^*(G_3)}}, M_{\ell_2})] \\
		\otimes_{C^*(G_3)} \cdots \otimes_{C^*(G_s)} [(C_c(G_s), L^2(G_s, V_s), M_{\ell_s})]
	\end{multline*}
	where each \( E_j \) is a Clifford \( \Cl_{\dim \mathfrak{g}_j / \mathfrak{g}_{j+1}} \)-module with generators \( (\gamma_{j, k})_{k=1}^{\dim \mathfrak{g}_j / \mathfrak{g}_{j+1}} \) and \( V = V_1 \tildeotimes \cdots \tildeotimes V_s \).
\end{prop}
\begin{proof}
	First, note that the maximal compact subgroup of \( G \) is the trivial subgroup, so its dual Dirac element \( \beta \) is in \( KK^G_{\dim \mathfrak{g}}(\C, C_0(G)) \) \cite[Definition 5.1]{Kasparov_1988}.
	Take \( \VEC{t} \in \Omega(\VEC{\epsilon}) \). Comparing with \cite[Proof of Theorem 5.7]{Kasparov_1988}, we see from the description as an iterated Kasparov product that \(  (\C, C_0(G, V)_{C_0(G)}, \overline{M_{\VEC{\ell}}}_t) \) represents the dual Dirac class \( \beta \). By definition, \( \alpha \otimes_{\C} \beta = 1 \in KK^G(C_0(G), C_0(G)) \) for the Dirac class \( \alpha \in KK^G_{\dim \mathfrak{g}}(C_0(G), \C) \).
	The class of
	\[ \left( C_c(G), L^2(G, V), (M_{\ell_n})_{n=1}^s \right) \]
	is the descent \( j^G(\beta) \in KK_{\dim \mathfrak{g}}(C^*(G), C_0(G) \rtimes G) = KK_{\dim \mathfrak{g}}(C^*(G), \C) \) of \( \beta \), which is nonzero because \( j^G(\alpha) \otimes_{C^*(G)} j^G(\beta) = j^G(\alpha \otimes_{\C} \beta) = 1 \).
\end{proof}

\begin{prop}
	Let \( G \) be a simply connected \( s \)-step nilpotent Lie group and \( H \) be a cocompact, closed subgroup. Then
	\[ \left( C^*(H), L^2(H, V), (M_{\ell_n})_{n=1}^s \right) \qquad \epsilon_{i j} = \max\{i - j, 0\} \]
	is an ST\textsuperscript{2} with nontrivial class in \( KK_{\dim \mathfrak{g}}(C^*(H), \C) \).
\end{prop}

\begin{proof}
	To show nontriviality, we argue along the lines of \cite[\S 6.2]{Valette_2002}. The spectral triple
	\[ \left( C^*(H), \ell^2(H, V), (M_{\ell_n})_{n=1}^s \right) \]
	has class \( \mathbf{x} = j^H(\beta \otimes_{C_0(G)} [\omega]) \otimes_{C_0(H) \rtimes H} [L^2(H)] \) in \( KK_{\dim \mathfrak{g}}(C^*(H), \C) \),
	where \( \beta \in KK^H_{\dim \mathfrak{g}}(\C, C_0(G)) \) is the dual Dirac element and \( \omega \) is the inclusion map \( H \hookrightarrow G\). Using the cocompactness of \( H \subseteq G \), one can construct a class \( [\theta] \in KK_0(\C, C_0(G) \rtimes H) \), as in \cite[\S 6.2]{Valette_2002}, for which \( [\theta] \otimes_{C_0(G) \rtimes H} j^H([\omega]) \otimes_{C_0(H) \rtimes H} [L^2(H)] = 1 \in KK_0(\C, \C) \). With the Dirac element \( \alpha \in KK^H_{\dim \mathfrak{g}}(C_0(G), \C) \), we have
	\begin{multline*}
		[\theta] \otimes_{C_0(G) \rtimes H} j^H(\alpha) \otimes_{C^*(H)} j^H(\beta \otimes_{C_0(G)} [\omega]) \otimes_{C_0(H) \rtimes H} [L^2(H)] \\
		= [\theta] \otimes_{C_0(G) \rtimes H} 1 \otimes_{C_0(G) \rtimes H} j^H([\omega]) \otimes_{C_0(H) \rtimes H} [L^2(H)] = 1 ,
	\end{multline*}
	showing that \( \mathbf{x} \) is nontrivial.
\end{proof}

Malcev completion says that a group \( \Gamma \) is isomorphic to a lattice in a simply connected nilpotent Lie group if and only if \( \Gamma \) is finitely generated, torsion-free, and nilpotent; see e.g. \cite[Theorem 2.18]{Raghunathan_1972}. We thereby obtain

\begin{prop}
	Let \( \Gamma \) be a finitely generated, torsion-free, nilpotent group. Let \( G \) be a simply connected nilpotent Lie group in which $\Gamma$ is a lattice. Then
	\[ \left( C_c(\Gamma), \ell^2(\Gamma, V), (M_{\ell_n})_{n=1}^s \right) \qquad \epsilon_{i j} = \max\{i - j, 0\} \]
	is an ST\textsuperscript{2} having a nontrivial class in \( KK_{\dim \mathfrak{g}}(C^*(\Gamma), \C) \). The ST\textsuperscript{2} is \( f \)-summable for
	
	\[ f (\VEC{t}) > \sum_{j=1}^s \frac{\dim \mathfrak{g}_j / \mathfrak{g}_{j+1}}{t_j} . \]
\end{prop}
\begin{proof}
	For the statement about summability, first remark that for \( \VEC{t} \in (0, \infty)^s \) the map
	\[ (x_1, \ldots, x_s) \mapsto \left( 1 + \sum_{j=1}^s \Biggl( \sum_{k=1}^{\dim \mathfrak{g}_j / \mathfrak{g}_{j+1}} x_{j, k}^2 \Biggr)^{t_j/2} \right)^{-1} \]
	is an element of \( L^p(\mathfrak{g}) \) for \( p > \sum_{j=1}^s \frac{\dim \mathfrak{g}_j / \mathfrak{g}_{j+1}}{t_j} \). The Haar measure on a simply connected nilpotent Lie group is the pushforward under the exponential of the Lebesgue measure on its Lie algebra, so \( L^p(\mathfrak{g}) \cong L^p(G) \). By \cite[Proposition 5.4.8(b)]{Corwin_1990}, \( \log_G \Gamma \) is the union of a finite number of additive cosets of a lattice in \( \mathfrak{g} \). By the integral test for convergence, then, the map
	\[ \exp_{\mathfrak{g}} (x_1, \ldots, x_s) \mapsto \left( 1 + \sum_{j=1}^s \Biggl( \sum_{k=1}^{\dim \mathfrak{g}_j / \mathfrak{g}_{j+1}} x_{j, k}^2 \Biggr)^{t_j/2} \right)^{-1} \]
	is an element of \( \ell^p(\Gamma) \) for \( p > \sum_{j=1}^s \frac{\dim \mathfrak{g}_j / \mathfrak{g}_{j+1}}{t_j} \). 
\end{proof}

If one chooses the Malcev basis to be \emph{strongly based on \( \Gamma \)}, as is always possible \cite[Theorem 5.1.6]{Corwin_1990}, then each \( \ell_j \) will be valued in the \( \Q \)-span of \( (\gamma_{j, k})_{k=1}^{\dim \mathfrak{g}_j / \mathfrak{g}_{j+1}} \) \cite[Theorem 5.1.8(a)]{Corwin_1990}. By rescaling by a suitable integer, one can ensure that each \( \ell_j \) will be valued in the \( \Z \)-span of \( (\gamma_{j, k})_{k=1}^{\dim \mathfrak{g}_j / \mathfrak{g}_{j+1}} \); cf. \cite[\S 5.4]{Corwin_1990}.

\subsubsection{Carnot groups and equivariance}

A \emph{Carnot group} is a simply connected nilpotent Lie group \( G \) with a stratification \( \mathfrak{g}= \bigoplus_{n=1}^s \mathcal{V}_n \) of its Lie algebra \( \mathfrak{g} \) such that \( [\mathcal{V}_1, \mathcal{V}_n] = \mathcal{V}_{n+1} \). A basic consequence of the stratification is that \( \mathfrak{g}_n = \bigoplus_{j=n}^s \mathcal{V}_j  \) and so naturally \( \mathcal{V}_n = \mathfrak{g}_n/\mathfrak{g}_{n+1} \); for more details see e.g.~\cite{Le_Donne_2017}.

\begin{prop}
\label{alkdnalkdn}
	Let \( G \) be a Carnot group and \( H \) be a cocompact, closed subgroup (including \( G \) itself). Choose a Malcev basis \( ((e_{j, k})_{k=1}^{\dim \mathfrak{g}_j / \mathfrak{g}_{j+1}})_{j=1}^s \) with the property that \( (e_{j, k})_{k=1}^{\dim \mathfrak{g}_j / \mathfrak{g}_{j+1}} \subset \mathcal{V}_j \). Then the collection \( (\ell_j)_{j=1}^s : G \to \End V \) of weights and, consequently, the ST\textsuperscript{2}
	\[ \big( C_c(H), L^2(H, V), (M_{\ell_j})_{j=1}^s \big) \]
	has the strictly lower triangular bounding matrix
	$$ \epsilon_{i j} = 
	\begin{cases}
	\bigl\lfloor \frac{i - 1}{j} \bigr\rfloor, \; &i>j,\\
	0, \; &i\leq j. \end{cases} $$
\end{prop}

The reader can note that the bounding matrix in Proposition \ref{alkdnalkdn} for Carnot groups improves the bounding matrix of Proposition \ref{lnaljdnaljn} built from a general nilpotent Lie group's lower central series. 

\begin{proof}
	To verify the new bounding matrix, we again restrict to considering the translation-boundedness of \( \ell_s \). Using the stratification of \( \mathfrak{g} \), for \( X \in \mathcal{V}_i \) and \( Y \in \mathcal{V}_j \),
	\[ \| \ell_{i+j}(\exp_{\mathfrak{g}} [X, Y]) \| \leq C_{i, j} \| \ell_i(\exp_{\mathfrak{g}} X) \| \| \ell_j(\exp_{\mathfrak{g}} Y) \| \]
	for some constant \( C_{i, j} \). Furthermore, for the Baker–Campbell–Hausdorff expansion \( z(X, Y) = \log(\exp X \exp Y) \),
	\[ \| \ell_s(\exp_{\mathfrak{g}} z(X, Y)) \| \leq C'_{i, j, s} \| \ell_i(\exp_{\mathfrak{g}} X) \|^{\lfloor (s-j)/i \rfloor} \| \ell_j(\exp_{\mathfrak{g}} Y) \|^{\lfloor (s-i)/j \rfloor} \]
	for some constant \( C'_{i, j, s} \). We obtain a bound
	\begin{align*}
		& \| \ell_s(\exp_{\mathfrak{g}} X \exp_{\mathfrak{g}} Y) - \ell_s(\exp_{\mathfrak{g}} Y) \| \\
		& \qquad \leq \| \ell_s(\exp_{\mathfrak{g}} X) \| + \| \ell_s(\exp_{\mathfrak{g}} z(X, Y)) \| \\
		& \qquad \leq \| \ell_s(\exp_{\mathfrak{g}} X) \| + C'_{i, j, s} \| \ell_i(\exp_{\mathfrak{g}} X) \|^{\lfloor (s-j)/i \rfloor} \| \ell_j(\exp_{\mathfrak{g}} Y) \|^{\lfloor (s-i)/j \rfloor} .
	\end{align*}
	Hence, remembering Remark \ref{remark:567n457buev56y65vy365ye56y}, \( \epsilon_{s j} = \bigl\lfloor \frac{s-1}{j} \bigr\rfloor \) is sufficient for translation-boundedness.
\end{proof}

For a 5-step Carnot group, the bounding graph produced by Proposition \ref{alkdnalkdn} is
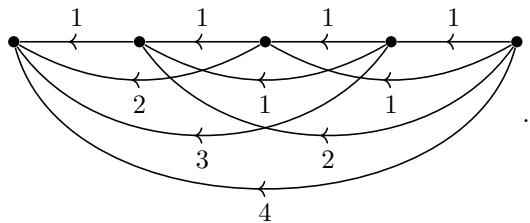
\[ \vcenter{\hbox{\begin{tikzpicture}
	\graph[grow right sep=1.5cm, empty nodes, nodes={fill=black, circle, inner sep=1.5pt}, edges={semithick}]{
	a --["$1$" inner sep=5.5pt, middlearrow={<}] b --["$1$" inner sep=5.5pt, middlearrow={<}] c --["$1$" inner sep=5.5pt, middlearrow={<}] d --["$1$" inner sep=5.5pt, middlearrow={<}] e;
	a --["$2$"' inner sep=5.5pt, middlearrow={<}, bend right] c;
	b --["$1$"' inner sep=5.5pt, middlearrow={<}, bend right] d;
	c --["$1$"' inner sep=5.5pt, middlearrow={<}, bend right] e;
	a --["$3$"' inner sep=5.5pt, middlearrow={<}, bend right=55] d;
	b --["$2$"' inner sep=5.5pt, middlearrow={<}, bend right=55] e;
	a --["$4$"' inner sep=5.5pt, middlearrow={<}, bend right=75] e;
	};
\end{tikzpicture}}} . \]

\begin{remark}
	It is notable that the behaviour here, in contrast to the general nilpotent Lie group case, is close to that of pseudodifferential operators. In the context of Remark \ref{remark:operators_with_prescribed_orders}, if we let \( \VEC{m} = (1, 2, \ldots, s) \), we expect a bounding matrix \( \epsilon'_{i j} = \frac{i - 1}{j} \), which is just fractionally larger than the \( \VEC{\epsilon} \) given above. We therefore may think of \( M_{\ell_j} \) as having order \( j \). The ray
	$$\VEC{t}_{\VEC{m}}(\tau):=\left(\frac{\tau}{j}\right)_{j=1}^s \qquad (\tau>0)$$
	is in the cone \( \Omega(\VEC{\epsilon}) \) when \( G \) is Carnot, but will not be, in general, for a nilpotent Lie group and \( \epsilon_{i j} = \max\{i - j, 0\} \).
\end{remark}

The stratification provides a canonical vector space isomorphism of \( \mathfrak{g} \) and \( \bigoplus_{n=1}^s \mathfrak{g}_n / \mathfrak{g}_{n+1} \) because \( \mathfrak{g}_n / \mathfrak{g}_{n+1} = \mathcal{V}_n \). We may write any element of \( \mathfrak{g} \) as a tuple \( (X_1, \ldots, X_s) \in \bigoplus_{n=1}^s \mathcal{V}_n \) and any element of \( G \) as the exponential of such a tuple. The stratification induces a dilation action of \( \R^{\times}_+ \) as Lie algebra automorphisms on \( \mathfrak{g} \), given by
\[ \delta_t : (X_1, X_2,\ldots, X_s) \mapsto (t X_1, t^2X_2,\ldots, t^s X_s) . \]
This action exponentiates to a dilation action on \( G \) by automorphisms, given by
\[ \delta_t : \exp_{\mathfrak{g}} (X_1,X_2, \ldots, X_s) \mapsto \exp_{\mathfrak{g}} (t X_1,t^2X_2, \ldots, t^s X_s) . \]
Let \( V_t \) be given by the pullback
\[ V_t \xi(\exp_{\mathfrak{g}}(X_1, \ldots, X_s)) = \xi(\exp_{\mathfrak{g}}(t^{-1} X_1, \ldots, t^{-s} X_s)) \]
on \( \xi \in L^2(G) \). Recall that the Haar measure on a simply connected nilpotent Lie group is the pushforward under the exponential of the Lebesgue measure on its Lie algebra. We compute that
\begin{align*}
	\langle V_t^* \xi \mid \eta \rangle
	& = \int \xi(\exp_{\mathfrak{g}}(t^{-1} X_1, \ldots, t^{-s} X_s)) \eta(\exp_{\mathfrak{g}}(X_1, \ldots, X_s)) dX_1 \cdots dX_S \\
	& = \int \xi(\exp_{\mathfrak{g}}(Y_1, \ldots, Y_s)) \eta(\exp_{\mathfrak{g}}(t Y_1, \ldots, t^s Y_s)) t^{\dim \mathcal{V}_1} dY_1 \cdots t^{s \dim \mathcal{V}_s} dY_S \\
	& = t^{\dim_h \mathfrak{g}} \langle \xi \mid V_{t^{-1}} \eta \rangle
\end{align*}
using the notation
\[ \dim_h \mathfrak{g} = \sum_{n=1}^s n \dim \mathcal{V}_n \]
for the homogeneous dimension of \( \mathfrak{g} \) (cf. \eqref{homodim}). Hence \( V_t^* = t^{\dim_h \mathfrak{g}} V_{t^{-1}} \). The unitary in the polar decomposition of \( V_t \) is given by \( U_t = t^{-\dim_h(\mathfrak{g})/2} V_t \). For \( 1 \leq j \leq s \),
\[ \ell_j(\exp_{\mathfrak{g}}(t^{-1} X_1, \ldots, t^{-s} X_s)) = t^{-j} \ell(\exp_{\mathfrak{g}}(X_1, \ldots, X_s)) \]
and we see that the operator \( M_{\ell_j} \) transforms as
\begin{align*}
	& (U_t M_{\ell_j} U_t^* \xi)(\exp_{\mathfrak{g}}(X_1, \ldots, X_s)) \\
	& \qquad = t^{-\dim_h(\mathfrak{g})/2} (M_{\ell_j} U_t^* \xi)(\exp_{\mathfrak{g}}(t^{-1} X_1, \ldots, t^{-s} X_s)) \\
	& \qquad = t^{-\dim_h(\mathfrak{g})/2} \ell_j(\exp_{\mathfrak{g}}(t^{-1} X_1, \ldots, t^{-s} X_s)) (U_t^* \xi)(\exp_{\mathfrak{g}}(t^{-1} X_1, \ldots, t^{-s} X_s)) \\
	& \qquad = t^{-j} \ell_j(\exp_{\mathfrak{g}}(X_1, \ldots, X_s)) \xi(\exp_{\mathfrak{g}}(X_1, \ldots, X_s)) \\
	& \qquad = t^{-j} (M_{\ell_j} \xi)(\exp_{\mathfrak{g}}(X_1, \ldots, X_s))
\end{align*}
on a vector \( \xi \in L^2(G, V) \). We thereby obtain, generalising \cite[Proposition 3.28]{AMsomewhere},

\begin{prop}
	\label{prop:carnot-equivariant}
	Let \( G \) be an \( s \)-step Carnot group. Then
	\[ \Big( C_c(G), L^2(G, V), \sum_{j=1}^s \sign(M_{\ell_j}) |M_{\ell_j}|^{\tau/j} \Big) \]
	is a conformally \( \R^{\times}_+ \)-equivariant higher order spectral triple for the dilation action \( \delta \) and conformal factor \( \mu_t = t^{-\tau/2} \).
\end{prop}

\section{Crossed products by parabolic diffeomorphisms}
\label{section:parabolic}

Spectral triples for crossed products $A\rtimes G$ built from spectral triples on $A$ and length functions on the group $G$ have been considered multiple times in the literature, first in \cite{Cornelissen_2008,Bellissard_2010} for the group \( \Z \), later in \cite{Hawkins_2013,Paterson_2014, rubinetal} for other discrete groups as well as further developments in unbounded $KK$-theory \cite{GMCK,GMRtwist}.

Let \( \alpha \) be an action of a locally compact group $G$ by automorphisms of a $C^*$-algebra \( A \). The (reduced) crossed product $C^*$-algebra \( A \rtimes_{\alpha} G \) possesses a densely defined, completely positive map \( \Phi :A \rtimes_{\alpha} G \dashrightarrow A \) given on \( f \in C_c(G, A) \) by evaluation at the identity \( e \in G \). We may complete \(  \Dom(\Phi)\subseteq A \rtimes_{\alpha} G \) to a right Hilbert
\(A\)-module under the inner product
\[ \langle f_1 | f_2 \rangle_A = \Phi(f_1^* f_2), \qquad \mbox{for}\qquad f_1, f_2 \in \Dom(\Phi). \]
There is a natural isomorphism of this Hilbert module with \( L^2(G, A)_A \). The resulting representation of \( A \rtimes_{\alpha} G \) on \( L^2(G, A)_A \) is given by
\[ f \xi(g) = \int_G \alpha_{g^{-1}}(f(h)) \xi(h^{-1} g) d\mu(h) \]
for \( f \in C_c(G, A) \subseteq A \rtimes_{\alpha} G \) and \( \xi \in C_c(G, A) \subseteq L^2(G, A) \).
Given a self-adjoint, proper, translation-bounded weight \( \ell : G \to \End V \), we may construct a vertical calculus for \( A \rtimes_{\alpha} G \), in the form of an unbounded Kasparov \(A \rtimes_{\alpha} G\)-\(A\)-module
\[ (A \rtimes_{\alpha} G, L^2(G, V) \otimes A_A, M_{\ell} \otimes 1) . \]
More details of this construction and its generalisation to Fell bundles over locally compact groups can be found in \cite[\S II.2]{AMthesis}. Two weights which we shall particularly consider in later examples are the inclusions \( \ell_{\Z} : \Z \to \C \) and \( \ell_{\R} : \R \to \C \). The first of these gives rise to the number operator \( N = M_{\ell_{\Z}} \) and the Pimsner–Voiculescu extension class and the second is related to the Connes–Thom isomorphism.

A horizontal calculus is just a spectral triple
\((\mathcal{A}, \mathpzc{H}, D)\). Provided that \( A \) is represented nondegenerately on \( \mathpzc{H} \), the internal tensor product module \( L^2(G, A) \otimes_A \mathpzc{H} \) is naturally isomorphic to \( L^2(G, \mathpzc{H}) \). To construct the Kasparov product of the vertical and horizontal calculi, a compatibility condition is required.

Let \( M \) be a \( \sigma \)-finite measure space and \( \mathpzc{H} \) a separable Hilbert space. A function \( f \) from \( M \) to bounded operators \( \mathbb{B}(\mathpzc{H}) \) is \emph{measurable} if, for every pair \( \xi, \eta \in \mathpzc{H} \), the function \( m \mapsto \langle \xi | f(m) \eta \rangle \) is measurable \cite[\S XIII.16]{Reed_1978}. It suffices to check measurability for \( \xi \) and \( \eta \) in a dense subspace of \( \mathpzc{H} \) (such as \( \Dom D \) in the context below), because of the separability of \( \mathpzc{H} \) and the fact that the pointwise limit of measurable functions is measurable.
One should compare the following definition to the fact that a Lipschitz function has a measurable weak derivative.

\begin{definition}
	\label{definition:pointwisebounded}
	cf. \cite[\S 1]{Paterson_2014}
	A spectral triple \( (\mathcal{A}, H, D) \) is \emph{pointwise bounded} with respect to an action \( \alpha \) of \( G \) on \( A \) if, for all \( a \in \mathcal{A} \), the function \(g \mapsto [D, \alpha_g(a)]\) is measurable and
	\[ \sup_{g \in G} \|[D, \alpha_g(a)]\| < \infty . \]
	In other words, \(g \mapsto [D, \alpha_g(a)]\) is \( L^{\infty} \).

\end{definition}

\begin{definition}
	Let \( \mathcal{A} \) be a dense $*$-subalgebra of a $C^*$-algebra \( A \). Let \( \alpha \) be an action of a locally compact group \( G \) on \( A \) which preserves \( \mathcal{A} \). If \( G \) is discrete, we write \( \mathcal{A} \rtimes_{\alpha} G \) for the algebraic crossed product. For a non-discrete group, we will generalise this by defining \( \mathcal{A} \rtimes_{\alpha} G \subseteq A \rtimes_{\alpha} G \) as the (dense) *-subalgebra generated by \( \mathcal{A} \) and \( C_c(G) \) under the canonical inclusions \( \mathcal{A} \subseteq A \subseteq M(A \rtimes_{\alpha} G ) \) and \( C_c(G) \subseteq C^*(G) \subseteq M(A \rtimes_{\alpha} G ) \).
\end{definition}

The following Theorem and other subsequent results will involve taking the Kasparov product of (unbounded) Kasparov modules. In the interests of economy, we will not give separate statements for different combinations of parities. Instead, we will use the symbol \( \tildeotimes \) to represent a flexible tensor product of possibly \( \Z/2\Z \)-graded Hilbert spaces and operators thereon. Let \( D_1 \) and \( D_2 \) be (odd) unbounded operators on (possibly graded) Hilbert spaces \( \mathpzc{H}_1 \) and \( \mathpzc{H}_2 \) respectively. In the case when \( \mathpzc{H}_1 \) and \( \mathpzc{H}_2 \) are both graded, \( \tildeotimes \) will simply mean the graded tensor product. In the case when \( \mathpzc{H}_1 \) is graded and \( \mathpzc{H}_2 \) ungraded, \( \mathpzc{H}_1 \tildeotimes \mathpzc{H}_2 \) will refer to the plain tensor product, giving an ungraded Hilbert module, and we will write \( D_1 \tildeotimes 1 = D \otimes 1 \) and \( 1 \tildeotimes D_2 = \gamma_1 \otimes D_2 \) where \( \gamma_1 \) is the grading on \( \mathpzc{H}_1 \). In the case when \( \mathpzc{H}_1 \) is ungraded and \( \mathpzc{H}_2 \) graded, \( \mathpzc{H}_1 \tildeotimes \mathpzc{H}_2 \) will again refer to the plain tensor product, and we will write \( D_1 \tildeotimes 1 = D \otimes \gamma_2 \) and \( 1 \tildeotimes D_2 = 1 \otimes D_2 \) where \( \gamma_2 \) is the grading on \( \mathpzc{H}_2 \). If both \( \mathpzc{H}_1 \) and \( \mathpzc{H}_2 \) are ungraded, we will let \( \mathpzc{H}_1 \tildeotimes \mathpzc{H}_2 = \mathpzc{H}_1 \otimes \mathpzc{H}_2 \otimes \C^2 \), with a grading given by \( 1 \otimes 1 \otimes \sigma_3 \) and write \( D_1 \tildeotimes 1 = D \otimes 1 \otimes \sigma_1 \) and \( 1 \tildeotimes D_2 = 1 \otimes D_2 \otimes \sigma_2 \), where \( \sigma_1 \), \( \sigma_2 \), and \( \sigma_3 \) are the Pauli matrices.

\begin{thm}
	\label{theorem:crossedproductspectraltriple}
	cf. \cite[Theorem 3.4]{Cornelissen_2008}, \cite[\S 3.4]{Bellissard_2010}, \cite[Theorem 2.7]{Hawkins_2013}, \cite[Proposition 4.1]{Paterson_2014}
	Let \( (\mathcal{A}, \mathpzc{H}, D) \) be a spectral triple. Let \( \alpha \) be an action of a locally compact group \( G \) on \( \mathcal{A} \). Let \( \ell : G \to \End V \) be a self-adjoint, proper, translation-bounded weight. If the spectral triple is pointwise bounded with respect to the action,
	\[ \left( \mathcal{A} \rtimes_{\alpha} G, L^2(G, V) \tildeotimes \mathpzc{H}, M_{\ell} \tildeotimes 1 + 1 \tildeotimes D \right) \]
	is a spectral triple, representing the Kasparov product
	\[ (\mathcal{A} \rtimes_{\alpha} G, L^2(G, V) \otimes A_A, M_{\ell} \otimes 1) \otimes_A (\mathcal{A}, \mathpzc{H}, D) . \]
\end{thm}

Theorem \ref{theorem:crossedproductspectraltriple} is known in the case of discrete groups but, to our knowledge, the generalisation to locally compact groups has not appeared in the literature, although see \cite[Note after Proposition 4.1]{Paterson_2014}.

\begin{proof}
	This is an instance of the constructive unbounded Kasparov product. We note that \( M_{\ell} \tildeotimes 1 \) and \( 1 \tildeotimes D \) anticommute, so, in order to apply \cite[Theorem 7.4]{bramlesch} we need only check the boundedness of commutators and the connection condition. For the latter, let \( \xi \otimes a \in C_c(G, V) \otimes \mathscr{A} \) and \( T_{\xi \otimes a} \in B(H, L^2(G, V) \tildeotimes H) \) be given by
	\[ (T_{\xi \otimes a} \eta)(g) = \xi(g) \tildeotimes \alpha_{g^{-1}}(a) \eta \qquad (\eta \in H) . \]
	Then, for \( \eta \in \Dom D \),
	\[ (((1 \tildeotimes D) T_{\xi \otimes a} - T_{\xi \otimes a} D) \eta)(g) = \xi(g) \tildeotimes [D, \alpha_{g^{-1}}(a)] \eta . \]
	Because \( \xi \) is compactly supported and \( g \mapsto [D, \alpha_g(a)]\) is measurable, \( (1 \tildeotimes D) T_{\xi \otimes a} - T_{\xi \otimes a} D \) is bounded.
	
	To check bounded commutators, by \cite[Corollary 2.2]{Forsyth_2014}, it suffices to show that the elements of \( \mathcal{A} \rtimes_{\alpha} G \) take a core for \( M_{\ell} \tildeotimes 1 + 1 \tildeotimes D \) to the domain and have bounded commutators on that core. Let \( \eta = \eta_1 \tildeotimes \eta_2 \in C_c(G, V) \tildeotimes \Dom D \), a core for \( M_{\ell} \tildeotimes 1 + 1 \tildeotimes D \). (If both \( V \) and \( \mathpzc{H} \) are ungraded, then \( \eta \) should have an extra \( \C^2 \) tensor factor, but this detail does not change the argument below.) Then
	\[ (\pi(a f) \eta)(g) = \int_G f(h) \eta_1(h^{-1} g) d\mu(h) \tildeotimes \alpha_{g^{-1}}(a) \eta_2 \in E \tildeotimes \Dom D \]
	for all \( g \in G \) and 
	\begin{align*}
		& \int_G \left\| (\ell(g) \tildeotimes 1 + 1 \tildeotimes D) \int_G f(h) \eta_1(h^{-1} g) d\mu(h) \tildeotimes \alpha_{g^{-1}}(a) \eta_2 \right\|^2 d\mu(g) \\
		& \qquad \leq \int_G \left\| (\ell(g) \int_G f(h) \eta_1(h^{-1} g) d\mu(h) \tildeotimes \alpha_{g^{-1}}(a) \eta_2 \right\|^2 d\mu(g) \\
		& \qquad\qquad + \int_G \left\| \int_G f(h) \eta_1(h^{-1} g) d\mu(h) \tildeotimes \left( [D, \alpha_{g^{-1}}(a)] + \alpha_{g^{-1}}(a) D \right) \eta_2 \right\|^2 d\mu(g)
	\end{align*}
	is finite owing to the compactness of the supports of \( f \) and \( \eta_1 \) and pointwise-boundedness. By \cite[Theorem XIII.85]{Reed_1978}, this implies that \( \pi(a f) \eta \) is in the domain of \( M_{\ell} \tildeotimes 1 + 1 \tildeotimes D \). It is routine to check that
	\[ [1 \tildeotimes D, \pi(a f)] \eta(g) = 1 \tildeotimes [D, \alpha_{g^{-1}}(a)] \pi(f) \eta(g) \]
	and 
	\[ [M_{\ell} \tildeotimes 1, \pi(a f)] \eta(g) = \alpha_{g^{-1}}(a) \int_G \left( (\ell(h^{-1} g) - \ell(g)) f(h) \tildeotimes 1 \right) \eta(h^{-1} g) d\mu(h) . \]
	The commutator \( [M_{\ell} \tildeotimes 1 + 1 \tildeotimes D, \pi(a f)] \) is then bounded because of pointwise-boundedness and the facts that \( \ell \) is translation bounded and that \( f \) is compactly supported. By the Leibniz rule, we are done.
\end{proof}

With the technology of ST\textsuperscript{2}s available, we are not so constrained as in the spectral triple case. We make the following definition.

\begin{definition}
	A spectral triple \( (\mathcal{A}, H, D) \) has \emph{parabolic order \( s \in [0, \infty) \)} with respect to an action \( \alpha \) of \( G \) on \( A \) and a weight \( \ell \) on \( G \) if, for all \( a \in \mathcal{A} \), the function \(g \mapsto [D, \alpha_g(a)]\) is measurable and, for all \( g \), the matrix inequality
	\[ \|[D, \alpha_g(a)]\| \leq C_a (1 + | \ell(g) |^s) \]
	holds for some constant \( C_a > 0 \). (If \( s = 0 \), we recover pointwise-boundedness.) 
\end{definition}

\begin{remark}
	\label{remark:conjugacypointwiseorder}
	Let \( \alpha \) be an action of a locally compact group \( G \) on \( \mathcal{A} \).
	If \(\beta\) is an automorphism of \(A\) preserving \(\mathcal{A}\), there is an isomorphism
	\[ \mathcal{A} \rtimes_{\beta \circ \alpha \circ \beta^{-1}} G \cong \mathcal{A} \rtimes_{\alpha} G . \]
	Let \((\mathcal{A}, H, D)\) be a spectral triple which is parabolic of order \( s \) with respect to the action and a weight \( \ell \). Suppose that there is a constant \(C' > 0\) such that, for all \( a \in \mathcal{A} \),
	\[ \|[D, \beta(a)]\| \leq C' \|[D, a]\| . \]
	Then \((\mathcal{A}, H, D)\) also is parabolic of order \( s \) with respect to \( \beta \circ \alpha \circ \beta^{-1} \) and \( \ell \) because
	\[ \|[D, \beta \circ \alpha_g \circ \beta^{-1}(a)]\| \leq C' \|[D, \alpha_g(\beta^{-1}(a))]\| \leq C' C_{\beta^{-1}(a)} (1 + | \ell(g) |^s) . \]
\end{remark}

\begin{thm}
	\label{thm:crossedproductST2}
	Let \( (\mathcal{A}, \mathpzc{H}, D) \) be a spectral triple. Let \( \alpha \) be an action of a locally compact group \( G \) on \( \mathcal{A} \). Let \( \ell : G \to \End V \) be a self-adjoint, proper, translation-bounded weight. If the spectral triple is parabolic of order \( s \) with respect to the action and weight,
	\[ \left( \mathcal{A} \rtimes_{\alpha} G, L^2(G, V) \tildeotimes \mathpzc{H}, (M_{\ell} \tildeotimes 1, 1 \tildeotimes D) \right) \]
	is an ST\textsuperscript{2} with bounding matrix
	\[ \VEC{\epsilon}=\begin{pmatrix}0&0\\s&0\end{pmatrix} \qquad \vcenter{\hbox{\begin{tikzpicture}
	\graph[grow right sep=1.5cm, empty nodes, nodes={fill=black, circle, inner sep=1.5pt}, edges={semithick}]{
	a --["$s$" inner sep=5.5pt, middlearrow={<}] b;
	};
	\end{tikzpicture}}} . \]
	The ST\textsuperscript{2} represents the Kasparov product
	\[ (\mathcal{A} \rtimes_{\alpha} G, L^2(G, V) \otimes A_A, M_{\ell} \otimes 1) \otimes_A (\mathcal{A}, H, D) . \]
\end{thm}
\begin{proof}
	The proof of Theorem \ref{theorem:crossedproductspectraltriple} carries over with the appropriate modifications for the tangled boundedness of commutators implied by the pointwise order.
	
	For the last point, using Kucerovsky's theorem \cite{danthedan} (and in particular its extension to higher order spectral triples in \cite[Theorem A.7]{GMCK}), we see that, e.g.~for \( m > s \), the higher order spectral triple
	\[ (\mathcal{A} \rtimes G, L^2(G, V) \tildeotimes \mathpzc{H}, M_{\ell |\ell|^{-1+m}} \tildeotimes 1 + 1 \tildeotimes D) \]
	represents the Kasparov product of \( (A \rtimes_{\alpha} G, L^2(G, V) \otimes A_A, M_{\ell |\ell|^{-1+m}}  \otimes 1) \) and \( (\mathcal{A}, \mathpzc{H}, D) \).
\end{proof}

\begin{remark}
	cf. \cite[Theorem 2.7]{Hawkins_2013}
	In the context of Theorem \ref{thm:crossedproductST2}, if \( G \) is discrete and \( (1 + |\ell|)^{-1} \in \ell^{p_1}(G, \End V) \), so that \( (C_c(G), \ell^2(G, V), M_{\ell}) \) is \( p_1 \)-summable, and \( (\mathcal{A}, \mathpzc{H}, D) \) is \( p_2 \)-summable, then
	\[ \left( \mathcal{A} \rtimes_{\alpha} G, \ell^2(G, V) \tildeotimes \mathpzc{H}, (M_{\ell} \tildeotimes 1, 1 \tildeotimes D) \right) \]
	is \( f \)-summable for \( f : (t_1, t_2) \mapsto \frac{p_1}{t_1} + \frac{p_2}{t_2} \).
\end{remark}

To see the meaning of parabolic order, we specialise to the case of a complete Riemannian manifold \( (X, \mathbf{g}) \) and a spectral triple \( (C_c^{\infty}(X), L^2(X, S), D) \), with either the Atiyah–Singer or Hodge–de Rham Dirac operator. Let \( \varphi \) be an action of a locally compact group \( G \) by diffeomorphisms; the resulting action on \( C_c^{\infty}(X) \) is given by \( \varphi^{-1 *} \), the pullback of the inverse. For \( f \in C_c^{\infty}(X) \), the commutator \( [D, f] \) is just the one-form \( \rd f \) acting by Clifford multiplication. Hence \( \| [D, f] \| = \| df \| \). Using the notation
\[ (d \varphi_g)_x : T_x X \to T_{\varphi_g(x)} X  \]
for the pushforward by \( \varphi_g \) at \( x \in X \), the chain rule gives
\[ d \varphi_g^*(f)_x = d f_{\varphi_g(x)} (d \varphi_g)_x . \]
Hence
\[ \| d \varphi_g^*(f) \|_{\infty} \leq \| d f \|_{\infty} \| d \varphi_g \|_{\infty} \]
and the parabolic order condition reduces to the matrix inequality
	\[ \| d \varphi_g \|_{\infty} \leq C (1 + | \ell(g) |^s) \]
	for a constant \( C > 0 \). In other words, the supremum norm of the Jacobian should be of polynomial order. To be clear, the norm of \( d \varphi_g \) at \( x \in X \) is
\[ \| (d \varphi_g)_x \| = \sup_{u \in T_x M} \frac{\| (d \varphi_g)_x u \|}{\| u \|} = \sup_{u \in T_x M} \frac{\mathbf{g}_{\varphi(x)}((d \varphi_g)_x u, (d \varphi_g)_x u)^{1/2}}{\mathbf{g}_x(u, u)^{1/2}} . \]
Making our observation precise, we obtain:

\begin{cor}
	\label{cor:parabolicmanifold}
	Let \( (C_c^{\infty}(X), L^2(X, S), D) \) be the Atiyah–Singer or Hodge–de Rham Dirac spectral triple on a complete Riemannian manifold \( (X, \mathbf{g}) \). Let \( \varphi \) be an action of a locally compact group \( G \) by diffeomorphisms on \( X \). Let \( \ell : G \to \End V \) be a self-adjoint, proper, translation-bounded weight. Suppose that, for some $s\geq 0$, the matrix inequality
	\[ \| d \varphi_g \|_{\infty} \leq C (1 + | \ell(g) |^s) \]
	holds for some constant \( C > 0 \). Then
\[ (C_c^\infty(X) \rtimes G, L^2(G, V) \tildeotimes L^2(X, S), (M_{\ell} \tildeotimes 1, 1 \tildeotimes D) \]
is a strictly tangled spectral triple with bounding matrix
	\[ \VEC{\epsilon}=\begin{pmatrix}0&0\\s&0\end{pmatrix} \qquad \vcenter{\hbox{\begin{tikzpicture}
	\graph[grow right sep=1.5cm, empty nodes, nodes={fill=black, circle, inner sep=1.5pt}, edges={semithick}]{
	a --["$s$" inner sep=5.5pt, middlearrow={<}] b;
	};
	\end{tikzpicture}}} . \]
This ST\textsuperscript{2} represents the Kasparov product of
\[ (C_c^\infty(X) \rtimes G, L^2(G, V) \otimes C_0(X)_{C_0(X)}, M_{\ell} \otimes 1) \]
and \( (C_c^\infty(X), L^2(X, S), D) \).
\end{cor}

The behaviour of dynamical systems can be loosely classified into three paradigms: \emph{elliptic}, \emph{parabolic}, and \emph{hyperbolic} \cite[\S 5.1.g]{Hasselblatt_2002}. These roughly refer to the Jacobian's having respectively constant growth, polynomial growth, or exponential growth. The classical example of the distinction is the classification of Möbius transformations, which we discuss in the following Example. The meaning of Corollary \ref{cor:parabolicmanifold}, then, is that ST\textsuperscript{2}s can be built for parabolic dynamical systems in addition to elliptic dynamical systems which already fall within the scope of Theorem \ref{theorem:crossedproductspectraltriple}. For a survey of parabolic dynamics, we refer to \cite[Chapter 8]{Hasselblatt_2002}; see also \cite{Fraczek_2004, Avila_2021}.

\begin{example}
	\label{example:mobius}
	In terms of the complex coordinate \(z\) on the Riemann sphere \( S^2 \), a Möbius transformation is given by
\[ z \mapsto \frac{a z + b}{c z + d} \qquad \begin{pmatrix} a & b \\ c & d \end{pmatrix} \in SL(2, \C) . \]
The centre \(\{+1, -1\}\) of \(SL(2, \C)\) acts trivially, so that the
group of Möbius transformations is \(PSL(2, \C)\). Equip \( S^2 \) with the round metric
	\[ ds^2 = \frac{4 dz d\bar{z}}{(1 + |z|^2)^2} \]
	and a corresponding spectral triple \( (C^{\infty}(S^2), L^2(S^2, S), D) \). We will consider the behaviour of a \( \Z \)-action generated by a single Möbius transformation, with the weight \( \ell \) corresponding to the number operator. A Möbius transformation \( \varphi \) is classified by its eigenvalues \( \lambda, \lambda^{-1} \) into three types:
	\begin{itemize}
		\item If \(\lambda, \lambda^{-1} \in \T \setminus \{ -1, 1 \}\), \( \varphi \) is \emph{elliptic}, possessing two fixed points. An elliptic Möbius transformation \( \varphi \) is (smoothly) conjugate to a rotation
			\( \tau : z \mapsto e^{i \theta} z \),
			for which \( \| d \tau^*(f) \| = 1 \). By Remark \ref{remark:conjugacypointwiseorder}, \( (C^{\infty}(S^2), L^2(S^2, S), D) \) is pointwise bounded with respect to the \( \Z \)-action generated by \( \varphi \), placing it under the aegis of Theorem \ref{theorem:crossedproductspectraltriple}.
		\item If \(\lambda = \lambda^{-1} = \pm 1\), \( \varphi \) is
	either the identity or it is \emph{parabolic}, possessing one fixed point (and not diagonalisable as a matrix). A parabolic Möbius transformation \( \varphi \) is (smoothly) conjugate to a translation
			\( \tau : z \mapsto z+1 \).
			We compute that
			\[ \| d \tau^n \|_{\infty} = \sup_z \frac{1 + |z|^2}{1 + |z+n|^2} = \frac{1}{2} \left(n^2 + |n| \sqrt{n^2+4}+2\right) \in O(n^2) . \]
			Again, by Remark \ref{remark:conjugacypointwiseorder}, \( (C^{\infty}(S^2), L^2(S^2, S), D) \) has pointwise order 2 with respect to the \( \Z \)-action generated by \( \varphi \) and the number operator weight \( \ell_{\Z} \).
		\item Otherwise, if \(\lambda, \lambda^{-1} \in \C \setminus \T \), \( \varphi \) is \emph{loxodromic}, possessing two fixed points. A loxodromic Möbius transformation \( \varphi \) is (smoothly) conjugate to a dilation, perhaps combined with a rotation, \( \tau : z \mapsto \lambda^2 z \).
			In this case,
			\[ \| d \tau^n \|_{\infty} = \sup_z |\lambda|^{2n} \frac{1 + |z|^2}{1 + |\lambda|^{4n} |z|^2} = \max\{ |\lambda|^{2n}, |\lambda|^{-2n} \} \]
			which is not of polynomial order in \( n \).
	\end{itemize}
\end{example}

\begin{example}
	\label{example:toruslargediff}
	cf. \cite[\S 8.3.a]{Hasselblatt_2002}
	The group \( SL(d, \Z) \) acts on the torus \( \T^d \) by large diffeomorphisms. The action is realised by identifying \( \T^d \) with \( \R^d / \Z^d \) and \( SL(d, \Z) \) acting on \( \R^d \) in the usual way that a matrix acts on a vector. For ease of exposition, equip \( \T^d \) with a constant Riemannian metric \( \mathbf{g} \). For \( A \in SL(d, \Z) \) and the corresponding action \( \varphi_A \) on \( \T^d \), \( \| (d \varphi^n)_x \| = \| A^{-n} \|_{\mathbf{g}} \), which generically will be exponentially divergent. However, suppose that \( A \in SL(d, \Z) \) is a unipotent matrix, i.e.~such that \( (A - 1)^{s+1} = 0 \) for some \( s \in \N \). By Newton's binomial series, for \( n \in \Z\),
	\[ A^n = \sum_{k=0}^s \binom{n}{k} (A - 1)^k \]
	and
	\[ \| (d \varphi_{-n})_x \| = \| A^{-n} \|_{\mathbf{g}} \leq \sum_{k=0}^s \left| \binom{n}{k} \right| \| (A - 1)^k \|_{\mathbf{g}} \in O(n^s) . \]
	Hence
	\[ \left( C^{\infty}(\T^d) \rtimes_{\varphi^A} \Z, \ell^2(\Z) \tildeotimes L^2(\T^d, S), (N \tildeotimes 1, 1 \tildeotimes D) \right) \]
	is an ST\textsuperscript{2} with bounding matrix
	\[ \VEC{\epsilon}=\begin{pmatrix}0&0\\s&0\end{pmatrix} \qquad \vcenter{\hbox{\begin{tikzpicture}
	\graph[grow right sep=1.5cm, empty nodes, nodes={fill=black, circle, inner sep=1.5pt}, edges={semithick}]{
	a --["$s$" inner sep=5.5pt, middlearrow={<}] b;
	};
	\end{tikzpicture}}} . \]
\end{example}

This Example admits the following generalisation to outer automorphisms of noncommutative tori.

\begin{example}
	Let \( \Theta \) be a skew symmetric \( d \)-by-\( d \) real matrix. For \( x \in \Z^d \), define an operator \( l_{\Theta}(x) \) on \( \ell^2(\Z^d) \) by
	\[ (l_{\Theta}(x) \xi)(y) = e^{\pi i \langle \Theta x, -x + y \rangle} \xi(-x + y) . \]
	The algebra \( C^{\infty}(\T^d_{\Theta}) \) of smooth functions on the noncommutative torus \( \T^d_{\Theta} \) is the *-algebra spanned by \( l_{\Theta}(x) \) for all \( x \in \Z^d \). We call the $C^*$-algebra envelope \( C(\T^d_{\Theta}) \). When \( \Theta = 0 \), we recover \( C(\T^d) \). As in the classical case, integer matrices can act by automorphisms. Following \cite[\S 2.3]{Jeong_2015}, let \( A \in SL(d, \Z) \) be such that \( A^* \Theta A = \Theta \). Then \( \alpha_A: l_{\Theta}(x) \mapsto l_{\Theta}(A x) \) defines an automorphism of \( C(\T^d_{\Theta}) \). (For \( d = 2 \), the condition \( A^* \Theta A = \Theta \) is automatically satisfied.)
	
	Let \( (v_i)_{i=1}^d \) be a basis of \( \R^d \). To simplify notation, we will also write \( (v_i)_{i=1}^d \) for their images in \( \Cl_d \). Let \( S \) be a Clifford module for \( \Cl_d \) and define an unbounded operator
	\[ (D \xi)(y) = \sum_{i=1}^d \langle e_i, y \rangle v_i \xi(y) \qquad (y \in \Z^d) \]
	on \( \ell^2(\Z^d, S) \). We obtain a spectral triple \( ( C^{\infty}(\T^d_{\Theta}), \ell^2(\Z^d, S), D) \). We have
	\[ ([D, l_{\Theta}(x)] \xi)(y) = -e^{\pi i \langle \Theta x, -x + y \rangle} \sum_{i=1}^d \langle e_i, x \rangle v_i \xi(-x + y) \]
	so that
	\[ \left\| [D, l_{\Theta}(x)] \right\| = \Biggl\| \sum_{i, j=1}^d \langle e_i, x \rangle \langle e_j, x \rangle v_i v_j \Biggr\|^{\frac{1}{2}} = \Biggl\| \sum_{i, j=1}^d \langle x, e_i \rangle \langle v_i, v_j \rangle \langle e_j, x \rangle \Biggr\|^{\frac{1}{2}} = \| V x \| \]
	where \( V : \Z^d \to \R^d \) is the linear map taking \( e_i \mapsto v_i \).
	
	If \( A \in SL(d, \Z) \) (with \( A^* \Theta A = \Theta \)) is unipotent, so that \( (A - 1)^{s+1} = 0 \) for some \( s \in \N \), then \( \| A^n \| \in O(n^s) \) as in Example \ref{example:toruslargediff}, and
	\[ \left\| [D, \alpha_A^n(l_{\Theta}(x))] \right\| = \left\| [D, l_{\Theta}(A^n x)] \right\| = \| V A^n x \| \leq \| V \| \| A^n \| \| x \| \in O(n^s) . \]
	Hence
	\[ \left( C^{\infty}(\T^d_{\Theta}) \rtimes_{\alpha_A} \Z, \ell^2(\Z) \tildeotimes \ell^2(\Z^d, S), (N \tildeotimes 1, 1 \tildeotimes D) \right) \]
	is an ST\textsuperscript{2} with bounding matrix
	\[ \VEC{\epsilon}=\begin{pmatrix}0&0\\s&0\end{pmatrix} \qquad \vcenter{\hbox{\begin{tikzpicture}
	\graph[grow right sep=1.5cm, empty nodes, nodes={fill=black, circle, inner sep=1.5pt}, edges={semithick}]{
	a --["$s$" inner sep=5.5pt, middlearrow={<}] b;
	};
	\end{tikzpicture}}} . \]
\end{example}

\subsection{Nilpotent flows on homogeneous spaces}

Let \( G \) be a connected Lie group. Right-invariant Riemannian metrics on \( G \) are in bijection with inner products on \( \mathfrak{g} \). If \( \mathbf{g}_e \) denotes such an inner product on \( \mathfrak{g} = T_e G \), we define the Riemannian metric \( \mathfrak{g} \) at any other point \( g \in G \) by
\[ \mathbf{g}_g(u, v) = \mathbf{g}_e((d R_{g^{-1}})_g u, (d R_{g^{-1}})_g v) \]
where \( R_{g^{-1}} \) is the diffeomorphism of \( G \) given by right translation by \( g^{-1} \) and \( d R_{g^{-1}} \) is its pushforward. If the group is noncompact, the metric so obtained will not usually be left invariant \cite[\S 7]{Milnor_1976}. The norm of the Jacobian of left translation \( L_g \) by \( g \in G \), at \( h \in G \) is
\begin{align*}
	\| (d L_g)_h \|
	& = \sup_{u \in T_h G} \frac{\mathbf{g}_{g h}((d L_g)_h u, (d L_g)_h u)}{\mathbf{g}_h(u, u)} \\
	& = \sup_{u \in T_h G} \frac{\mathbf{g}_e((d R_{(g h)^{-1}})_{g h} (d L_g)_h u, (d R_{(g h)^{-1}})_{g h} (d L_g)_h u)}{\mathbf{g}_e((d R_{h^{-1}})_h u, (d R_{h^{-1}})_h u)} \\
	& = \sup_{v \in T_e G = \mathfrak{g}} \frac{\mathbf{g}_e((d\Ad_g)_e v, (d\Ad_g)_e v)}{\mathbf{g}_e(v, v)} \\
	& = \| (d\Ad_g)_e \| \\
	& = \| \Ad_g \|_{\mathbf{g}_e}
\end{align*}
where we have used the identity
\( (d R_{h^{-1}})_{g h} (d L_g)_h = (d L_g)_e (d R_{h^{-1}})_h \)
resulting from the facts that left and right actions commute and that the pushforward at \( e \in G \) of the adjoint action on \( G \) is the adjoint action on \( \mathfrak{g} \).

If \( H \) is any closed subgroup of \( G \) then \( G/H \) is a quotient manifold. A right-invariant Riemannian metric \( \mathbf{g} \) on \( G \) reduces to a Riemannian metric \( \mathbf{h} \) on \( G/H \). To construct \( \mathbf{h} \), let \( \pi : G \to G/H \) be the quotient map. Its pushforward at any point \( g \in G \),
\[ d\pi_g : T_g G \to T_{g H} (G/H) , \]
restricts to an isomorphism between \( T_g (g H)^{\perp} = (\ker d\pi_g)^{\perp}\) and \( T_{g H} (G/H) \). Define \( \mathbf{h} \) by
\[ \mathbf{h}_{g H}(u, v) = \mathbf{g}_g(d\pi_g|_{T_g (g H)^{\perp}}^{-1} u, d\pi_g|_{T_g (g H)^{\perp}}^{-1} v) . \]
There remains a left action of \( G \) on \( G/H \). As a crude estimate, we have
\[ \| (d L_g)_{h H} \| \leq \| (d L_g)_h  \| = \| \Ad_g \|_{\mathbf{g}_e} \]
for the Jacobian of left translation \( L_g \).

Recall the Campbell identity
\[ \Ad_{\exp X}(Y) = \exp(\ad_X)(Y) = \sum_{n=0}^{\infty} \frac{1}{n!} \ad_X^n(Y) . \]
An element \( X \in \mathfrak{g} \) is \emph{nilpotent} if \( \ad_X^{s+1} = 0 \) for some step size \( s \in \N \). In that case,
\[ \Ad_{\exp t X}(Y) = \sum_{n=0}^s \frac{t^n}{n!} \ad_X^n(Y) . \]
Consider the flow \( \phi^X \) given by \( \phi^X_t = L_{\exp t X} \) on \( G/H \). We have
\[ \| d \phi^X_t \|_{\infty} \leq \| \Ad_{\exp t X}(Y) \|_{\mathbf{g}_e} \in O(t^s) \]
so that
\[ \left( C_c^{\infty}(X) \rtimes_{\phi^X} \R, L^2(\R) \tildeotimes L^2(G/H, S), (M_{\ell_{\R}} \tildeotimes 1, 1 \tildeotimes D) \right) \]
is an ST\textsuperscript{2} with bounding matrix
\[ \VEC{\epsilon}=\begin{pmatrix}0&0\\s&0\end{pmatrix} \qquad \vcenter{\hbox{\begin{tikzpicture}
\graph[grow right sep=1.5cm, empty nodes, nodes={fill=black, circle, inner sep=1.5pt}, edges={semithick}]{
a --["$s$" inner sep=5.5pt, middlearrow={<}] b;
};
\end{tikzpicture}}} . \]
Such flows \( \phi^X \) constitute an important family of parabolic dynamical systems \cite[\S 8.3.b]{Hasselblatt_2002}.

\begin{example}
	cf. \cite[\S 8.3.3]{Hasselblatt_2002}
	Let \( \Gamma \subset SL(2, \R) \) be a cocompact lattice. A \emph{horocycle} flow \( \phi^X \) on \( SL(2, \R)/\Gamma \) is generated by a nilpotent element \( X \in \mathfrak{sl}(2, \R) \). Of necessity, \( X \) will be conjugate to
	\[ \begin{pmatrix} 0 & 1 \\ 0 & 0 \end{pmatrix} \in \mathfrak{sl}(2, \R) \]
	and so will be 2-step nilpotent.
\end{example}

\begin{example}
	cf. \cite[\S 8.3.2]{Hasselblatt_2002} \cite[\S 2.2]{Avila_2021}
	A compact \emph{nilmanifold} is a quotient \( G/\Gamma \) of a simply connected nilpotent Lie group \( G \) by a lattice \( \Gamma \subset G \). The \emph{nilflow} \( \phi^X \) generated by a vector field \( X \in \mathfrak{g} \) is the restriction of the left action of \( G \) to the one-parameter subgroup \( (\exp t X)_{t \in \R} \). Every element of a nilpotent Lie algebra is nilpotent, with step size less than or equal to the step size of the Lie algebra, so the above construction may be applied.
\end{example}

\begin{example}
	Let \( P \subseteq SO_0(n, 1) \) denote the standard parabolic subgroup. The homogeneous space \( SO_0(n, 1)/P \) is \( S^{n-1} \) and the Lorentz group \( SO_0(n, 1) \) acts by Möbius transformations on \( S^{n-1} \). We thereby recover Example \ref{example:mobius} as a special case.
\end{example}

\appendix

\section{A nearly convex set from relatively bounded commutators}
\label{section:nearly-convex}

A subset \( S \subseteq \R^n \) is \emph{nearly convex} if there exists a convex subset \( C \subseteq \R^n \) such that \( C \subseteq S \subseteq \overline{C} \) \cite[Definition 2.1]{Moffat_2016}. (Remark that \( \overline{S} = \overline{C} \) is convex.)

\begin{thm}
	\label{theorem:nearly-convex}
	Let \( A \) and \( B \) be densely defined operators on a Hilbert space \( \mathpzc{H} \) such that
	$A$ is self-adjoint, $B$ is positive and invertible, and $A$ and $B$ commute on a common core.
	Let \( T \in \mathbb{B}(\mathpzc{H}) \) and define the subset \( S \subset \R^2 \) consisting of \( (\alpha, \beta) \in [0, \infty) \times [0, \infty) \) such that \( T  \) preserves \( \Dom A |A|^{-1+\alpha} \) and
	\[ [A |A|^{-1+\alpha}, T] B^{-\beta} \]
	extends from \( \Dom A|A|^{-1+\alpha} \) to a bounded operator on \( \mathpzc{H} \).
	
	The subset \( S \) is nearly convex and contains \( \{ 0 \} \times [0, \infty) \).
	More precisely, if \( (\alpha_1, \beta_1) \in S \) then \( (\alpha_2, \beta_2) \in S \) for any \( 0 \leq \alpha_2 \leq \alpha_1 \) and \( \beta_2 > \frac{\alpha_2 \beta_1}{\alpha_1} \).
	If, further, \( (\alpha_3, \beta_3) \in S \) with \( \alpha_3 \geq \alpha_1 \), then \( (\alpha_2, \beta_2) \in S \) for any \( \alpha_1 \leq \alpha_2 \leq \alpha_3 \) and
	\[ \beta_2 > \frac{(\alpha_3 - \alpha_2) \beta_1 + (\alpha_2 - \alpha_1) \beta_3}{\alpha_3 - \alpha_1} . \]
\end{thm}

A similar result holds on Hilbert $C^*$-modules; see \cite[Appendix A.3.1]{AMthesis}. The following diagram shows the relationship between the parameters.
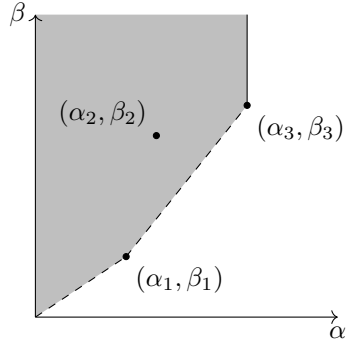
\[	\begin{tikzpicture}[scale=4]
		\coordinate (A) at (0.3,0.2);
		\coordinate (At) at (0.3,1);
		\coordinate (B) at (0.7,0.7);
		\coordinate (Bt) at (0.7,1);
		\fill[lightgray] (0, 1) -- (0, 0) -- (A) -- (B) --(Bt);
%		\fill[lightgray] (At) -- (A) -- (B) -- (Bt);
		\draw[->] (0,0) -- (1,0) node[below]{\(\alpha\)};
		\draw[->] (0,0) -- (0,1) node[left]{\(\beta\)};
%		\draw (A) -- (At);
		\draw (B) -- (Bt);
		\draw[dashed] (A) -- (B);
		\draw[dashed] (0, 0) -- (A);
		\filldraw[black] (A) circle (0.01) node[below right]{\((\alpha_1,\beta_1)\)};
		\filldraw[black] (B) circle (0.01) node[below right]{\((\alpha_3,\beta_3)\)};
		\filldraw[black] (0.4, 0.6) circle (0.01) node[above left]{\((\alpha_2,\beta_2)\)};
	\end{tikzpicture}
\]
For the proof, we require the following lemma.

\begin{lemma}
	\label{lemma:56une56ue56ub56une6}
	cf. \cite[Lemma 10.17]{varillybook}
	Let \( A \) be a self-adjoint operator on a Hilbert space \( \mathpzc{H} \). Let \( T \in \mathbb{B}(\mathpzc{H}) \) preserve \( \Dom A \) and have \( [A, T] \) extend to an bounded operator. Then, for any \( \alpha \in (0, 1) \) and \( y \in \R \), \( T \) preserves \( \Dom A|A|^{-1+\alpha} = \Dom A|A|^{-1+\alpha+y i} = \Dom |A|^{\alpha} \) and
	\[ [A|A|^{-1+\alpha+y i}, T] \]
	extends to a bounded operator and
	\[ \sup_{y \in \R} e^{-\tfrac{\pi}{2} |y|} \| [A|A|^{-1+\alpha+y i}, T] \| < \infty . \]
\end{lemma}

For the proof, we require the integral formula
\begin{equation}
(1 + D^2)^{-\alpha} = \frac{\sin(\alpha \pi)}{\pi} \int_0^{\infty} \lambda^{-\alpha} (\lambda + 1 + D^2)^{-1} d\lambda , \label{M} 
\end{equation}
norm-convergent for \( 0 < \Re(\alpha) < 1 \); for more details we refer to \cite[Lemma A.4]{Carey_1998}.

\begin{proof}
	Let \( \langle A\rangle = (1 + A^2)^{1/2} \) .
	First, note that
	\[ [\langle A\rangle^{\alpha+y i}, T] = - \langle A\rangle^{\alpha+y i} [\langle A\rangle^{-\alpha-y i}, T] \langle A\rangle^{\alpha+y i} . \]
	By the integral formula \eqref{M} and using \cite[Lemma 2.3]{Carey_1998}, on \( \Dom A \), 
	\begin{align*}
		& -\langle A\rangle^{\alpha} [\langle A\rangle^{-\alpha-y i}, T] \langle A\rangle^{\alpha} \\
		& \qquad = \frac{\sin \frac{(\alpha+y i) \pi}{2}}{\pi} \int_0^{\infty} \lambda^{-\frac{\alpha+y i}{2}} \langle A\rangle^{\alpha} [T, (\lambda + 1 + A^2)^{-1}] \langle A\rangle^{\alpha} d\lambda \\
		& \qquad = \frac{\sin \frac{(\alpha+y i) \pi}{2}}{\pi} \int_0^{\infty} \lambda^{-\frac{\alpha+y i}{2}} \langle A\rangle^{\alpha} \left( A (\lambda + 1 + A^2)^{-1} [A, T] (\lambda + 1 + A^2)^{-1} \right. \\
		& \qquad\qquad \left. + (\lambda + 1 + A^2)^{-1} [A, T] A (\lambda + 1 + A^2)^{-1} \right) \langle A\rangle^{\alpha} d\lambda .
	\end{align*}
	The integral is norm-convergent and we obtain a bound
	\begin{align*}
		& \big\| \langle A\rangle^{\alpha} [\langle A\rangle^{-\alpha-y i}, a] \langle A\rangle^{\alpha} \big\| \\
		& \qquad \leq \frac{| \sin \frac{(\alpha+y i) \pi}{2} |}{\pi} \int_0^{\infty} \lambda^{-\frac{\alpha}{2}} \left( \big\| A \langle A\rangle^{\alpha} (\lambda + 1 + A^2)^{-1} \big\| \| [A, T] \| \big\| \langle A\rangle^{\alpha} (\lambda + 1 + A^2)^{-1} \big\| \right. \\
		& \qquad\qquad \left. + \big\| \langle A\rangle^{\alpha}  (\lambda + 1 + A^2)^{-1} \big\| \| [A, T] \| \big\| A \langle A\rangle^{\alpha} (\lambda + 1 + A^2)^{-1} \big\| \right) d\lambda \\
		& \qquad \leq \frac{| \sin \frac{(\alpha+y i) \pi}{2} |}{\pi} 2 \| [A, T] \| \int_0^{\infty} \lambda^{-\frac{\alpha}{2}} (\lambda + 1)^{-\frac{3}{2} + \alpha} d\lambda \\
		& \qquad = \sqrt{\cosh (y \pi)-\cos (\alpha \pi)} \frac{2^{\alpha}}{\sqrt{2 \pi}} \frac{\Gamma(1-\alpha)}{\Gamma(\frac{3}{2}-\alpha)} \| [A, T] \| .
	\end{align*}
	Next, with \( F_A = A \langle A\rangle^{-1} \), a similar application of the integral formula \eqref{M}, as in \cite[Proposition 2.4]{Carey_1998}, gives
	\[ \big\| [F_A, T] \langle A\rangle^{\alpha} \big\| \leq C_{\alpha} \big\| [D, T] \langle D\rangle^{-\alpha} \big\| \]
	for some constant \( C_{\alpha} \).
	Write
	\[ [F_A \langle A\rangle^{\alpha+y i}, T] = [F_A, T] \langle A\rangle^{\alpha+y i} + F_A \langle A\rangle^{\alpha+y i} [\langle A\rangle^{-\alpha-y i}, T] \langle A\rangle^{\alpha+y i} \]
	so that
	\begin{align*}
		\big\| [F_A \langle A\rangle^{\alpha+y i}, T] \big\|
		& \leq \big\| [F_A, T] \langle A\rangle^{\alpha} \big\| + \big\| \langle A\rangle^{\alpha} [\langle A\rangle^{-\alpha-y i}, a] \langle A\rangle^{\alpha} \big\| \\
		& \leq C'_{\alpha} (1 + \big| \sin \tfrac{(\alpha+y i) \pi}{2} \big|) \big\| [A, T] \big\|
	\end{align*}
	for some constant \( C'_{\alpha} \). Hence \( [F_A \langle A\rangle^{\alpha+y i}, T] \) extends to an adjointable operator.
	
	Next, for \( x \in \R \),
	\begin{align*}
		\big| x |x|^{-1+\alpha +y i} - x \langle x\rangle^{-1+\alpha+y i} \big|
		& = |x| \big| |x|^{-1+\alpha+y i} - \langle x\rangle^{-1+\alpha+y i} \big| \\
		& \leq \left( |x|^{\alpha} \big| (|x| \langle x\rangle^{-1})^{y i} - 1 \big| + |x| \big| |x|^{-1+\alpha} - \langle x\rangle^{-1+\alpha} \big| \right) |\langle x\rangle^{y i}| \\
		& \leq |x|^{\alpha} \big| 1 - (|x| \langle x\rangle^{-1})^{y i} \big| + |x| \big| |x|^{-1+\alpha} - \langle x\rangle^{-1+\alpha} \big| .
	\end{align*}
	Now
	\begin{align*}
		\big| 1 - (|x| \langle x\rangle^{-1})^{y i} \big|
		& = \sqrt{(1 - \cos(y \log (|x| \langle x\rangle^{-1})) )^2 + \sin(y \log (|x| \langle x\rangle^{-1}))^2} \\
		& = \sqrt{2 - 2 \cos(y \log (|x| \langle x\rangle^{-1}))^2} \\
		& \leq |y| (\log \langle x\rangle - \log |x|)
	\end{align*}
	since \( |1 - \cos \theta| \leq \frac{1}{2} \theta^2 \). One can check that there exist \( c_1, c_2 > 0 \) such that \( |x|^{\alpha} (\log \langle x\rangle - \log |x|) \leq c_1 \alpha \) and \( |x| \big| |x|^{-1+\alpha} - \langle x\rangle^{-1+\alpha} \big| \leq c_2 \alpha \). Hence
	\[ \big| x |x|^{-1+\alpha +y i} - x \langle x\rangle^{-1+\alpha+y i} \big| \leq C'' \alpha (1 + |y|) \]
	for some \( C'' > 0 \) and therefore
	\[ \big| A |A|^{-1+\alpha +y i} - F_A \langle A\rangle^{\alpha+y i} \big| \leq C'' \alpha (1 + |y|) . \]
	Hence
	\begin{align*}
		\| [A|A|^{-1+\alpha+y i}, T] \| & \leq \big\| [F_A \langle A\rangle^{\alpha+y i}, T] \big\| + \big\| [A |A|^{-1+\alpha +y i} - F_A \langle A\rangle^{\alpha+y i}, T] \big\| \\
		& \leq C'_{\alpha} (1 + \big| \sin \tfrac{(\alpha+y i) \pi}{2} \big|) \big\| [D, S] \langle D\rangle^{-\alpha} \big\| + 2 C'' \alpha (1 + |y|) \| T \|
	\end{align*}
	and \( [A|A|^{-1+\alpha+y i}, T] \) extends to a bounded operator.
	We finally obtain that
	\[ \sup_{y \in \R} e^{-\tfrac{\pi}{2} |y|} \| [A|A|^{-1+\alpha+y i}, T] \| < \infty \]
	as required.
\end{proof}

\begin{proof}[Proof of Theorem \ref{theorem:nearly-convex}]
	Let \( (\alpha_1, \beta_1) \in S \). First, noting that \( \Dom A |A|^{-1 + \alpha} = \Dom |A|^{\alpha} \) for all \( \alpha \geq 0 \),
	by \cite[Theorem 12.5]{Krasnoselskii_1976}, \( T \) preserves \( \Dom A |A|^{-1+\alpha} \) for all \( \alpha \leq \alpha_1 \). Second,
	\[ [A |A|^{-1+\alpha_1}, T] B^{-\beta} \]
	is bounded for all \( \beta \geq \beta_1 \).
	Third, since $A$ and $B$ commute on a common core,
	\[ [A |A|^{-1+\alpha_1}, T] B^{-\beta_1} = [A |A|^{-1+\alpha_1}, T B^{-\beta_1}] \]
	extends to a bounded operator. Fix \( \alpha_1' \in (0, \alpha_1) \) so that by Lemma \ref{lemma:56une56ue56ub56une6}, the operator
	\[ [A |A|^{-1+\alpha_1' + y i}, T B^{-\beta_1}] = [A |A|^{-1+\alpha_1'+y i}, T] B^{-\beta_1} \]
	is bounded for any \( y \in \R \), with
	\begin{align*}
		M_1 & := \sup_{y \in \R} e^{(\alpha_1'+y i)^2} \| [A |A|^{-1+\alpha_1'+y i}, T] B^{-\beta_1} \| \\
		& \; = e^{\alpha_1'^2} \sup_{y \in \R} e^{-y^2} \| [A |A|^{-1+\alpha_1'+y i}, T] B^{-\beta_1} \| < \infty .
	\end{align*}
	Let \( (\alpha_3, \beta_3) \in S \) with \( \alpha_3 \geq \alpha_1 \) and fix \( \alpha_3' \in (0, \alpha_3) \) so that, similarly,
	\[ M_3 := \sup_{y \in \R} e^{(\alpha_3'+y i)^2} \| [A |A|^{-1+\alpha_3'+y i}, T] B^{-\beta_3} \| < \infty . \]
	Let \( \eta \in \mathpzc{H} \) and \( \xi \in \Dom(A) \). Define the holomorphic function
	\[ f : z \mapsto e^{(\alpha_1' + (\alpha_3' - \alpha_1') z)^2} \langle \eta \mid [A |A|^{-1 + \alpha_1' + (\alpha_3' - \alpha_1') z}, T] B^{-\beta_1 - (\beta_3 - \beta_1) z} \xi \rangle \]
	on the strip where \( 0 \leq \Re(z) \leq 1 \). We have
	\[ \| f(z) \| \leq \big| e^{(\alpha_1' + (\alpha_3' - \alpha_1') z)^2} \big| \| \eta \| \big\| \mid [A |A|^{-1 + \alpha_1' + (\alpha_3' - \alpha_1') z}, T] B^{-\beta_1 - (\beta_3 - \beta_1) z} \xi \big\| . \]
	For \( y \in \R \),
	\begin{align*}
		\| f(1 + y i) \|
		& \leq \big| e^{(\alpha_3' + (\alpha_3' - \alpha_1') y i)^2} \big| \| \eta \| \big\| [A |A|^{-1 + \alpha_3' + (\alpha_3' - \alpha_1') y i}, T] B^{-\beta_3} B^{(\beta_3 - \beta_1) y i} \xi \big\| \\
		& \leq M_3 \| \eta \| \| \xi \|
	\end{align*}
	and
	\begin{align*}
		\| f(y i) \|
		& \leq \big| e^{(\alpha_1' + (\alpha_3' - \alpha_1')^2} \big| \| \eta \| \big\| [A |A|^{-1 + \alpha_1' + (\alpha_3' - \alpha_1') y i}, T] B^{- \beta_1} B^{(\beta_1 - \beta_3) y i} \xi \big\| \\
		& \leq M_1 \| \eta \| \| \xi \| .
	\end{align*}
	By Hadamard's three-line theorem, we obtain for \( \alpha_1' \leq \alpha_2 \leq \alpha_3' \) that
	\[ \big| e^{\alpha_2^2} \langle \eta \mid [A |A|^{-1 + \alpha_2}, T] B^{-\gamma} \xi \rangle \big| = | f(\tfrac{\alpha_2 - \alpha_1'}{\alpha_3' - \alpha_1'}) | \leq M_1^{\frac{\alpha_3' - \alpha_2}{\alpha_3' - \alpha_1'}} M_3^{\frac{\alpha_2 - \alpha_1'}{\alpha_3' - \alpha_1'}} \| \eta \| \| \xi \| \]
	where \( \gamma = \frac{(\alpha_3' - \alpha_2) \beta_1 + (\alpha_2 - \alpha_1') \beta_3}{\alpha_3' - \alpha_1'} \).
	Hence, putting \( \eta = [A |A|^{-1 + \alpha_2}, T] B^{-\gamma} \xi \),
	\[ \| \eta \|^2 \leq e^{-\alpha_2^2} M_1^{\frac{\alpha_3' - \alpha_2}{\alpha_3' - \alpha_1'}} M_3^{\frac{\alpha_2 - \alpha_1'}{\alpha_3' - \alpha_1'}} \| \eta \| \| \xi \| \]
	and so
	\[ \| \eta \| \leq e^{-\alpha_2^2} M_1^{\frac{\alpha_3' - \alpha_2}{\alpha_3' - \alpha_1'}} M_3^{\frac{\alpha_2 - \alpha_1'}{\alpha_3' - \alpha_1'}} \| \xi \| . \]
	By the density of \( \Dom(A) \) in \( \mathpzc{H} \),
	\[ [A |A|^{-1 + \alpha_2}, T] B^{-\gamma} \leq e^{-\alpha_2^2} M_1^{\frac{\alpha_3' - \alpha_2}{\alpha_3' - \alpha_1'}} M_3^{\frac{\alpha_2 - \alpha_1'}{\alpha_3' - \alpha_1'}} . \]
	By making suitable choices of \( \alpha_1' \) and \( \alpha_3' \), we obtain that
	\[ [A |A|^{-1+\alpha_2}, T] B^{-\beta_2} \]
	is bounded for any \( \alpha_1 \leq \alpha_2 \leq \alpha_3 \) and \( \beta_2 > \frac{(\alpha_3 - \alpha_2) \beta_1 + (\alpha_2 - \alpha_1) \beta_3}{\alpha_3 - \alpha_1} \).
	
	It is immediate that \( \{ 0 \} \times [0, \infty) \subseteq S \). The statement for \( 0 \leq \alpha_2 \leq \alpha_1 \) follows by substituting \( (0, 0) \) for \( (\alpha_1, \beta_1) \) and \( (\alpha_1, \beta_1) \) for \( (\alpha_3, \beta_3) \).
	
	Define the subset \( C \) of \( (\alpha, \beta) \in S \) such that, for all \( y \in \R \),
	\[ [A |A|^{-1+\alpha + i y}, T] B^{-\beta} = [A |A|^{-1+\alpha + i y}, T B^{-\beta}] \]
	extends from \( \Dom |A|^{\alpha} \) to a bounded operator on \( \mathpzc{H} \) and
	\[ \sup_{y \in \R} e^{-y^2} \| [A|A|^{-1+\alpha+y i}, T] B^{-\beta} \| < \infty . \]
	Then we have seen that \( C \) is convex and Lemma \ref{lemma:56une56ue56ub56une6} says that \( S \subseteq \overline{C} \).
\end{proof}

\providecommand{\bysame}{\leavevmode\hbox to3em{\hrulefill}\thinspace}
\providecommand{\href}[2]{#2}


\begin{thebibliography}{10}

\bibitem{AMY}
Iakovos Androulidakis, Omar Mohsen, and Robert Yuncken, \emph{A pseudodifferential calculus for maximally hypoelliptic operators and the {H}elffer-{N}ourrigat conjecture}, \href{https://arxiv.org/abs/2201.12060}{{\ttfamily arXiv:2201.12060 [math.AP]}}, 2022.

\bibitem{Avila_2021}
Artur Avila, Giovanni Forni, Davide Ravotti, and Corinna Ulcigrai, \emph{Mixing for smooth time-changes of general nilflows}, Adv. Math. \textbf{385} (2021), 107759.

\bibitem{baajjulg}
Saad Baaj and Pierre Julg, \emph{Th\'eorie bivariante de {K}asparov et op\'erateurs non born\'es dans les {$C^*$}-modules hilbertiens}, C. R. Acad. Sci. Paris S\'er. I Math. \textbf{296} (1983), no.~21, 875--878.

\bibitem{baderetal}
Uri Bader, Alex Furman, Tsachik Gelander, and Nicolas Monod, \emph{Property ({T}) and rigidity for actions on {B}anach spaces}, Acta Math. \textbf{198} (2007), no.~1, 57--105.

\bibitem{Bellissard_2010}
Jean~V. Bellissard, Matilde Marcolli, and Kamran Reihani, \emph{Dynamical systems on spectral metric spaces}, \href{https://arxiv.org/abs/1008.4617}{{\ttfamily arXiv:1008.4617 [math.OA]}}, 2010.

\bibitem{berghloef}
J\"oran Bergh and J\"orgen L\"ofstr\"om, \emph{Interpolation spaces: {A}n introduction}, Grundlehren der Mathematischen Wissenschaften, vol. No. 223, Springer-Verlag, Berlin-New York, 1976.

\bibitem{boutededkn}
Louis Boutet~de Monvel, \emph{Residual trace of {T}oeplitz or pseudodifferential projectors}, Lett. Math. Phys. \textbf{106} (2016), no.~12, 1633--1638.

\bibitem{leschbruening}
Jochen Br\"uning and Matthias Lesch, \emph{Hilbert complexes}, J. Funct. Anal. \textbf{108} (1992), no.~1, 88--132.

\bibitem{CaGaReSu}
Alan~L. Carey, Victor Gayral, Adam Rennie, and Fedor~A. Sukochev, \emph{Index theory for locally compact noncommutative geometries}, vol. 231, Mem. Amer. Math. Soc., no. 1085, American Mathematical Society, 2014.

\bibitem{Carey_1998}
Alan Carey and John Phillips, \emph{Unbounded Fredholm modules and spectral flow}, Canadian Journal of Mathematics \textbf{50} (1998), no.~4, 673--718.

\bibitem{ConnesNCDG}
Alain Connes, \emph{Non-commutative differential geometry}, Publications Math{\'e}matiques de l'IH{\'E}S \textbf{62} (1985), no.~1, 41--144.

\bibitem{connesmetric}
\bysame, \emph{Compact metric spaces, {F}redholm modules, and hyperfiniteness}, Ergodic Theory Dynam. Systems \textbf{9} (1989), no.~2, 207--220.

\bibitem{ConnesBigRed}
\bysame, \emph{Noncommutative geometry}, Academic Press, Inc., San Diego, CA, 1994.

\bibitem{connesgravity}
\bysame, \emph{Gravity coupled with matter and the foundation of non-commutative geometry}, Comm. Math. Phys. \textbf{182} (1996), no.~1, 155--176.

\bibitem{connessuq2}
\bysame, \emph{Cyclic cohomology, quantum group symmetries and the local index formula for {${\rm SU}_q(2)$}}, J. Inst. Math. Jussieu \textbf{3} (2004), no.~1, 17--68.

\bibitem{connesreconstr}
\bysame, \emph{On the spectral characterization of manifolds}, J. Noncommut. Geom. \textbf{7} (2013), no.~1, 1--82.

\bibitem{conneslott}
Alain Connes and John Lott, \emph{The metric aspect of noncommutative geometry}, NATO Adv. Sci. Inst. Ser. B: Phys., vol. 295, Plenum, New York, 1992, pp.~53--93.

\bibitem{Connes_1995}
Alain Connes and Henri Moscovici, \emph{The local index formula in noncommutative geometry}, Geom. Funct. Anal. \textbf{5} (1995), no.~2, 174--243.

\bibitem{connesskandalis}
Alain Connes and Georges Skandalis, \emph{The longitudinal index theorem for foliations}, Publ. Res. Inst. Math. Sci. \textbf{20} (1984), no.~6, 1139--1183.

\bibitem{Cornelissen_2008}
Gunther Cornelissen, Matilde Marcolli, Kamran Reihani, and Alina Vdovina, \emph{Noncommutative geometry on trees and buildings}, Traces in number theory, geometry and quantum fields, Aspects Math., vol. E38, Friedr. Vieweg, Wiesbaden, 2008, pp.~73--98.

\bibitem{Corwin_1990}
Lawrence~J. Corwin and Frederick~P. Greenleaf, \emph{Representations of nilpotent {L}ie groups and their applications, {P}art {I}: Basic theory and examples}, Cambridge Stud. Adv. Math., vol.~18, Cambridge Uni. Press, Cambridge, 1990.

\bibitem{capslovak}
Andreas \v{C}ap and Jan Slov\'ak, \emph{Parabolic geometries {I}: {B}ackground and general theory}, Math. Surveys Monogr., vol. 154, American Mathematical Society, Providence, RI, 2009.

\bibitem{capslovaksoucek}
Andreas \v{C}ap, Jan Slov\'ak, and Vladim\'ir Sou\v{c}ek, \emph{Bernstein-{G}elfand-{G}elfand sequences}, Ann. of Math. (2) \textbf{154} (2001), no.~1, 97--113.

\bibitem{davehallerheat}
Shantanu Dave and Stefan Haller, \emph{The heat asymptotics on filtered manifolds}, J. Geom. Anal. \textbf{30} (2020), no.~1, 337--389.

\bibitem{davehaller}
\bysame, \emph{Graded hypoellipticity of {BGG} sequences}, Ann. Global Anal. Geom. \textbf{62} (2022), no.~4, 721--789.

\bibitem{fefferpara}
Charles Fefferman, \emph{Parabolic invariant theory in complex analysis}, Adv. in Math. \textbf{31} (1979), no.~2, 131--262.

\bibitem{Forsyth_2014}
Iain Forsyth, Bram Mesland, and Adam Rennie, \emph{Dense domains, symmetric operators and spectral triples}, New York J. Math. \textbf{20} (2014), 1001--1020.

\bibitem{frieskhom}
Magnus Fries, \emph{Relative {$K$}-homology of higher-order differential operators}, J. Funct. Anal. \textbf{288} (2025), no.~1, 110678.

\bibitem{Fraczek_2004}
Krzysztof Fr\k{a}czek, \emph{On diffeomorphisms with polynomial growth of the derivative on surfaces}, Colloq. Math. \textbf{99} (2004), no.~1, 75--90.

\bibitem{GGJNCG}
Heiko Gimperlein and Magnus Goffeng, \emph{Commutator estimates on contact manifolds and applications}, J. Noncommut. Geom. \textbf{13} (2019), no.~1, 363--406.

\bibitem{goffeng24}
Magnus Goffeng, \emph{Solving the index problem for (curved) {B}ernstein-{G}elfand-{G}elfand sequences}, \href{https://arxiv.org/abs/2406.07033}{{\ttfamily arXiv:2406.07033 [math.KT]}}, 2024.

\bibitem{goffhelff}
Magnus Goffeng and Bernard Helffer, \emph{The index of sub-laplacians: beyond contact manifolds}, \href{https://arxiv.org/abs/2408.00091}{{\ttfamily arXiv:2408.00091 [math.AP]}}, 2025.

\bibitem{goffengkuzmin}
Magnus Goffeng and Alexey Kuzmin, \emph{Index theory of hypoelliptic operators on {C}arnot manifolds}, \href{https://arxiv.org/abs/2203.04717}{{\ttfamily arXiv:2203.04717 [math.DG]}}, 2024.

\bibitem{GMCK}
Magnus Goffeng and Bram Mesland, \emph{Spectral triples and finite summability on {C}untz-{K}rieger algebras}, Doc. Math. \textbf{20} (2015), 89--170.

\bibitem{GMRtwist}
Magnus Goffeng, Bram Mesland, and Adam Rennie, \emph{Untwisting twisted spectral triples}, Internat. J. Math. \textbf{30} (2019), no.~14, 1950076.

\bibitem{GRU}
Magnus Goffeng, Adam Rennie, and Alexandr Usachev, \emph{Constructing {KMS} states from infinite-dimensional spectral triples}, J. Geom. Phys. \textbf{143} (2019), 107--149.

\bibitem{varillybook}
Jos\'e{}~M. Gracia-Bond\'ia, Joseph~C. V\'arilly, and H\'ector Figueroa, \emph{Elements of noncommutative geometry}, Birkh\"auser Advanced Texts: Basler Lehrb\"ucher, Birkh\"auser Boston, Inc., Boston, MA, 2001.

\bibitem{grenshomo}
Martin Grensing, \emph{Universal cycles and homological invariants of locally convex algebras}, J. Funct. Anal. \textbf{263} (2012), no.~8, 2170--2204.

\bibitem{hallergenfive}
Stefan Haller, \emph{Analytic torsion of generic rank two distributions in dimension five}, J. Geom. Anal. \textbf{32} (2022), no.~10, 248.

\bibitem{Hasselblatt_2002}
Boris Hasselblatt and Anatole Katok, \emph{Principal structures}, Handbook of dynamical systems, {V}ol.\ 1{A}, North-Holland, Amsterdam, 2002, pp.~1--203.

\bibitem{Hasselmann_2014}
Stefan Hasselmann, \emph{Spectral triples on {C}arnot manifolds}, Ph.d. thesis, Fakult{\"a}t f{\"u}r Mathematik und Physik der Gottfried Wilhelm Leibniz Universit{\"a}t Hannover, 2014.

\bibitem{Hawkins_2013}
Andrew Hawkins, Adam Skalski, Stuart White, and Joachim Zacharias, \emph{On spectral triples on crossed products arising from equicontinuous actions}, Math. Scand. \textbf{113} (2013), no.~2, 262--291.

\bibitem{heckkolb}
Istv\'an Heckenberger and Stefan Kolb, \emph{Differential forms via the {B}ernstein-{G}elfand-{G}elfand resolution for quantized irreducible flag manifolds}, J. Geom. Phys. \textbf{57} (2007), no.~11, 2316--2344.

\bibitem{hirachiinv}
Kengo Hirachi, \emph{Logarithmic singularity of the Szeg\"o kernel and a global invariant of strictly pseudoconvex domains}, Ann. of Math. 63 (2006), 499--515.

\bibitem{Jeong_2015}
Ja~A. Jeong and Jae~Hyup Lee, \emph{Finite groups acting on higher dimensional noncommutative tori}, J. Funct. Anal. \textbf{268} (2015), no.~2, 473--499.

\bibitem{joswigess}
Michael Joswig, \emph{Essentials of tropical combinatorics}, Graduate Studies in Mathematics, vol. 219, American Mathematical Society, Providence, RI, 2021.

\bibitem{Julg_1995a}
Pierre Julg, \emph{Complexe de {R}umin, suite spectrale de {F}orman et cohomologie {$L^2$} des espaces sym\'etriques de rang {$1$}}, C. R. Acad. Sci. Paris S\'er. I Math. \textbf{320} (1995), no.~4, 451--456.

\bibitem{julghow}
\bysame, \emph{How to prove the {B}aum-{C}onnes conjecture for the groups {$Sp(n,1)$}?}, J. Geom. Phys. \textbf{141} (2019), 105--119.

\bibitem{julgkasparov}
Pierre Julg and Gennadi Kasparov, \emph{Operator {$K$}-theory for the group {${\rm SU}(n,1)$}}, J. Reine Angew. Math. \textbf{463} (1995), 99--152.

\bibitem{Kaad_2020}
Jens Kaad and David Kyed, \emph{Dynamics of compact quantum metric spaces}, Ergodic Theory Dynam. Systems \textbf{41} (2021), no.~7, 2069--2109.

\bibitem{kaadkyed}
\bysame, \emph{The quantum metric structure of quantum {$\rm SU(2)$}}, Memoirs of the European Mathematical Society, vol.~18, EMS Press, Berlin, 2025.

\bibitem{kaadlesch}
Jens Kaad and Matthias Lesch, \emph{A local global principle for regular operators in {H}ilbert {$C^*$}-modules}, J. Funct. Anal. \textbf{262} (2012), no.~10, 4540--4569.

\bibitem{leschkaad1}
\bysame, \emph{Spectral flow and the unbounded {K}asparov product}, Adv. Math. \textbf{248} (2013), 495--530.

\bibitem{kaadnestwolf}
Jens Kaad, Ryszard Nest, and Jesse Wolfson, \emph{A higher {K}ac-{M}oody extension for two-dimensional gauge groups}, \href{https://arxiv.org/abs/2110.01649}{{\ttfamily arXiv:2110.01649 [math.KT]}}, 2021.

\bibitem{Kasparov_1988}
G.~G. Kasparov, \emph{Equivariant {$KK$}-theory and the {N}ovikov conjecture}, Invent. Math. \textbf{91} (1988), no.~1, 147--201.

\bibitem{kordy}
Yuri~A. Kordyukov, \emph{Noncommutative geometry of foliations}, J. K-Theory \textbf{2} (2008), no.~2, 219--327.

\bibitem{Hoch3}
Ulrich Kr\"ahmer, Adam Rennie, and Roger Senior, \emph{A residue formula for the fundamental {H}ochschild 3-cocycle for {${\rm SU}_q(2)$}}, J. Lie Theory \textbf{22} (2012), no.~2, 557--585.

\bibitem{Krasnoselskii_1976}
M.A. Krasnoselskii, P.P. Zabreiko, E.I. Pustylnik, and P.E. Sobolevskii, \emph{Integral operators in spaces of summable functions}, Noordhoff International Publishing, Leyden, Netherlands, 1976.

\bibitem{danthedan}
Dan Kucerovsky, \emph{The {$KK$}-product of unbounded modules}, $K$-Theory \textbf{11} (1997), no.~1, 17--34.

\bibitem{Le_Donne_2017}
Enrico Le~Donne, \emph{A primer on {C}arnot groups: homogenous groups, {C}arnot-{C}arath\'eodory spaces, and regularity of their isometries}, Anal. Geom. Metr. Spaces \textbf{5} (2017), no.~1, 116--137.

\bibitem{bramlesch}
Matthias Lesch and Bram Mesland, \emph{Sums of regular self-adjoint operators in {H}ilbert-{$C^*$}-modules}, J. Math. Anal. Appl. \textbf{472} (2019), no.~1, 947--980.

\bibitem{LoReVa}
Steven Lord, Adam Rennie, and Joseph~C. V\'arilly, \emph{Riemannian manifolds in noncommutative geometry}, J. Geom. Phys. \textbf{62} (2012), no.~7, 1611--1638.

\bibitem{lottlimit}
John Lott, \emph{Limit sets as examples in noncommutative geometry}, $K$-Theory \textbf{34} (2005), no.~4, 283--326.

\bibitem{AMthesis}
Ada Masters, \emph{Groups and their actions in unbounded Kasparov theory}, Ph.D. thesis, School of Mathematics and Applied Statistics, University of Wollongong, 2025.

\bibitem{AMsomewhere}
Ada Masters and Adam Rennie, \emph{Conformal transformations and equivariance in unbounded {$KK$}-theory}, \href{https://arxiv.org/abs/2412.17220}{{\ttfamily arXiv:2412.17220 [math.OA]}}, 2025.

\bibitem{meslandbeast}
Bram Mesland, \emph{Unbounded bivariant {$K$}-theory and correspondences in noncommutative geometry}, J. Reine Angew. Math. \textbf{691} (2014), 101--172.

\bibitem{Milnor_1976}
John Milnor, \emph{Curvatures of left invariant metrics on {L}ie groups}, Advances in Math. \textbf{21} (1976), no.~3, 293--329.

\bibitem{Moffat_2016}
Sarah~M. Moffat, Walaa~M. Moursi, and Xianfu Wang, \emph{Nearly convex sets: Fine properties and domains or ranges of subdifferentials of convex functions}, Mathematical Programming \textbf{160} (2016), no.~1--2, 193--223.

\bibitem{neshtuset}
Sergey Neshveyev and Lars Tuset, \emph{The {D}irac operator on compact quantum groups}, J. Reine Angew. Math. \textbf{641} (2010), 1--20.

\bibitem{Paterson_2014}
Alan L.~T. Paterson, \emph{Contractive spectral triples for crossed products}, Math. Scand. \textbf{114} (2014), no.~2, 275--298.

\bibitem{pongeconfinv}
Rapha\"el Ponge, \emph{A microlocal approach to {F}efferman's program in conformal and {CR} geometry}, Motives, quantum field theory, and pseudodifferential operators, Clay Math. Proc., vol.~12, Amer. Math. Soc., Providence, RI, 2010, pp.~221--244.

\bibitem{pongemore} \bysame, \emph{Noncommutative residue invariants for CR and contact manifolds}, J. Reine Angew. Math. (Crelle’s Journal) 614 (2008), 117–151.

\bibitem{pongefurther} \bysame, \emph{The Logarithmic singularities of the Green functions of the conformal powers of the Laplacian}, Contemp. Math. Vol. 630 (2014), pp. 247–273.

\bibitem{pongewang2}
Rapha\"el Ponge and Hang Wang, \emph{Noncommutative geometry and conformal geometry, {II}. {C}onnes-{C}hern character and the local equivariant index theorem}, J. Noncommut. Geom. \textbf{10} (2016), no.~1, 307--378.

\bibitem{pongewang1}
\bysame, \emph{Noncommutative geometry and conformal geometry. {I}. {L}ocal index formula and conformal invariants}, J. Noncommut. Geom. \textbf{12} (2018), no.~4, 1573--1639.

\bibitem{puschhigh}
Michael Puschnigg, \emph{Finitely summable {F}redholm modules over higher rank groups and lattices}, J. K-Theory \textbf{8} (2011), no.~2, 223--239.

\bibitem{Raeburn_1988}
Iain Raeburn, \emph{On crossed products and {T}akai duality}, Proc. Edinburgh Math. Soc. (2) \textbf{31} (1988), no.~2, 321--330.

\bibitem{Raghunathan_1972}
Madabusi~S. Raghunathan, \emph{Discrete subgroups of {L}ie groups}, Ergebnisse der Mathematik und ihrer Grenzgebiete, vol.~68, Springer-Verlag, New York-Heidelberg, 1972.

\bibitem{Reed_1978}
Michael Reed and Barry Simon, \emph{Methods of modern mathematical physics {IV}: {A}nalysis of operators}, Academic Press, New York-London, 1978.

\bibitem{renniereconst}
Adam Rennie and Joseph~C. Varilly, \emph{Reconstruction of manifolds in noncommutative geometry}, \href{https://arxiv.org/abs/math/0610418}{{\ttfamily arXiv:math/0610418 [math.OA]}}, 2008.

\bibitem{rieffelquant}
Marc~A. Rieffel, \emph{Compact quantum metric spaces}, Operator algebras, quantization, and noncommutative geometry, Contemp. Math., vol. 365, Amer. Math. Soc., Providence, RI, 2004, pp.~315--330.

\bibitem{rubinetal} Alessandro Rubin, and Ludwik Dąbrowski, \emph{Real spectral triples on crossed products}, Rev. Math. Phys. 34 (2022)

\bibitem{rumincomp}
Michel Rumin, \emph{Formes diff\'erentielles sur les vari\'et\'es de contact}, J. Differential Geom. \textbf{39} (1994), no.~2, 281--330.

\bibitem{Rumin_2000}
\bysame, \emph{Sub-{R}iemannian limit of the differential form spectrum of contact manifolds}, Geom. Funct. Anal. \textbf{10} (2000), no.~2, 407--452.

\bibitem{Rumin_2005}
\bysame, \emph{An introduction to spectral and differential geometry in {C}arnot-{C}arath\'eodory spaces}, Rend. Circ. Mat. Palermo (2) Suppl. (2005), 139--196.

\bibitem{ruminsesh}
Michel Rumin and Neil Seshadri, \emph{Analytic torsions on contact manifolds}, Ann. Inst. Fourier (Grenoble) \textbf{62} (2012), no.~2, 727--782.

\bibitem{schoen95}
Richard Schoen, \emph{On the conformal and {CR} automorphism groups}, Geom. Funct. Anal. \textbf{5} (1995), no.~2, 464--481.

\bibitem{sww}
Elmar Schrohe, Markus Walze, and Jan-Martin Warzecha, \emph{Construction de triplets spectraux \`a\ partir de modules de {F}redholm}, C. R. Acad. Sci. Paris S\'er. I Math. \textbf{326} (1998), no.~10, 1195--1199.

\bibitem{Valette_2002}
Alain Valette, \emph{Introduction to the {B}aum-{C}onnes conjecture}, Lectures in Mathematics ETH Z\"urich, Birkh\"auser Verlag, Basel, 2002.

\bibitem{vanerpyuncktang}
Erik van Erp and Robert Yuncken, \emph{On the tangent groupoid of a filtered manifold}, Bull. Lond. Math. Soc. \textbf{49} (2017), no.~6, 1000--1012.

\bibitem{vanerpyunck}
\bysame, \emph{A groupoid approach to pseudodifferential calculi}, J. Reine Angew. Math. \textbf{756} (2019), 151--182.

\bibitem{waltbook}
Walter~D. van Suijlekom, \emph{Noncommutative geometry and particle physics}, Mathematical Physics Studies, Springer, Dordrecht, 2015.

\bibitem{voigtyuncken}
Christian Voigt and Robert Yuncken, \emph{Equivariant {F}redholm modules for the full quantum flag manifold of {${\rm SU}_q(3)$}}, Doc. Math. \textbf{20} (2015), 433--490.

\bibitem{wagnerobua}
Elmar Wagner, Fredy D\'iaz~Garc\'ia, and R\'eamonn O'Buachalla, \emph{A {D}olbeault-{D}irac spectral triple for the {$B_2$}-irreducible quantum flag manifold}, Comm. Math. Phys. \textbf{395} (2022), no.~1, 365--403.

\bibitem{wahlhomo}
Charlotte Wahl, \emph{Homological index formulas for elliptic operators over {$C^*$}-algebras}, New York J. Math. \textbf{15} (2009), 319--351.

\bibitem{yunckensl3}
Robert Yuncken, \emph{The {B}ernstein-{G}elfand-{G}elfand complex and {K}asparov theory for {${\rm SL}(3,\mathbb{C})$}}, Adv. Math. \textbf{226} (2011), no.~2, 1474--1512.

\end{thebibliography}
\end{document}